\documentclass[11pt,a4paper,reqno]{amsart}
\usepackage{amsmath,amssymb,graphics,epsfig,color,enumerate}
\usepackage[dvipsnames]{xcolor}
\usepackage{dsfont}  
\usepackage{verbatim}
\definecolor{refkey}{gray}{.4}
\definecolor{labelkey}{gray}{.4}

\newtheorem{Theorem}{Theorem}[section]

\newtheorem{Lemma}[Theorem]{Lemma}

\newtheorem{Proposition}[Theorem]{Proposition}
\newtheorem{Corollary}[Theorem]{Corollary}
\newtheorem{Remark}[Theorem]{Remark}

\newtheorem{Definition}[Theorem]{Definition}

\newtheorem{Assumption}[Theorem]{Assumption}

 \definecolor{darkgreen}{rgb}{0,0.4,0}

\definecolor{light}{gray}{0.9}

\newcommand{\cA}{\ensuremath{\mathcal A}}
\newcommand{\cB}{\ensuremath{\mathcal B}}
\newcommand{\cC}{\ensuremath{\mathcal C}}

\newcommand{\cE}{\ensuremath{\mathcal E}}
\newcommand{\cF}{\ensuremath{\mathcal F}}

\newcommand{\cL}{\ensuremath{\mathcal L}}

\newcommand{\cN}{\ensuremath{\mathcal N}}

\newcommand{\cP}{\ensuremath{\mathcal P}}

\newcommand{\cR}{\ensuremath{\mathcal R}}

\newcommand{\cX}{\ensuremath{\mathcal X}}

\newcommand{\cZ}{\ensuremath{\mathcal Z}}

\newcommand{\bbC}{{\ensuremath{\mathbb C}} }

\newcommand{\bbE}{{\ensuremath{\mathbb E}} }

\newcommand{\bbI}{{\ensuremath{\mathbb I}} }

\newcommand{\bbL}{{\ensuremath{\mathbb L}} }
\newcommand{\bbM}{{\ensuremath{\mathbb M}} }
\newcommand{\bbN}{{\ensuremath{\mathbb N}} }

\newcommand{\bbP}{{\ensuremath{\mathbb P}} }

\newcommand{\bbR}{{\ensuremath{\mathbb R}} }

\newcommand{\bbT}{{\ensuremath{\mathbb T}} }

\newcommand{\bbY}{{\ensuremath{\mathbb Y}} }
\newcommand{\bbZ}{{\ensuremath{\mathbb Z}} }

\let\a=\alpha \let\b=\beta   \let\d=\delta  \let\e=\varepsilon
 \let\g=\gamma     \let\k=\kappa  \let\l=\lambda
      \let\o=\omega    \let\p=\pi  
  \let\s=\sigma \let\t=\tau   
  
\let\D=\Delta      
\let\O=\Omega

\newcommand{\be}{\begin{equation}}
\newcommand{\en}{\end{equation}}
\newcommand{\bes}{\begin{equation*}}
\newcommand{\ens}{\end{equation*}}

\newcommand{\ra}{\rangle}
\newcommand{\la}{\langle}

\newcommand{\rosso}{\textcolor{black}}

\author[A.\ Faggionato]{Alessandra Faggionato}
\address{Alessandra Faggionato.
  Dipartimento di Matematica, Universit\`a di Roma `La Sapienza'
  P.le Aldo Moro 2, 00185 Roma, Italy}
\email{faggiona@mat.uniroma1.it}
\author[V.\ Silvestri]{Vittoria Silvestri}
\address{Vittoria Silvestri.
  Dipartimento di Matematica, Universit\`a di Roma `La Sapienza'
  P.le Aldo Moro 2, 00185 Roma, Italy}
\email{silvestri@mat.uniroma1.it}

\title[Time-dependent and time-periodic linear response]{A martingale approach to time-dependent and time-periodic linear response in Markov jump processes}

\begin{document}

\maketitle

  
  \begin{abstract}
 We consider a Markov jump process on a general  state space to which we apply
 a time-dependent weak perturbation over a finite time interval.   By martingale-based stochastic calculus, under a  suitable exponential moment bound for the perturbation we show that    the perturbed process does not explode almost surely and  we study the linear response  (LR) of  observables and additive functionals. When the unperturbed process is stationary, the above LR  formulas become computable in terms of  the steady state   two-time correlation function and of the stationary distribution. Applications are discussed for birth and death processes, random walks in a  confining potential, random walks in a random  conductance field.
 We then move to a Markov jump process on a  finite state space and investigate the LR  of  observables and additive functionals  in the oscillatory steady state (hence, over an infinite time horizon), when the perturbation is time-periodic. As an application we provide a formula for  the complex mobility matrix of  a random walk on a discrete $d$-dimensional torus, with possibly heterogeneous jump rates.

 \medskip
 

\medskip

   \end{abstract}
  
  

  \section{Introduction}
 Markov jump processes in continuous time and  with  general state space form  a   fundamental  class of stochastic processes. They are often called  Markov chains  when the state space is discrete and countable (finite or infinite).
If the state space is infinite, the phenomenon of explosion can take place and it consists of the accumulation of infinitely many jumps in finite time.  We consider here an unperturbed system modelled by a general  Markov jump process  with time-homogeneous transition kernel, assuming that a.s.\ explosion does not take place.

We study the linear response of the system in two regimes. In the first regime we take a time-dependent weak perturbation  and a fixed  initial distribution, i.e.~not depending on the perturbation. In the second regime, restricting to  finite state spaces, we consider a time-periodic weak perturbation and take as initial distribution the one producing the oscillatory steady state in the perturbed dynamics. In both regimes  we focus on the linear response of the expected value of observables at a fixed time and of the expected value of empirical additive functionals in the time interval $[0,t]$ under observation, while in the second regime we also 
give a mathematical formulation of the complex mobility matrix.

In the last years several rigorous results  have been obtained for the linear response (and in particular for the Einstein's relation) of  Markov processes, even in a random environment, under  a weak external field homogeneous in time and space, with initial distribution given by  the stationary  one  for the perturbed dynamics
 (see e.g.~\cite{FGS2,GGN,GMP,KO,LR,MP} and references therein). Often the unperturbed  dynamics in these  models is  reversible. Our context is simpler from a technical viewpoint, on the other hand we aim at providing (in a rigorous way) explicit formulas for the linear response under time-dependent or time-periodic weak external  fields, not necessarily homogeneous in space  (without restricting to a  reversible unperturbed dynamics).    As a  natural development one could then consider e.g.
  the linear response in the oscillatory steady state for random walks in random environments (our second regime covers the case of a random walk on the lattice in a periodized environment).

We  now detail  our results in the first regime.
 We  apply a time-dependent weak perturbation such that the perturbed process is again a Markov jump process (now with time-dependent transition kernel), whose law on the path space associated to a finite time interval $[0,t]$ of observation is absolutely continuous (when explosion does not take place) w.r.t.\  the corresponding law of the unperturbed Markov jump process.  
 We isolate an exponential moment condition (see  Condition $C[\nu,t]$ in Definition \ref{def_papaya}) under which we  show  that the perturbed process a.s.~does not explode (see Theorem \ref{th:explosionX})  and linear response takes place. More precisely,   the expected value of the observables at time $t$, as well as of empirical additive functionals in the time-interval $[0,t]$, is differentiable in the perturbation  parameter $\l$ at $\l=0$, and we  provide formulas for  the derivative at $\l=0$ (see Theorem \ref{th:JS}). We point out that non--explosion is unstable under weak perturbation (even of a mild form) as shown by the counterexample in Section \ref{contro}. 
When the initial distribution is stationary for the unperturbed process, our formulas  allow explicit computations in terms of the stationary distribution and the  two-time correlation function of the unperturbed process (see Theorem \ref{cor:JS}).  As examples of  applications of our results, in Section \ref{sec_esempi} we discuss birth and death processes, random walks on $\bbZ^d$ in a confining potential and random walks in a random conductance field.

In deriving Theorems \ref{th:explosionX}, \ref{th:JS} and \ref{cor:JS} mentioned above, we do not use  operator perturbative theory.  Our starting point is the explicit Radon-Nykodim derivative of the law of the perturbed process restricted to paths (without explosion) in the time interval $[0,t]$ w.r.t.\ the law of the unperturbed process. Using  stochastic calculus for jump processes (cf.~\cite{JS} and the short overview provided in Section \ref{sec:SC}), and in particular introducing suitable martingales, we then obtain  both the non-explosion of the perturbed process and the LR  formulas for additive functionals which are cumulative at jump times. We point out that analyzing  the  Radon-Nykodim derivative to derive  LR  has been a common approach in several contributions in probability (see e.g.\  \cite{FGS2,GGN,GMP,KO,LR,MP} and references therein), more often known under the name of ``trajectory-based approach" in statistical physics (see  e.g.\ \cite{BM,M1} and references therein). We mention  the paper  \cite{HM} of Hairer and Majda for a different approach to  the study of  LR in  stochastic systems, and that of Dembo and Deuschel \cite{DD} in which LR, and in particular the Fluctuation Dissipation Theorem,  is discussed as a result of perturbations of Markovian semi-groups.

We now move to the second regime.
When the perturbation (in the same form of the first regime) is time-periodic, the perturbed system admits an oscillatory steady state (OSS), which is left invariant by time translations by multiples of the period. It is then natural to investigate  the LR in the OSS (which is now an infinite-time horizon problem). The rigorous derivation of the existence of the OSS and of the LR is,  in general, not a simple problem, especially if one considers  stochastic processes in a random environment (we refer to \cite{FM} for results on  reversible models without random environment). We restrict here to a finite state space and in Theorems \ref{alpha_omega}, \ref{th_kobo} and \ref{maldive} we describe the LR for the expected value of observables and additive functionals in the OSS. Here we use both matrix perturbation theory and our previous results for the LR over  a finite observation time interval.  As a special model for transport in heterogeneous media, we consider as unperturbed process a random walk on a discrete $d$-dimensional torus with heterogeneous jump rates   (equivalently one could consider  a random walk on $\bbZ^d$ with spatially periodic jump rates). \rosso{In Theorem \ref{teo_CM} for nearest-neighbor jumps and in Theorem \ref{teo_CM_esteso} for long jumps, we derive}  a formula for the complex mobility matrix $\s(\o)$ when the perturbation is of cosine-type in time     (see \cite[Section~1.6]{KTH} for some  examples of complex mobility). In  Section \ref{sec_esempi}  we compute  $\s(\o)$ explicitly in particular cases. When the system is very heterogenous $\s(\o)$ cannot be computed explicitly, but our formulas for $\s(\o)$ remain useful to investigate properties of $\s(\o)$ (as in \cite{FM})  and to prove  homogenization  of $\s(\o)$ under the infinite volume limit   in the  case of random unperturbed jump rates (cf.~\cite{FS}). We also mention  \cite{JPS} for rigorous LR results in the OSS of Langevin dynamics.

\medskip

\noindent
{\bf Outline of the paper}: In Section \ref{mango25} we introduce the unperturbed and the perturbed Markov jump processes, Condition $C[\nu,t]$ and we  discuss explosion.  In Section \ref{sec_response} we present our main results concerning linear response in a finite time window and with a fixed initial distribution. In Section \ref{sec_OSS} we focus on the linear response in the oscillatory steady state of an irreducible  Markov chain  with finite state space and  under time-periodic perturbation. In Section \ref{sec_CM} we analyse the complex mobility matrix for a random walk on a discrete torus with heterogenous jump rates. In Section \ref{sec_esempi} we discuss several examples. In Section \ref{sec:SC} we collect some useful facts from the theory of stochastic calculus for processes with jumps. Sections \ref{sec:preliminary} to 
\ref{tortelli} and  Appendix \ref{appendix} are devoted to proofs.  Finally in Appendix  \ref{sec:pollofritto}  we  consider time-independent perturbations and 
we comment on how our  results in Section \ref{sec_response} compare to  known linear response results  when starting with the invariant distribution of the perturbed process.

 \section{Continuous-time Markov jump processes}\label{mango25}
 
\subsection{Unperturbed Markov jump process}
Let $(\cX, \cB)$ be a measure space such that singletons $\{x\}$ are measurable.  
We consider the Markov jump process  $(X_t)_{t\geq 0}$  with initial distribution $\nu$ and
transition kernel given by $r(x,dy)$. Here $\nu$ is a given probability measure on $(\cX, \cB)$, and  $r(x,dy)$ satisfies the following:  
\begin{itemize}
\item For any $x\in \cX$,  $r(x,\cdot)$ is a measure with finite and positive total mass on $(\cX, \cB)$, and 
\item For any $B\in \cB$, the map $\cX \ni x \mapsto r(x, B)\in [0,+\infty)$ is measurable. 
\end{itemize} 
We define the holding time parameter
\be\label{7guitar}
\hat r(x):= r(x, \cX)\in (0,+\infty) \,, 
\en
and assume that $r(x,\{x\})=0$ without loss of generality. 
Then the stochastic dynamics of $(X_t)_{t\geq 0}$ is  described as follows. At time $t=0$ the Markov jump  process starts with $X_0 $ having distribution $\nu$. Once arrived at $x$, the  process waits there an exponential time with parameter $\hat r (x)$ (independently from the rest), after which it jumps to $y$ with jump probability $r(x,dy)/\hat r(x)$. 

Note that, when $\cX$ is infinite,  such a process may explode in finite time, i.e.~it may be the case that \rosso{$\tau_\infty<+\infty$},  where $\tau_\infty$ denotes the explosion time defined as the supremum of the jump times. By adding a cemetery state $\dagger$ to the state space $\cX$ and setting $X_t = \dagger$ for all $t \geq \tau_\infty$,  we may assume that the Markov jump process is defined for all times.

If $X_0$ has distribution $\nu$, we write  $\bbP_\nu$ for the probability associated to the unperturbed process and $\bbE_\nu$ for the corresponding expectation.

\subsection{Non--explosion of the unperturbed process}\label{sec_kalush}
The following assumption will be understood throughout the text, without further mention:
\subsection*{Assumption}
\emph{From now on we fix a probability measure $\nu$ on $\cX$ corresponding to  the distribution of $X_0$, and assume non--explosion of the unperturbed process $\bbP_\nu$--almost surely, without further mention. When $\nu$ is the stationary distribution we will denote it by $\pi$ (see Section \ref{sec:stationary}). }
\medskip

 Trivially,  if  $\sup_{x\in \cX} \hat r(x)<+\infty$, then the unperturbed process does not explode $\bbP_\nu$--a.s.  as can be  easily checked by a suitable coupling with a Poisson process. When  $\hat r(\cdot)$ is unbounded, the existence of a Lyapunov function is enough to guarantee non--explosion. Let us explain this point in more  detail.  
 Given a measurable function $f:\cX\to \bbR$ such  that  either $\int _{\cX} |f(y)| r (x,dy)<+\infty$ for all $x\in \cX$, or $f\geq 0$, or $f\leq 0$, we define
 \be\label{usignolo}
 Lf (x) := \int _{\cX} [f(y)-f(x)] r (x,dy)\,.
 \en  
 Note that, due to the assumptions on $f$, the r.h.s.\ of \eqref{usignolo} is well defined in $\bbR\cup \{-\infty,+\infty\}$. We call the above operator $L$ the formal generator of the Markov jump process.
Then, by \cite[Theorem~4.6]{V}, for the unperturbed process not to explode for any initial point (and therefore  also  $\bbP_\nu$--a.s.) it suffices that there exist a constant $C\geq 0$  and  a non--negative function $U$ on $\cX$ such that
\be\label{grano}
LU (x) \leq C U(x) \qquad \forall x \in \cX
\en
and $U(x)\to +\infty$ whenever $\hat r(x)\to +\infty$.

\subsection{Perturbed Markov jump process}
 We fix a \underline{bounded} measurable function $g: [0,+\infty ) \times \cX \times \cX \to \bbR$. Given $\l>0$, the $\l$--perturbed Markov jump process $(X^\l_t)_{t\geq 0}$ is the time--inhomogeneous Markov jump process  with initial distribution $\nu$ and 
transition kernel
 \begin{equation}\label{kolpakovZZZ}
  r^\l_t (x,dy) = r (x,dy) e^{ \l g(t,x,y) }\,.
 \end{equation}
The precise definition of $X^\l:=(X^\l_t)_{t\geq 0}$  can be given in terms  of piecewise deterministic Markov processes (PDMPs) (cf.~\cite{D}):   $\bigl( t, X^\l_t \bigr)_{t\geq 0}$ is the time--homogeneous PDMP with  vector field $\partial _t$ and transition kernel $Q( (s,x), (dt, dy) )= \d_s(dt) r^\l _t(x,dy) $. To recall the construction of $X^\l$ we introduce the 
  holding time parameters 
  \be\label{7guitarZZZ} \hat{r}^\l_t (x) := \int _{\cX} r^\l _t(x,dy)= \int_{\cX} r(x,dy) e^{ \l g(t,x,y) }\,.
\en
Note that, as the function $g$ is bounded and due to \eqref{7guitar}, we have $\hat{r}^\l_t (x)\in (0,+\infty)$ for all 
$x\in \cX$. Then, 
 up to the possible explosion time $\tau_\infty ^\l$,   the process  $X^\l_t $ can be realized as follows.  
Starting from a state $x$, the Markov jump process   spends at  $x$ a random time $\tau^\l_1$ such that 
$$ P( \tau^\l_1 >t)= \exp \left\{ - \int _0 ^t \hat{r}^\l _s(x ) ds \right\} \,.
$$
Knowing that $\tau^\l_1=t_1$, at time $t_1$ the Markov jump process  jumps to a new state $x_1$ chosen randomly with probability $r^\l_{t_1}(x,d x_1 )/ \hat{r}^\l_{t_1}(x)$. It then  waits at $x_1$  until  the time $\tau _2^\l>t_1$ with law
$$ P ( \tau^\l_2>t)= \exp \left\{ - \int _{t_1} ^{t} \hat{r}^\l_s(x_1) ds \right\}\,, \qquad t\geq t_1 \,.
$$
Knowing that $\tau_2^\l=t_2$, at time $t_2$ the Markov jump process  jumps to a new state $x_2$ chosen randomly with probability $r^\l_{t_2}(x_1,dx_2)/ \hat{r}^\l_{t_2}(x_1)$,
and so on. Again if the process explodes in finite time we set $X^\l_t = \dagger$ for all $t\geq \tau^\l_\infty$, so that the perturbed process is well defined for all times. 
%

\begin{Remark}  In what follows we will mainly be  interested in the perturbed process in some time interval $[0,t]$.  Due to the above construction, it is clear that then only the value of $g$ up to time $t$ is relevant. As a consequence, in the rest $g$ will simply be a bounded measurable function defined for times varying in the 
 observation time interval.
\end{Remark}

If $X^\l_0$ has distribution $\nu$, we write  $\bbP_\nu$ for the probability associated to the perturbed process and $\bbE_\nu$ for the corresponding expectation (the notation is the same as the one we use for the unperturbed process, but the event and function under consideration will present the superscript $\l$).

\subsection{Finite exponential moments condition}
Let us introduce the following notation, that will be used throughout. 
For  $\a:[0,t] \times \cX \times \cX \to\bbR$ measurable function and  $r(x,dy) $ transition kernel of the unperturbed dynamics, the contraction of $\a$ with respect to the  kernel $r$ is defined by
	\begin{equation} \label{def:contractionX}
	 \a_r (s,x) := \int_\cX \a(s,x,y) r(x,dy) 
	 \end{equation}
	 \rosso{(when the integral in the r.h.s. is well posed)}.
With this notation in place we can define the finite exponential moments condition, which in particular will assure the linear response 
regime when applied to $\a=g$.

\begin{Definition}[Exponential moments condition]\label{def_papaya}
We say that $\a:[0,t] \times \cX \times \cX \to\bbR$ satisfies Condition $C[\nu,t]$ with parameter  $\theta>0$ if 
\begin{equation}\label{papayaX}
\bbE_\nu \Big[ \exp\Big\{\theta  \int_0 ^t | \a|_r (s, X_s)  ds \Big \}\Big]<+\infty\,.
\end{equation}
We say that $\a$ satisfies Condition $C[\nu,t]$ if the above holds for some parameter $\theta>0$.
\end{Definition}
We now give a  criterion assuring  Condition $C[\nu,t]$. Recall from \eqref{usignolo}  the definition of $Lf$.

\begin{Lemma}\label{tik_tok}  For a given function $\alpha : [0,t] \times \cX \times \cX \to \bbR$ assume that there exist a function $U : \cX \to \bbR$  and positive  constants $\theta,C,c$  such that
\begin{itemize}
\item[(a)]  $U(x) \geq c  $ for all $x\in \cX$;
\item[(b)]    $U_r (x) : = \int_{ \cX} U(y) r(x,dy)  <+\infty$ for all $x\in \cX$;
 \item[(c)]  $L U \leq C U  - \theta \,|\a |_r  U  $;
 \item[(d)] $\nu[ U] <+\infty$.
 \end{itemize}
 Then $ \a$ satisfies Condition $C[\nu,t]$ with parameter $\theta$. 
\end{Lemma}
Note  that, if   $U(x) \to +\infty$ when $\hat r(x) \to +\infty$, then  Item (c) in Lemma \ref{tik_tok} is a reinforced Lyapunov condition (compare with \eqref{grano}).

This criterion is a special case of a more general (and more technical) criterion presented  in Lemma \ref{tik_tok_ext}  in Section \ref{dim_lemma_tik_tok},  
inspired by Lyapunov functions and the arguments in \cite[Section~3]{BFG1}.  See \cite[Condition 2.2]{BFG1} and \cite[p.~392]{DV3} for related conditions
in the context of large deviations. We point out that, while  Lemma \ref{tik_tok} gives   sufficient conditions for Condition $C[\nu,t]$ to hold,  in some cases one can directly  and more efficiently verify Condition $C[\nu,t]$ using Definition \ref{def_papaya}. To this aim, see the example in Section 
\ref{nascite}.

\smallskip 
The next result states that the exponential moment condition $C[\nu ,t ]$ implies finiteness of small exponential moments for the sum of the values of $\alpha$ over the jumps of the unperturbed process.

 \begin{Lemma}\label{ananas78X}
Given  $\alpha : [0,t] \times \cX \times \cX \to \bbR$ measurable and bounded, suppose that  $\a$ satisfies Condition $C[\nu,t]$ with parameter $\theta>0$.  
Then for $\g:= 4^{-1} \min\{ \theta, \|\a\|^{-1}_\infty \} $ it holds
\be\label{kinney2}
\bbE_\nu \Big[ \exp\Big\{\g  \sum_{ \substack{s\in (0,t]: \\  X_{s-} \neq X_s} }
  |\a (s, X_{s-}, X_s) | \Big \}  \Big]<+\infty\,.
\en
\end{Lemma}
The above lemma in proved in Section \ref{sec:preliminary}.  We remark that the condition $\a$ bounded is necessary: as a counterexample one can take the unperturbed process $(X_s)_{s\in [0,t]}$ to be a Poisson process of rate $1$ (with $\nu=\d_0$), and pick $\a (\cdot ,x,y) := x$. Then,   Condition $C[\nu, t]$ is satisfied while \eqref{kinney2} is violated for all $\gamma >0$. Indeed, in this case $\int_0^t |\a|_r (s,X_s) ds = \int _0 ^t X_s ds \leq t X_t$, thus allowing to check Condition $C[\nu,t]$. On the other hand,
 the sum in \eqref{kinney2} equals $ (X_t^2- X_t)/2\geq (X_t^2-1)/4$ and $\bbE_\nu[e^{(\g/4) X_t^2}]$ diverges for all $\g>0$.

\subsection{Non--explosion of the perturbed process}
We recall that $\tau_\infty ^\l $ denotes the explosion time of the perturbed process $X^\l$, given by the supremum of the jump times. We also  recall that we have assumed that the unperturbed process with initial distribution $\nu$ a.s.\ does not to explode.

 As already mentioned, our linear response results will be derived under the assumption that $g$ satisfies condition $C[\nu,t]$. In fact, this condition automatically implies that the perturbed process does not explode, and hence we do not need to assume non-explosion of the perturbed process separately. Of course, if one is just interested in the non-explosion of the perturbed process, one can more efficiently use the criteria developed e.g.\ in \cite{CZ}.

Recall that $g$, defined in \eqref{kolpakovZZZ}, is measurable and bounded. 
\begin{Theorem}\label{th:explosionX}
Suppose that $g$ satisfies Condition $C[\nu,t]$ with parameter $\theta>0$.
Then  for all  $\lambda\leq 8^{-1} \min\{  \theta  ,\; \|g\|^{-1}_\infty\}$,
the perturbed process $X^\l$ does not explode in $[0,t]$  $\bbP_\nu$--a.s., i.e.~$\bbP_\nu (\tau^\l_\infty >t ) =1$. 
\end{Theorem}
The above theorem is proved in Section \ref{sec:explosion} using stochastic calculus techniques inspired by \cite{PR} (see Lemma 3.1 therein).

\begin{Remark}[Instability of non-explosion under small perturbations]  \label{rem_instabile}
At this point the reader may wonder whether the fact that $g$ is assumed to be bounded, by itself implies that if the unperturbed process does not explode then the perturbed process does not either, at least for $\lambda$ small enough. This turns out not to be the case: see Section \ref{contro} for a counterexample.
\end{Remark}

%
%

%
%

\smallskip




%
%
%
%
%

\section{Linear response of Markov jump processes}\label{sec_response}
We start by fixing some notation. We denote a path $(\xi_s)_{s\in [0,t]}$ simply by $\xi_{[0,t]}$. $D([0,t],\cX)$ is   the Skohorod space of c\`adl\`ag paths from $[0,t]$ to $\cX$, while $D_f ([0,t],\cX)$ is the subset of  $D([0,t],\cX)$ 
given by the paths  with a finite number of jumps. For any $\xi_{[0,t]} \in D_f ([0,t],\cX)$, we abbreviate 
	\begin{equation}\label{shorthand}
	 \sum_{s\in(0,t]} \a(s, \xi_{s-}, \xi_s)  := 
	\sum_{s\in(0,t]\,:\, \xi_{s-}\not =\xi_s}  \a(s, \xi_{s-}, \xi_s) 
	\end{equation}
throughout this note.

Below we will assume that $g$ satisfies Condition $C[\nu,t]$ and we will  take $\l$ small. As a consequence, by Theorem \ref{th:explosionX},   the perturbed  Markov jump process  does not explode $\bbP_\nu$--a.s. in the time interval $[0,t]$. 
Due to non explosion (recall our main assumption at the beginning of Section \ref{sec_kalush}),
 almost surely  the paths $X_{[0,t]}$ and $X^\l_{[0,t]}$ belong to the set $D_f ([0,t],\cX)$.

As in the trajectory-based approach to linear response (cf. \cite{BM,M1}), the starting point to analyze the response of the perturbed system is the following well-known Girsanov-type expression, which can be easily verified:
for any measurable function  $F: D([0,t],\cX)\to \bbR$, bounded or non-negative, and any initial distribution  $\nu$ it holds
\be\label{ugo}
\bbE_{\nu}\Big[ F(X^\l_{[0,t]} )  \Big]=\bbE_{\nu}  \Big[ F(X_{[0,t]}) 
e^{- \cA_\l ( X_{[0,t]} )}\Big]
\en
where the \emph{action} $\cA_\l :D_f ([0,t],\cX)\to \bbR$ is defined as (see \eqref{7guitar}, \eqref{kolpakovZZZ} and  \eqref{7guitarZZZ})
\begin{equation}\label{ughetto}
\begin{split}
\cA_\l\bigl( \xi_{[0,t]}\bigr) :&=
  \int_0^t \bigl[ \hat{r}^\l_s(\xi_s)- \hat r ( \xi_s)\bigr] ds -\l
  \sum _{s \in (0,t] } g(s,\xi_{s-} , \xi_s)
\\
&= \int _0^t ds \int_\cX r(\xi_s,dy) \bigl(e^{\l g(s, \xi_s,y)}-1\bigr)  - \l 
 \sum _{s \in (0,t]} g(s, \xi_{s-}, \xi_s) \,.
 \end{split}
\end{equation}

The next result, proved in Section \ref{sec:derivoX}, is the starting point of our linear response analysis. 

\begin{Proposition}\label{derivoX}  
Suppose that $g$ satisfies Condition $C[\nu,t]$.
Then for any measurable function 
 $F: D_f([0,t],\cX)\to \bbR$    such that $F \bigl( X_{[0,t] } \bigr)\in L^p(\bbP_\nu)$ for some $p\in (1,+\infty]$, the map 
   $ \l \mapsto  \bbE_\nu \bigl[ F \bigl( X^\l_{ [0,t] } \bigr) \bigr] $ 
 is differentiable at $\l=0$. Moreover,  it holds  
  \be\label{titans2X}
\partial_{\l=0} \bbE_\nu \bigl[ F \bigl(  X^\l _{ [0,t] } \bigr) \bigr] = \bbE_\nu\bigl[F  \bigl( X_{ [0,t] } \bigr) 
G_t   \bigl( X_{ [0,t] } \bigr)  \bigr]
\en
where the map $ G_t: D_f( [0,t]; \cX) \to \bbR$ is defined by
\begin{equation}\label{ginoX} 
G_t \big(\xi_{ [0,t]} \big) :=  \sum_{s\in (0,t]} g(s,  \xi_{s-}, \xi_s) - \int_0^t g_r  (s, \xi_s)  ds  \,\end{equation}
with the shorthand notation introduced in \eqref{shorthand}. 
\end{Proposition}


The above statement should be understood to include that all the expectations appearing are well defined and finite under the stated assumptions. 
Although the time $t$ is fixed once and for all and omitted from the notation, for later use we have made explicit the dependence on $t$ of $G_t$. We also point out that  one could give  a quantitative bound on the range of values of $\l$ for which the claim in  Proposition \ref{derivoX} holds true by taking more care of the constants in the proof.
\begin{Remark} \label{rem:Gmart} As we will show in Section \ref{sec:SC},  provided g satisfies Condition $C[\nu,t]$,  $G_t(X_{[0,t]})$ is a martingale (it is in fact a purely discontinuous martingale, in the sense of \cite[Def.~4.11]{JS}). As a consequence, the r.h.s.\ of \eqref{titans2X} equals the covariance ${\rm Cov} ( F  \bigl( X_{ [0,t] } \bigr) ,G_t   \bigl( X_{ [0,t] } \bigr)  \bigr)$ with respect to the probability measure $\bbP_\nu$.
\end{Remark}

 \smallskip
 
\subsection{Linear response for observables and additive functionals}\label{sec:types}
We can  give explicit expressions for the r.h.s.\ of \eqref{titans2X} for specific classes of functionals $F$. We are mainly interested in the following three basic cases (by additivity, functionals given by sums of the following ones can be treated as well):
\begin{itemize}
\item[(1)] $F\bigl(  \xi_{  [0,t] } \bigr) =v(\xi_t)$ for some measurable function $v: \cX \to \bbR$; 
\item[(2)] $F\bigl(  \xi_{[0,t] } \bigr) = \displaystyle  \int_0^t  v(s,\xi_s) ds$, with  $v:[0,t]\times \cX\to \bbR$ measurable;  
\item[(3)] $F\bigl(  \xi_{ [0,t] } \bigr) = \displaystyle 
\sum _{s \in (0,t]} \alpha (s,  \xi_{s-}, \xi_s) $ for  $\alpha : [0,t] \times \cX \times \cX \to \bbR$ measurable.
\end{itemize}
To this aim, fix the following terminology.

\begin{Definition}\label{def:Pint}
We say that  a measurable function $\a :[0,t ] \times \cX \times \cX \to \bbR$ is $\bbP_\nu$--integrable if one of the following equivalent bounds is satisfied: 
	 \begin{equation}\label{intA}
	  \bbE_\nu \bigg[ \sum _{s \in (0,t]} |\a  (s,  X_{s-}, X_s)| \bigg] 
	< \infty  , \qquad 
	 \bbE_\nu \bigg[ \int_0^t  |\a|_r (s, X_s) ds \bigg]  < \infty . 
	\end{equation}
\end{Definition}
The equivalence in the above definition comes from the following fact:
\begin{Lemma} \label{giacinto}
Given a measurable function $\a :[0,t ] \times \cX \times \cX \to \bbR$, it holds
 \begin{equation}\label{intA_zero}
	  \bbE_\nu \bigg[ \sum _{s \in (0,t]} |\a  (s,  X_{s-}, X_s)| \bigg] 
	=
	 \bbE_\nu \bigg[ \int_0^t  |\a|_r (s, X_s) ds \bigg]   . 
	\end{equation}
In particular, 
the two bounds in \eqref{intA}  are equivalent. As a consequence, if $\a$ satisfies Condition $C[\nu,t]$, then $\a$ is $\bbP_\nu$--integrable.
\end{Lemma}
The proof of the above lemma is given in Section \ref{sec:SC}. For the next result recall the definition of $G_t$ given in  \eqref{ginoX}.  
  \begin{Theorem}\label{th:JS}
Suppose that $g$ satisfies Condition $C[\nu,t]$.
Then  
 the following holds:
 \begin{itemize}
 \item[(1)] Let  $v: \cX \to \bbR$ be a measurable function such that $v(X_t) \in L^p(\bbP_\nu)$ for some $p\in(1,+\infty]$. Then 
 	\begin{equation}\label{res1}
 	 \partial_{\l=0} \bbE_\nu \bigl[ v ( X^\l_t)\bigr] = 
 	 \bbE_\nu \bigl[ v(X_t ) G_t ( X_{ [0,t] } ) \bigr] . 
 	 \end{equation}
 \item[(2)]  For $v: [0,t]\times \cX \to \bbR$ measurable such that  $\int_0^t \|v(s,X_s) \|_{ L^p(\bbP_\nu )}ds <+\infty $ for some $p\in(1,+\infty]$, it holds 
  	\begin{equation}\label{res2}
 	 \partial_{\l=0} \bbE_\nu \Bigl[ \int_0^t v (s, X^\l_s) ds \Bigr] = 
 	 \int_0^t \bbE_\nu \bigl[ v(s,X_s ) G_s (X_{ [0,s] } ) \bigr] ds  . 
 	 \end{equation}
 \item[(3)] Let  $F : D_f ([0,t] ; \cX ) \to \bbR $  be the additive functional of the form 
 	 	\begin{equation}\label{eq:F}
 	 	 F\left(\xi_{[0,t]} \right) = \sum _{s \in (0,t]} \alpha (s,  \xi_{s-}, \xi_s) \,, 
 	 	 \end{equation}
 	with $\alpha : [0,t] \times \cX \times \cX \to \bbR$ 
measurable and such that 
\begin{equation} \label{aa}
 \sum _{s \in (0,t]} |\a  (s,  X_{s-}, X_s)| \quad 
\mbox{ and }  \quad 
	 \int_0^t  |\a|_r (s, X_s) ds 
	 \end{equation}
 belong to $L^p ( \bbP_\nu )$ 
for some $p \in (1, + \infty ]$. 
For example take $\a $ bounded and such that it satisfies Condition $C[\nu , t ]$. 
Then  it holds 
 	 	\begin{equation}\label{res3}
 	 	\begin{split} 
 	 	 \partial_{\l=0} \bbE_\nu \bigl[ F  \bigl( X^\l_{ [0,t] } \bigr) \bigr] 
 	 	  =  & \int_0^t \bbE_\nu \big[ (\alpha g )_r (s,X_s) \big] ds \\ & +  
 	\int_0^t \bbE_\nu \big[ \alpha_r (s,X_s ) G_s (  X_{ [0,s] } ) \big]  ds , 
 	\end{split} 
 	\end{equation}
 where $\a_r $ and $(\a g )_r$ denote the contraction of the functions $\a , \a g$ with respect to the transition  kernel $r$, as in \eqref{def:contractionX}. 
 \end{itemize}
 \end{Theorem}

 
 \smallskip 

The above statement should be understood to include that all the expectations appearing are well defined and finite under the stated assumptions. 
The proof of Theorem \ref{th:JS}   is given in Section  \ref{sec:JS1}. Stochastic calculus for processes with jumps will be crucial to derive the above Item (3), we collect in Section \ref{sec:SC}  the needed theoretical background. 

\smallskip


\subsection{Linear response at stationarity} \label{sec:stationary}
 A special role is played by invariant distributions. We recall that a distribution $\pi$ on $\cX$ is called \emph{invariant} for the Markov jump process $(X_t)_{t\geq 0}$ if, when starting with initial distribution $\pi$, it holds  $(X_{t+T}) _{t\geq 0} \stackrel{\cL}{=} (X_t)_{t\geq 0}$ for all $T>0$. If there is no explosion, 
a distribution $\pi$ is invariant if and only if we have the following identity between measures on $\cX$:
\be 
\pi (dx) \int _{\cX} r(x, dy)= \int _{\cX} \pi( dy) r(y,dx)\,,
\en
i.e.~ $\pi(dx) \hat r (x)=  \int _{\cX} \pi( dy) r(y,dx)$.
We denote by $(X^*_t)_{t\geq 0}$  the stationary time-reversed process. 
 This is again a non-explosive  Markov jump process with initial distribution $\pi$ and with transition kernel $r^*$ satisfying the detailed balance equation
\be\label{secondino}
\pi(dx)r(x,dy)= \pi(dy) r^*(y,dx)\,.
\en
Note that \eqref{secondino} is an identity between measures on $\cX \times \cX$.  When $(X_t)_{t\geq 0}$ is a Markov chain, writing $r(x,dy)$ as $r(x,y) \d_y$ and $\pi(dx)$ as $\pi(x) \d_x$, we have the explicit well known expression  $r^*(y,x)= \pi(x) r(x,y) / \pi(y)$. For generic Markov jump processes with non atomic measure $r(x,dy)$, the transition kernel  $r^*(y,dx) $ might not be explicit.

Set
\[ g^*(s,x,y) := g(s,y,x)\,, \]
and introduce the function 
\be \label{def:psi}
\psi_s(x):=\int_\cX   g(s,y,x) r^*(x,dy) - \int_\cX g(s,x,y)r(x,dy) 
  =
g^*_{r^*} (s,x) - g_r (s,x)\,.
\en 

\begin{Theorem}\label{cor:JS} Suppose that the unperturbed Markov jump process is stationary with initial distribution $\pi$.   
Then, under the assumptions of Theorem \ref{th:JS} with $\nu$ replaced by $\pi$ and with the same notation for the functionals, we have:
\[ \begin{split} 
 \partial_{\l=0} \bbE_\pi \bigl[ v ( X^\l_t) \bigr] & = 
  \int_0^t ds \, \bbE_\pi \bigl[ v(X_t) \psi_{t-s} (X_{t-s} ) ] =  \int_0^t ds \, \bbE_\pi \bigl[ v(X_s) \psi_{t-s} (X_0 ) ]
\\ 
\partial_{\l=0} \bbE_\pi \Bigl[ \int_0^t v (s, X^\l_s) ds \Bigr] 
& = \int_0^t ds \int_0^s du \, \bbE_\pi \bigl[ v ( s ,X_s) \psi_{s-u} (X_{s-u} ) \bigr] 
 \\
	\partial_{\l=0} \bbE_\pi \bigl[ F \bigl( X^\l_{ [0,t] } \bigr) \bigr]&  = 
	 \int_0^t \bbE_\pi \big[ (\alpha g )_r (s,X_s ) \big] ds  
	+ \int_0^t ds \int_0^s du \, 
	\bbE_\pi \big[ \alpha_r (s, X_s ) \psi_{s-u}(X_{s-u} )  \big] 
	.
	\end{split} 
	\]
If, in particular, the perturbation $g$ is of the form $g(s,x,y) = \tau (s) E(x,y)$ (decoupled case), then with $E^*(x,y) := E(y,x)$ 
	\[ \begin{split} 
	 \partial_{\l=0} \bbE_\pi \bigl[ v ( X^\l_t) \bigr] 
 &  = 
\int_0^t ds \, \tau (t-s) \bbE_\pi \bigl[ v(X_s) \big(  E^*_{r^*}( X_{0}) - E_r( X_{0} ) \big) \bigr]
\\ 
\partial_{\l=0} \bbE_\pi \Bigl[ \int_0^t v (s, X^\l_s) ds \Bigr] 
 & = 
\int_0^t ds \int_0^s du \, \tau (s-u ) \bbE_\pi \bigl[ v (s,X_u) \big( E^*_{r^*} ( X_{0} ) - E_r ( X_{0} ) \big) \bigr] 
\\
	\partial_{\l=0} \bbE_\pi \bigl[ F \bigl(  X^\l_{[0,t] } \bigr) \bigr]&  = 
	 \int_0^t ds \, \tau (s) \bbE_\pi \big[ (\alpha E )_r (s,X_s ) \big] 
	\\ & + \int_0^t ds \int_0^s du \, \tau (s-u) \bbE_\pi \Big[ 
	\alpha_r (s, X_u ) \big( E^*_{r^*} ( X_{0} ) - E_r ( X_{0} ) \big) \Big]
	\,.
	\end{split} 
	\]
\end{Theorem}
The proof of Theorem \ref{cor:JS}   is provided in Section  \ref{sec:JS2}. Note that the second and third formulas in Theorem \ref{cor:JS}  can be rewritten by replacing $\bbE_\pi [ v ( s ,X_s) \psi_{s-u} (X_{s-u} ) ] $ with $\bbE_\pi  [ v ( s ,X_u) \psi_{s-u} (X_{0} ) ] $ and $\bbE_\pi [ \alpha_r (s, X_s ) \psi_{s-u}(X_{s-u} )  ] $ by $\bbE_\pi [ \alpha_r (s, X_u ) \psi_{s-u}(X_{0} )  ]$ (the equivalence follows from the stationarity of $\pi$).

\begin{Remark} \label{rem:zeromean} 
Note that in the stationary case, covered by Theorem \ref{cor:JS}, the linear response of  all the functionals under consideration can be computed explicitly from the 2-time distributions  of the stationary time-reversed process.
Moreover, we note that the random variable $ \psi_{s-u}(X_{s-u}) = g^*_{r^*} (s-u , X_{s-u} ) - g_r (s-u , X_{s-u} )  $ has $\bbP_\pi$--zero mean, since 
	\[ \begin{split}
	\bbE_\pi \big[ g^*_{r^*} (s-u , X_{s-u} ) \big] & = 
	\bbE_\pi \bigg[ \int_\cX g( s-u , y , X_{s-u} ) r^* (X_{s-u} , dy ) \bigg] 
	\\ & = \int_\cX  \int_\cX g(s-u , y , x ) \pi ( dx ) r^* (x,dy) 
	\\ & = \int_\cX  \int_\cX g(s-u , y , x ) \pi ( dy ) r (y,dx) 
	\\ & = \bbE_\pi \bigg[ \int_\cX g( s-u , X_{s-u},y ) r (X_{s-u} , dy ) \bigg] 
	= \bbE_\pi \big[ g_r (s-u , X_{s-u} ) \big] . 
	\end{split}\]
As a consequence, the $2$--time expectations appearing in the first part of Theorem \ref{cor:JS} are indeed correlations. 
\end{Remark}
 A comparison of our results with the  linear response when starting with the invariant distribution of the perturbed process is provided in Appendix \ref{sec:pollofritto}.

%
%

\section{Linear response of periodically driven Markov jump processes  in the oscillatory steady state}\label{sec_OSS}

In this section, and the next one, we focus on linear response of Markov jump processes in the oscillatory steady state. 
We take $\cX$ finite and we consider the unperturbed Markov jump process  $(X_t)_{t\geq 0} $ on $\cX$ with transition rates $r(x,y)$ (with our previous notation the transition kernel would be $r(x, dy)=\sum_{z\in \cX} r(x,z) \d_{z}(dy)$).

\begin{Assumption}\label{semplice}
The process  $(X_t)_{t\geq 0}$  is irreducible, i.e.\  it can go from any state $x$ to any $y$ via jumps with positive transition rate.  
\end{Assumption}
The above assumption is equivalent to the fact that zero is a simple eigenvalue of the generator $\cL$. 
We call $\pi$ the  unique invariant distribution of  the unperturbed Markov jump process. 

The perturbed process $(X_t^\l)_{t\geq 0}$ is then the Markov jump process  with transition rates 
	\[ r^\l_s(x,y)= e^{\l g(s,x,y)} r(x,y)\,,\] 
$g(\cdot, x,y)$ being \underline{periodic} on $\bbR$, bounded  and measurable with period $T\in (0,+\infty)$ for any $x,y\in \cX$. As $\cX$ is finite and $g$ is bounded, no explosion takes place. Moreover,  also the discrete--time Markov chain $(X_{n T}^\l)_{n\geq 0}$ is irreducible and therefore it admits a unique invariant distribution $\pi_\l$. Then the law of the perturbed process $(X_t^\l)_{t\geq 0}$ with initial distribution $\pi_\l$ (called \emph{oscillatory steady state}, shortly OSS) is left invariant by time  translations which are multiples of $T$.   It is simple to check that $\pi_\l$ is indeed the unique initial distribution leading to  this type of invariance.   In what follows we aim to investigate the linear response of mean observables and additive functionals on the time interval $[0,t]$ under $\bbP_{\pi_\l}$ (note that now the initial distribution changes with $\lambda$).

 We consider  the complex Hilbert space  $L^2(\pi )$ with  scalar product
\be\label{scalare}
\la f, h\ra = \sum_{x \in \cX} \pi (x) \bar f(x) h(x) 
\en
and write $\|\cdot \|$ for the associated norm.
We define $\cL:L^2(\pi )\to L^2(\pi )$ as the Markov generator of the unperturbed process   $(X_t)_{t\geq 0}$ and write $\cL^*$ for its adjoint operator in $L^2(\pi )$:
\begin{align*} 
 \cL f (x) &=\sum _{y\in \cX } r(x,y) [ f(y)-f (x)]\,,\;\qquad x \in\cX\,,\\
 \cL^* f(x) & = \sum _{y \in \cX} r^*(x,y)[ f(y)-f(x) ]  \,,\qquad x \in\cX\,,
\end{align*}
where $r^*(x,y)= \pi(y) r(y,x)/\pi(x)$. Then $\la f,\cL h\ra = \la \cL^* f, h \ra $ for all $f,h\in L^2(\pi)$. 
The following lemma will be proved in Section \ref{tortelli}. 
\begin{Lemma}\label{cop25}
Zero is a simple eigenvalue of $\cL^*$ with eigenspace given by the  constant functions. All other complex eigenvalues of $\cL^*$ have strictly negative real part. 
\end{Lemma}

 We set  \[ L^2_0(\pi):= \{f\in L^2(\pi)\,:\, \pi[f]=0\}\,,\]
where $\pi [f] = \sum_{x} \pi (x) f(x) $. Then 
 $\cL^*$ is an isomorphism if restricted to  $L^2_0(\pi)$, indeed  $\pi [\cL^* f]=0$ by stationarity of $\pi$ (hence $\cL^* f\in L^2_0(\pi)$) and $\cL^*$ restricted to the finite-dimensional space $L^2_0(\pi)$  is injective by  Lemma \ref{cop25}.   In what follows, we use the following notation:
 \be\label{raggio}
 f\in L^2_0(\pi) \; \;\Rightarrow \; \;(\cL^*)^{-1} f := h \text{ where } h \in L^2_0(\pi)\,,\; \cL^* h=f\,.
 \en
Moreover, given $c\in \bbR\setminus\{0\}$,  the operator $(ic + \cL^*): L^2(\pi)\to L^2(\pi)$ is an isomorphism, since it is injective by Lemma \ref{cop25} and $ L^2(\pi)$ is finite dimensional.

We can decompose the space $L^2(\pi)$ as direct sum of the $\cL^*$--invariant subspaces  $ L^2_0(\pi)$ and $\{\text{constant functions}\}$. Furthermore, we can decompose  $ L^2_0(\pi)$ as direct sum of $\cL^*$--invariant subspaces where, in a suitable basis, $\cL^*$ has the canonical  Jordan form. Fixed a dimension $n$, let $A_i$ be the matrix with ones on the $i$--th upper diagonal, and zeros on the other entries (i.e. $(A_i)_{j,k}= \d_{j+i,k}$, thus implying that $A_0=\bbI$). The canonical Jordan form in dimension $n$ is given by  $J_\g:= \g \bbI+A_1$ for some $\g\in \bbC$. We have  $e^{ s J_\g}= e^{s \g } (\bbI+s A_1+ (s^2/2!) A_2+\cdots + (s^{n-1}/(n-1)!) A_{n-1} )$.  Therefore, if $\Re(\g)<0$, all entries of $e^{s J_\g}$ decay exponentially in $s$.  Moreover, since 
for $\g \not =0$   we have $J_\g^{-1}=\g^{-1} \bbI - \g^{-2} A_1+ \g^{-3} A_2+ \cdots + (-1)^{n-1}\g^{-n} A_{n-1}$,  it is simple to check that $\int_0^{+\infty} e ^{s J_\g}ds  =- J_\g^{-1}$ if $\Re(\g)<0$. Since $ic+ J_\g=J_{ic+\g}$, the above formula also implies that $\int_0^{+\infty} e ^{(ic+J_\g)s}ds  =- (ic+J_\g)^{-1}$ if $\Re(\g)<0$.
Writing $\|\cdot\|$ for  the norm in $L^2_0 (\pi)$,
the above observations and Lemma \ref{cop25} imply that  there exists $\k >0$ such that 
\be\label{pandorino89}
\| e ^{s \cL^* } f   \| \leq   e^{- \k s} \|f \|  \qquad \forall  f \in L^2_0 (\pi) 
\en
and that (recall \eqref{raggio})
\be\label{linz}
 (ic +\cL^*)^{-1} f  =- \int _0^{+\infty}  e^{(ic+\cL^*)s } f ds \,, \qquad \forall c\in \bbR\,,\; \forall f\in L^2_0 (\pi)\,.\\
\en
We will frequently use the above formulas in what follows.

 We introduce the transition   matrix  $P_{\l,t}=\bigl( P_{\l,t}(x,y) \bigr)_{x,y\in \cX}$ defined as 
 \[   P_{\l,t} (x,y):= \bbP_x ( X_t^\l=y)\,.
 \]
  When $\l=0$ we simply write $P_t$. Note that, for $t>0$,  the matrix $P_{\l,t}$ has positive entries. Hence, by Perron-Frobenius Theorem, 
   $1$ is a simple eigenvalue of  $P_{\l,t}$ for $t>0$ and  the distribution $\pi_{\l} $ is the only  row vector satisfying  $\pi _{\l} P_{\l,T}= \pi_{\l}$, $  \sum _{x\in \cX}
 \pi_{\l} (x)=1$.

 By Proposition \ref{derivoX}  the matrix $P_{\l,t}$  is differentiable at $\l=0$.  As $1$ is a simple eigenvalue of $P_t$, by standard finite dimensional perturbation theory 
 \cite{Ka} we get that $\pi_\l$ is differentiable at $\l=0$. 
 By setting $\dot \pi:= \partial_{\l=0} \pi_\l$ and $\dot P_{T}:= \partial_{\l=0} P_{\l,T}$ we have 
 \be\label{pianino}
 \dot \pi (P_T-\bbI)= - \pi \dot P_T\,.
 \en
 Define 
 \[a(x) := \dot \pi(x) /\pi(x) \qquad \forall x \in \cX\] and recall from 
 \eqref{def:psi} 
 that  \be\label{def:psi:bis}
\psi_t(x):=\sum _{y \in \cX}   \left(r^*(x,y) g(t,y,x)- r(x,y) g(t,x,y)\right)
  =
g^*_{r^*} (t,x) - g_r (t,x)\,.\en  In what follows we think of $a$ and $\psi_t$ as column vectors. Note that $\psi_t$ is $T$--periodic in time. Moreover, $\psi_t \in L^2_0(\pi)$ for all $t$ by Remark \ref{rem:zeromean}. 
Due to \eqref{pandorino89}  and since $\sup_{t\in \bbR} \|\psi_t \|<+\infty$, we  get for some $C,\k>0$ that 
\be\label{stimetta}
\sup_{u}\| e^{s \cL_*} \psi_u \| \leq C e ^{- \k s}\qquad \forall s\geq 0
\,.
\en
 In particular,  the integral $ \int _0^\infty ds  \,e^{s \cL^*}\psi_{t-s}$ is well defined for any $t\in \bbR$. The linear response of $\pi_\l$ is described by the following result, proved in Section \ref{tortelli}:
 \begin{Lemma}\label{apogeo} We have $a= \int _0^\infty ds  \,e^{s \cL^*}\psi_{-s}$.
\end{Lemma}
Up to now we have focused on the linear response of the marginal $\pi_\l$  at time zero  of the OSS, but there is nothing special about time zero.  In particular, writing $\pi_{\l,t}$ for the marginal at time $t$  of the OSS (i.e. $\pi_{\l,t}(x):= \bbP _{\pi_\l}(X_t =x)$), Lemma \ref{apogeo} implies the following \rosso{(we omit the proof since immediate)}:
\begin{Corollary}\label{superapogeo}
 Defining the column vector $a_t$ as $a_t(x):= \frac{\partial_{\l=0} \pi_{\l,t} (x)}{\pi(x)} $ for $x\in \cX$, we have    $a_t= \int _0^\infty ds  \,e^{s \cL^*}\psi_{t-s}$.
\end{Corollary}
By combining Theorem \ref{cor:JS}  with 
the above result, we get the linear response in the OSS  for the same functionals of Theorem \ref{cor:JS}:
\begin{Theorem}\label{alpha_omega} Consider the OSS  of the perturbed dynamics.\begin{itemize}
\item[(1)]
For  $v:\cX\to \bbR$ it  holds 
\be\label{kiriku1}
\begin{split} 
\partial_{\l=0} \bbE_{\pi_\l}  [v(X^\l_t ) ]
& = \int _0^\infty ds  \, \la e^{s \cL} v, \psi_{t-s}\ra 
 = \int_0^\infty ds \, \bbE_\pi [ v(X_s) \psi_{t-s} (X_0) ]  
\,.
\end{split} 
\en
\item[(2)]
For $v: [0,t]\times \cX \to \bbR$ measurable such that  $\int_0^t |v(s,x ) |ds <+\infty $ for all $x\in \cX$, it holds 
\be\label{kiriku2}
\begin{split} 
\partial_{\l=0} \bbE_{\pi_\l}  \Big[\int_0^t v(s, X^\l_s)ds \Big]
& = \int_0^t   du  \int _0^\infty ds  \, \la  e^{s \cL} v (u , \cdot ) , \psi_{u-s}\ra
\\ & =  \int_0^t   du  \int _0^\infty ds  \, \bbE_\pi [ v(u, X_s) \psi_{u-s} (X_0) ] 
\,. \end{split} 
\en
\item[(3)]  For  $F : D_f ([0,t] ; \cX ) \to \bbR $ additive functional of the form  \eqref{eq:F}, i.e.\
 	 $F\left(\xi_{[0,t]} \right) = \sum _{s \in (0,t]} \alpha (s,  \xi_{s-}, \xi_s)$, 
  	with $\alpha : [0,t] \times \cX \times \cX \to \bbR$ measurable and such that $\int_0^t |\a|_r (s,x)ds<+\infty$ for all $x\in \cX$. Then it holds
\be\label{kiriku3}
\begin{split}
	\partial_{\l=0} \bbE_{\pi_\l}  \Big [\sum_{s\in(0,t]}  \a(s, X^\l_{s-}, X^\l_s ) \Big ]  
	 = &  \int_0^t \bbE_\pi \big[ (\alpha g )_r (s,X_s ) \big] ds 
	\\  & +   \int_0^t ds \int _0^\infty du \, 
\bbE_\pi \Big[ \alpha_r (s, X_u ) \psi_{s-u}( X_0 ) \Big ] . 
\end{split}
\en	
\end{itemize}
\end{Theorem}
We refer to Section \ref{tortelli} for the proof of Theorem \ref{alpha_omega}.
\begin{Remark}
By  the  first formula in Theorem \ref{cor:JS}  and  the $T$--periodicity of $\psi_t$,   from \eqref{kiriku1} we get that 
  \be
  \partial_{\l=0} \bbE_{\pi_\l}  [v(X^\l_t ) ]= \lim_{n\to +\infty}  \partial_{\l=0} \bbE_{\pi}  [v(X^\l_{t+nT} ) ]\,.
  \en
  \end{Remark}

Let  $\o$ denote the frequency associated to the period $T$, i.e. $T= 2\pi /\o$. Given a $T$-periodic integrable real function $f$ we write 
\[ c_k( f):=\frac{1}{T} \int_0^{T} e^{-i k \o  t} f(t)dt\,, \qquad k\in \bbZ \,,  
\]
for its Fourier coefficients, thus leading to    $f(t)=\sum_{k\in\bbZ} c_k(f) e^{i k \o t}$.
We also write in  Fourier representation  
\[  
\psi_t (x) : = \sum_{k\in \bbZ} \hat \psi_k(x) e^{ik \o t} \,, \qquad g(t,x,y)= \sum _{k\in \bbZ} \hat g_k(x,y) e^{i k\o  t}
\]
thus leading to 
$ \hat \psi_k(x)= \sum_{y\in \cX}\left (r^*(x,y)  \hat g_k(y,x)-  r(x,y)\hat  g_k(x,y)\right)$.
For the next linear response result, 
recall also the notation \eqref{raggio} and note that $\hat \psi_k \in L^2_0(\pi)$. Indeed, as already observed, $\psi_t \in L^2_0(\pi)$ and therefore the same holds for $\hat \psi_k= \frac{1}{T} \int_0 ^T e^{- i k \o t} \psi_t dt $.

\begin{Theorem}\label{th_kobo}
Given $v:\cX \to \bbR$,  the map $t\mapsto  f_\l (t):= \bbE_{\pi_\l}  [v(X^\l_t ) ]$ is $T$-periodic in time. 
Moreover, for any $k\in \bbZ$ it holds 
\be\label{kobo3}
\begin{split} 
\partial_{\l =0}c_k ( f_\l) & =
  \int _0^\infty ds     \la  e^{s ( \cL + i k \o ) } v,  \hat \psi_k \ra   
  \\ & = \int _0^\infty ds \, e^{-ik\o s } \bbE_{\pi} [ v(X_s) \hat \psi_k (X_0) ] 
  \,. \end{split} 
\en
\end{Theorem}
\begin{proof}[ Proof of Theorem \ref{th_kobo}]
By \eqref{kiriku2}  
\be\label{lavatrice}
\partial_{\l =0}c_k ( f_\l) =\frac{1}{T} \int_0^T   dt e^{-i k \o t}  \int _0^\infty ds   \, \la e^{s \cL}  v, \psi_{t-s}\ra . 
\en
As $\psi_t$ is $T$-periodic we have 
$ \frac{1}{T} \int_0^T   dt e^{-i k \o t} \psi_{t-s}= e^{- i k \o s}\hat \psi_k$, which gives the result.
\end{proof}

In the special decoupled case  $g(s,x,y)=\t(s) E(x,y)$  the linear response formulas collected up to now admit a simplified form, we omit the proof since straightforward. 
\begin{Theorem}\label{maldive} Suppose that  $g(s,x,y)=\t(s) E(x,y)$ and let $E^*(x,y):= E(y,x)$. Then, $\psi_t(x)=\t(t)( E^*_{r^*} (x) -E_r (x) )$ and,
in the same setting of Theorems \ref{alpha_omega} and Theorem \ref{th_kobo}, formulas \eqref{kiriku1}, \eqref{kiriku2}, 
\eqref{kiriku3} and \eqref{kobo3} read: 
	\[ \begin{split} 
	\partial_{\l=0} \bbE_{\pi_\l}  [v(X^\l_t ) ]
	& = \int_0^\infty ds \, \tau (t-s) \bbE_\pi [ v(X_s) ( E^*_{r^*} (X_0) - E_r (X_0) ) ] 
	\\ 
	\partial_{\l=0} \bbE_{\pi_\l}  \Big[\int_0^t v(s, X^\l_s)ds \Big]
	& = 
	 \int_0^t   du  \int _0^\infty ds  \, \tau (u-s) \bbE_\pi [ v(u, X_s) ( E^*_{r^*} (X_0) - E_r (X_0) ) ]
	 \\ 
	 \partial_{\l=0} \bbE_{\pi_\l}  \Big [\sum_{s\in (0,t]}  \a(s, X^\l_{s-}, X^\l_s ) \Big ]  
	  & =   \int_0^t \tau (s) \, \bbE_\pi \big[ (\alpha E )_r (s,X_s ) \big] ds 
	\\  & +   \int_0^t ds \int _0^\infty du \, \tau(s-u) 
\bbE_\pi \Big[ \alpha_r (s, X_u ) ( E^*_{r^*} (X_0) - E_r (X_0) ) \Big ] 
	\\ 
	\partial_{\l =0}c_k ( f_\l)  & = \hat \tau_k
	\int_0^\infty e^{-ik\o s } \bbE_\pi \big[ v(X_s)  ( E^*_{r^*} (X_0) - E_r (X_0) ) \big] ds \, . 
	\end{split}\] 
\end{Theorem}

\section{Complex mobility matrix}\label{sec_CM}
As an example of application of the results in Section \ref{sec_OSS}, we discuss the complex mobility matrix of a random walk on a torus with heterogeneous jump rates.  To this aim, given an integer $N\geq 1$, we consider the torus $ \bbT^d_N:=\bbZ^d/N\bbZ^d$.

The unperturbed Markov jump process $(X_t)_{t\geq 0}$ is given by the random walk on $\bbT^d_N$ jumping between nearest-neighbour points with  jump rates  $r(x,y)>0$ \rosso{($r(x,y):=0$ if $x,y$ are not nearest-neighbours)}. By irreducibility, the random walk admits 
  a unique invariant distribution   $\pi $ on $\bbT^d_N$. Let $r^*(x,y)$ be the time-reversed jump rates, i.e. $r^*(x,y)= \pi(y) r(y,x)/\pi (x)$. A special case is given by the \emph{reversible random walk}  on the torus, for which 
 $  r^*(x,y)=r(x,y)$. For example, if $r(x,y)=r(y,x)$ for all $x,y$,  then  $\pi$ is the uniform distribution and $r^*(x,y)=r(x,y)$.

We introduce a time-oscillatory field along the direction of a fixed  unit vector $v\in \bbR^d$. Given $\l>0$ and $\o\in \bbR\setminus \{0\}$, the perturbed random walk $(X^\l_t)_{t\geq 0}$ has jump rates at time $t$ given by 
  \begin{equation}\label{pert_rates}
r_t^\l(x,y)\rosso{:=} \exp \{ \l \cos(\o t ) (y-x) \cdot v\} \, r(x,y)
  \end{equation}
  \rosso{for $x,y$  nearest-neighbours  ($r_t^\l(x,y):=0$ otherwise)}.
 Above $w\cdot v$ denotes the Euclidean scalar product of the vectors $v,w$.   As before, we write $\pi_\l$ for the initial distribution of the OSS. 
 Note that the perturbation is of decoupled form $g(s,x,y)= \t(s) E(x,y)$ with   $\t(s)= \cos(\o s)$ and $E(x,y)= (y-x) \cdot v$   \rosso{for $x,y$  nearest-neighbours   and $E(x,y)=0$ otherwise}. Setting 
\be\label{kinder_cards} \Psi(x):=- \sum _{e:|e|=1} (  r^*(x,x+e)+ r(x,x+e))  e\in \bbR^d\,,
 \en
 we have (recall that $E^*(x,y):= E(y,x)$)
 \be\label{salazar}  E^*_{r^*} (x) -E_r (x)=\Psi(x) \cdot v\,.
 \en
 Note that
 $  \Psi(x)=-2  \sum _{e:|e|=1} r(x,x+e)  e  $  for the reversible random walk. As an  immediate consequence of Theorem \ref{maldive} we get that, 
   for any function $f:\bbT^d_N \to \bbR$,  
   \be\label{telefono}
   \begin{split}
   \partial _{\l=0}  \bbE_{\pi_\l}  [ f(X_t^\l)]
   & =\int_0 ^\infty \cos ( \o(t-s) ) \la e^{s\cL }f ,\Psi \cdot v\ra ds \\
    &= \Re\Big(  \int_0 ^\infty e^{i  \o t }
    \la f ,e^{- (i \o -\cL ^*) s} ( \Psi\cdot v ) \ra ds   \Big)\\
    & = \Re\Big( e^{i \o t} \la f , (i \o -\cL^*)^{-1}( \Psi\cdot v)  \ra\Big)\,.
     \end{split}
    \en
For the above formula,    recall  \eqref{linz}, that the above integrands decay exponentially fast in $s$  and that $\Re(z)$ denotes the real part of the complex number $z$. 
Calling $(Y_t^\l)_{t\geq 0}$  the random walk obtained by lifting to $\bbZ^d$ the original one $(X_t^\l)_{t\geq 0}$, we get that the  \emph{mean instantaneous 
velocity} in the OSS at time $t$ 
is given by  
\be\label{matteo105}
V_\l (t):=\frac{d}{dt} \bbE_{\pi_\l} \bigl[ Y_t^\l ] = \sum_{e:|e|=1 } \bbE_{\pi_\l}\bigl[ r^\l_t (X^{\l}_t,X^\l_t+e) \bigr] e \,.
\en
In what follows we denote by $e_1, e_2, \dots, e_d$ the canonical basis of $\bbR^d$.
Moreover, we let  $c, \g : \rosso{\bbT^d_N} \to \bbR^d$ be defined as 
\begin{align}
& c(x):=\sum_{j=1}^d \bigl[ r(x,x+e_j)+ r(x, x-e_j) \bigr] e_j\,, \label{cocinella1}\\
& \g(x) := \sum_{j=1}^d \bigl[ r(x,x+e_j)- r(x, x-e_j) \bigr] e_j= \sum_{e:|e|=1} r(x,x+e) e\,. \label{cocinella2}
\end{align}
\begin{Theorem}\label{teo_CM} \rosso{Given $\o\not =0$} it holds 
 \begin{equation}\label{stilton}
\partial_{\l=0} V_\l (t)   = \Re \left( e^{i \o t} \s  (\o) v \right) \,,
\end{equation}
where   the {\bf complex mobility matrix} $\s(\o)=\bigl( \s_{j,k}(\o) \bigr)$  is the $d\times d$ matrix with complex entries  given by
\begin{equation}\label{jabba1}
\begin{split}
\s_{j, k} (\o)& =  \pi [c_j]\d_{j,k}+ \la \g_j, (i \o -\cL^*)^{-1} \Psi_k \ra \\
&=  \pi [c_j]\d_{j,k}+ \int_0^{+\infty} \la \g_j, e^{ -(i \o - \cL^*) s} \Psi_k \ra\,.\end{split}
\end{equation}
 For  the reversible random walk it holds $\Psi(x) = - 2 \g(x)$ and $\cL^*=\cL$, thus implying that
 $\s(\o)$ is symmetric and 
 \begin{equation}\label{jabba2}
\begin{split}
\s_{j, k} (\o)& =  \pi [c_j]\d_{j,k}-2 \la \g_j, (i \o -\cL)^{-1} \g_k \ra \\
&=  \pi [c_j]\d_{j,k}-2 \int_0^{+\infty} \la \g_j, e^{ -(i \o - \cL) s} \g_k \ra\,.\end{split}
\end{equation}
 \end{Theorem}
The proof of the above theorem is given in Section \ref{sec_dim_teo_CM}.

 \begin{Remark} Given   $d\times d$ complex matrices $A,B$ with 
$ \Re ( e^{i \o t} A v)=\Re ( e^{i \o t} B v)$ for all  $t \geq 0$ and $ v\in \bbR^d$,
 then necessarily $A=B$, since it must be $\cos(\o t) \Re(A-B) v=0$ and $\sin (\o t) \Im (A -B) v=0$ for all  $t \geq 0$ and $ v\in \bbR^d$.
  In particular, the validity of the identity \eqref{stilton}  for all $t,v$ univocally determines $\s(\o)$.
 \end{Remark}
In Section \ref{es_CM} we will compute $\s(\o)$ explicitly in particular cases. 
 When the system is very  heterogenous, $\s(\o)$ cannot be computed explicitly. Formulas \eqref{jabba1} and \eqref{jabba2} in Theorem \ref{teo_CM} are nevertheless  useful for investigating the properties of $\s(\o)$ (cf.~\cite{FM})  and also for proving homogenization  of $\s(\o)$ as $N\to +\infty$   in the case of random unperturbed jump rates (cf.~\cite{FS}).

\subsection{Extension to more general jump rates}
We can introduce and analyze the complex mobility matrix also when the unperturbed process has long jumps. We describe below how to modify the above  discussion in the general case.

We fix a finite set $\cZ\subset \bbZ^d$ such that the canonical projection $\pi:\bbZ^d \to  
\bbT^d_N=\bbZ^d/N \bbZ^d$ is injective when restricted to $\cZ$. Given $x\in \bbT^d$ and $z\in \cZ$ we write $x+z$ for the site $x+\pi(z)$ in  $\bbT^d_N$. The above sum is understood in the additive quotient group $\bbT^d_N=\bbZ^d/N \bbZ^d$ (also before we wrote $x+e$ for $x+\pi(e)$). Since $\pi$ restricted to $\cZ$ is injective, the sites $x+z$ with $z\in \cZ$ are all distinct. 

We assume now that the unperturbed Markov jump process $(X_t)_{t\geq 0}$ is an irreducible Markov chain (random walk) on $\bbT^d_N$ with jump rates $r(x,y)$ and that $r(x,y)=0$ if $y\not\in\{x+z\,:\, z\in \cZ\}$. Fixed a unit vector $v\in \bbR^d $ we define the perturbed jump rates as 
\be\label{roseto}
r_t^\l(x,y):=\begin{cases} \exp \{ \l \cos(\o t ) (z \cdot v)\} \, r(x,x+z) &\text{ if } y=x+z \text{ for some } z\in \cZ\,,\\
0 & \text{ otherwise}\,.
\end{cases}
\en
Since $\pi$ is injective on $\cZ$, if $y=x+z$ for some $z\in \cZ$ then $z$ in univocally determined, thus assuring the the above definition is well posed.
 Due to \eqref{roseto}  $g(s,x,y)= \t(s) E(x,y)$ with   $\t(s)= \cos(\o s)$ and 
  \[E(x,y)=\begin{cases}  z \cdot v 
  &\text{ if } y=x+z \text{ for some } z\in \cZ\,,\\
0 & \text{ otherwise}\,.
\end{cases}
\]
   In place of \eqref{kinder_cards} we now set 
  $\Psi(x):=-\sum_{z\in \cZ}\left( r^*(x,x+z)+r(x,x+z)\right) z$.
  Then \eqref{salazar} remains valid. Indeed, for $z\in \cZ$ we have   $E(x+z,z)=-z\cdot v$ and therefore
  $E_{r^*}^*(x)=\sum_{y\in \bbT^d_N} r^*(x,y) E(y,x)=-\sum_{z\in \cZ} r^*(x,x+z)z\cdot v$.
  Due to \eqref{salazar}   equation \eqref{telefono} is still valid.

We now introduce the lifted random walk $Y^\l:=(Y^\l_t)_{t\geq 0}$ by requiring that
$\pi(Y^\l_t)=X^\l_t$ and that $Y^\l $ makes a jump  $z$ whenever $X^\l$ jumps from $x$ to $x+z$ for some $z\in \cZ$ (due to our assumptions, a.s. the jumps of $X^\l$   belongs to  $\cZ$). As $Y^\l_0$ we can take any point in $\pi^{-1}( X^\l_0)$.   The mean instantaneous velocity is now given by
\be\label{matteo106}
V_\l (t):=\frac{d}{dt} \bbE_{\pi_\l} \bigl[ Y_t^\l ] = \sum_{z\in \cZ } \bbE_{\pi_\l}\bigl[ r^\l_t (X^{\l}_t, X^{\l}_t+z) \bigr] z\,.
\en
\begin{Theorem}\label{teo_CM_esteso} 
Fix  $\o\not =0$  and set 
$\g(x) := \sum_{z\in \cZ}  r(x,x+z) z$. Then  the content of Theorem \ref{teo_CM_esteso} remains valid by replacing the term 
$\pi [c_j]\d_{j,k} $ in \eqref{jabba1} and \eqref{jabba2} with $ \sum_{z\in \cZ} \pi[r(\cdot, \cdot+z)]z_j z_k$.
\end{Theorem}
The proof is a slight modification of the proof of Theorem \ref{teo_CM} and it is sketched  in Section \ref{sec_dim_teo_CM_esteso}.

%
%

\section{Examples}\label{sec_esempi}
In this section we present some applications of the  theoretical results  developed so far.
%
%

\subsection{Random walks on $\bbZ^d$  with confining potential and external field}
Below, given sites $y,z\in \bbZ^d$, we write $y\sim z$ if $|y-z|=1$.  \subsubsection{Unperturbed random walk}
As unperturbed process we take the 
 nearest--neighbour random walk $(X_t)_{t\geq 0}$ on $\bbZ^d$ with transition rates given by 
	\be\label{tassometro} r(y,z) = \exp \Big\{ -\frac 12 (V(z) - V(y)) +\frac 12 f(y,z) \Big\} \,,
	\qquad y\sim z\,,\en
for $V$  potential function.
 We assume that 
\be\label{trattore}\lim_{|y|\to +\infty} V(y)=+\infty  \;\;\text{ and }\;\; \| f \|_\infty <\infty\,.
\en

At cost of including the inverse temperature $\b$ in $V$ and $f$, we take $\b=1$. If the above rates come from a local detailed balance  then it must be $\frac{r(y,z)}{r(z,y)}=e^{-\D H(y,z)}$, where $\D H(y,z)$ is the energetic variation in a transition from $y$ to $z$. In this case, due to \eqref{tassometro}, we have
$ \D H(y,z)=  (V(z) - V(y)) + \frac 12 [f(z,y) -f(y,z)]$ for $ y \sim z$.
It is then natural to think of $f(y,z)$ as the work done by an external field on the particle during the transition from $y$ to $z$ and therefore to take
$f(y,z)=-f(z,y)$, thus leading to 
\be\label{LDB}
\D H(y,z)=  (V(z) - V(y)) -   f(y,z) \,, \qquad y \sim z \,.
\en
 The special case of a spatially uniform external field equal to $v\in \bbR^d$  (in addition to the conservative field associated to $V$) can  be described by taking $f(y,z)=  v\cdot (z-y)$, or equivalently by changing the potential $V(y)$ into $V(y)- v\cdot y$. In general, one can  include into $V$ the effect of all potential fields.

The factor $e^{- \frac 12 f(y,z) }$ in the rate $r(y,z)$ can also be due to a microscopic energetic barrier (as in the random barrier model) and in this case it is natural to have $f(y,z)=f(z,y)$. Of course, we can take $f\equiv 0$ as well.

%

\smallskip

Following \cite[Section~10.5]{BFG2}, when     $V\in C^1(\bbR^d) $,  we say that 
$V$ has \emph{diverging radial variation which dominates the transversal variation} if, by orthogonally decomposing $\nabla V(y)$ with $y\not =0$  as   
\[ \nabla V(y)=\la \nabla V(y), \hat{y} \ra \hat{y} + W(y) \text{ with }  \hat y:= y/|y|\,,
\] it holds
\be\label{radio_italia}
\lim_{|y|\to+\infty} \la \nabla V(y), \hat y \ra =+\infty\; \text{ and }\;|W(y)| \leq \frac{\a}{\sqrt{d}} \la \nabla V(y), \hat y\ra + C
\en
for $\a\in [0,1)$ and $C\geq 0$. Note that \eqref{radio_italia} implies that  $\lim_{|y|\to +\infty} V(y)=+\infty$

We recall some results for the unperturbed random walk obtained (sometimes implicitly)  in \cite{BFG2}:
\begin{Proposition}\cite{BFG2}   \label{kalush}
The following hold:
\begin{itemize}
\item[(i)] The unperturbed random walk  does not explode almost surely for any starting point.
\item[(ii)]  If $f(y,z)=f(z,y)$ for all $y\sim z$ and if $Z:= \sum_{ y \in \bbZ^d} e^{-V(y)} < \infty$, then  the  unperturbed random walk  is reversible with respect to the stationary distribution $\pi (x) = e^{-V(x)} /Z$.
\item[(iii)]  The unperturbed random walk admits a stationary distribution  if 
\be\label{cond_190}  \lim_{|y|\to +\infty} -\frac{LU}{U}(y)=+\infty\,, \qquad U(y):= e^{V(y)/2}\,.
\en
\item[(iv)] Setting  $r_0(y,z): = \exp \{ -\frac 12 (V(z) - V(y)) \}$,
the above condition \eqref{cond_190} is satisfied 
if $\hat r_0 (y):= \sum_{z:z\sim y} r_0(y,z) \to +\infty $ as $|y|\to\infty$, and this  in turn holds whenever
$V\in C^1(\bbR^d)$ has diverging radial variation which dominates the transversal variation. 
\end{itemize}
\end{Proposition}
We  refer the interested reader to \cite[Section~10.5]{BFG2} for a class of external forces $f$ for which the stationary distribution exists and is given by $\pi (x) = e^{-V(x)} /Z$.
\begin{proof}[Proof of Proposition \ref{kalush}]
Non--explosion in Item (i) is guaranteed by the existence of a diverging non-negative function $U$ on $\bbZ^d$ satisfying \eqref{grano}. As discussed in \cite[Section~10.5]{BFG2},  this can be taken to be $U(y) = e^{V(y)/2}$, to find that
	\be\label{nielsen}
	\begin{split}
	\frac{LU}{U}(y) & = \sum_{z:z\sim y} \big( e^{\frac{V(z)-V(y)}{2}} -1\big) r(y,z)
	 \\
	& = \sum_{z:z\sim y}\left( 1- e^{-\frac 12 (V(z) - V(y) }\right)e^{\frac 12 f(y,z)}\leq 2d \, e^{\frac{\| f\|_\infty }{2}}\qquad \forall y \in \bbZ^d\,.
	\end{split}
	\en
	To prove  Item (ii) one easily checks detailed balance.  To prove Item (iii), by \cite[Proposition~4.1]{BFG2}, it is enough to show that \eqref{cond_190} implies Condition $C(\s)$ with $\s=0$  defined in \cite[Section~3]{BFG2}.
By taking $u_n:=U$ there, this condition $C(0)$ reduces to the following: (a) $\sum_{z:z\sim y} r(y,z) U(z)<+\infty$ for all $y$; (b) $U$ is bounded from below by a positive constant; 
(c) $\lim_{|y|\to +\infty} W(y)=+\infty$ where
 $W(y):=-LU(y)/U(y)$; (d) $W$ is bounded from below. We note that  (a) is trivially satisfied; (b) is valid as $U= e^{V/2}$ and  $\lim_{|y|\to +\infty} V(y)=+\infty$;
 (d) follows from (c), and (c) corresponds to \eqref{cond_190}.

Finally,  Item (iv) follows from the observations contained in the proof of  \cite[Lemma~10.3]{BFG2}.  For the reader's convenience we just point out that the first part follows from the estimate
\be\label{castagnole}
 -\frac{LU}{U}(y)= \sum_{z:z\sim y} r(y,z) -  \sum_{z:z\sim y}   e^{\frac{1}{2}f(y,z)}  \geq \hat r_0(y) e^{-\frac{1}{2} \|f\|_\infty} - 2d e^{\frac{1}{2} \|f\|_\infty}\,.
\en
We point out  that the derivation of the second part of Item (iv)  in the proof of \cite[Lemma~10.3]{BFG2} does not use that $\sum_{y\in \bbZ^d} e^{-V(y)}<+\infty$ as assumed at the beginning of Section~10.5 in \cite{BFG2} .\end{proof}

In the  case $d=1$ we can say more.  Indeed, writing $m(y)=e^{-V(y)} \phi(y)$, 
the measure $m(y)$ satisfies  detailed balance if and only if 
\[
\phi(y)e^{\frac{1}{2} f(y,y+1)}= \phi(y+1)e^{ \frac{1}{2} f(y+1,y)}\, 
\qquad \forall y \in \bbZ\,,
\]
which means $\phi(y)= \phi(0) c(y) $ for all $y\in\bbZ$, where 
\be\label{rainbow100}
c(y):= \begin{cases}
\prod_{j=0}^{y-1} 
e^{ \frac{1}{2} \bigl( f(j,j+1)-f(j+1,j)\bigr) }& \text{ if } y\geq 1\,,\\
\prod_{j=y}^{-1} 
e^{ \frac{1}{2} \bigl( f(j+1,j)-f(j,j+1)\bigr) }& \text{ if } y \leq -1\,.
\end{cases}
\en
As an immediate consequence we have:
\begin{Proposition}\label{kalushbis} For $d=1$  the unperturbed random walk admits a reversible distribution $\pi$  if and only if $\cZ:=\sum_{y\in \bbZ} e^{-V(y)} c(y)<+\infty$. In this case we have  $\pi(y) = e^{-V(y)}c(y)/\cZ$. In particular reversibility  takes place in the following cases: (i) $f(y,z)=f(z,y)$ for all $y\sim z$ and $\sum_{y\in \bbZ} e^{-V(y)} <+\infty$, (ii) $f$ is only non--zero on a finite family of edges and $\sum_{y\in \bbZ} e^{-V(y)} <+\infty$, (iii)  
$\sum_{y\in \bbZ} e^{-V(y)+  \|f\|_\infty |y| } <+\infty$.
\end{Proposition}
\subsubsection{Perturbed random walk}
For the perturbed process we fix $\l>0$ and a bounded and measurable function $g : [0,t] \times \bbZ^d \times \bbZ^d \to \bbR$, and set $r^\l (s,y,z) := e^{\l g(s,y,z)} r(y,z) $ for all $s \in [0,t]$ and neighbouring vertices $y\sim z$.  

We isolate the following technical result for later applications:
\begin{Lemma}\label{grandepuffo} Let $\a:  [0,t] \times \bbZ^d \times \bbZ^d \to \bbR$ be bounded and measurable (for example, $\a=g$). Then $\a$ satisfies Condition $C[\nu,t]$ with parameter $\theta:=\| \a\|_\infty^{-1} e^{- \| f\|_\infty}$ for $\nu =\d_x$ and any $x\in \bbZ^d$, and in general for any distribution   $\nu$ such that $\nu[e^{V/2}]<+\infty$.
\end{Lemma}
\begin{proof}
By Lemma \ref{tik_tok}, in order to guarantee that a function $\alpha$ satisfies Condition $C[\nu,t]$ with parameter $\theta>0$ it suffices to find a positive function $U: \bbZ^d \to \bbR$, bounded away from zero, such that $LU \leq CU - \theta |\a |_r U $ for some $C>0$ and such that $\nu[U]<+\infty$. Again we take $U:= e^{V/2} $. Since $\lim_{|y|\to+\infty}V(y) =+\infty$, $U$ is bounded away from zero. By \eqref{castagnole} 
	$(LU/U)(y) \leq 2d \, e^{\| f \|_\infty /2} - \hat r_0(y) e^{-\| f \|_\infty /2}$,
while 
\[ |\a |_r (s,y) := \sum_{z:z\sim y } |\a (s,y,z) | r(y,z) \leq \| \a \|_\infty \hat r_0(y) e^{\| f\|_\infty /2}\,.
\] It thus suffices to take $\theta :=\| \a\|_\infty^{-1} e^{- \| f\|_\infty} $ to have that $LU /U\leq C - \theta |\a |_r $ 
for some $C>0$.  \end{proof}

The application of  Lemma \ref{grandepuffo} above is twofold. Firstly, one can take $\a=g$ (since $g$ is bounded), to get that $g$ satisfies Condition $C[\nu, t]$ for $\nu$ as in the lemma. This, by Theorem \ref{th:explosionX}, automatically implies non-explosion of the perturbed process for $\l < 1/\left(8 \| g\|_\infty  e^{ \| f\|_\infty}\right)$,  as well as the linear response results stated in Theorems \ref{th:JS} and \ref{cor:JS}. Secondly, one can apply Lemma \ref{grandepuffo} to a bounded function $\a$ entering in the definition of the additive functional \eqref{eq:F}, to get that $\a$ satisfies Condition $C[\nu,t]$ and therefore the quantities in \eqref{aa} belong to $\bbL^p(\bbP_\nu)$.  For example,  by Lemma \ref{grandepuffo}, if $\a$ and $v$ are bounded  then  \eqref{res1}, \eqref{res2}, \eqref{res3} hold for $\nu=\d_x$ with $x\in \cX$ and in general for any initial distribution $\nu$ with  $\nu[e^{V/2}]<+\infty$.



We conclude this section by discussing  an application of Theorem \ref{cor:JS} on 
 linear response starting from the unperturbed stationary distribution  $\pi$. We consider jump rates defined in terms of a local detailed balance.  As for \eqref{LDB} we consider $g(s,\cdot,\cdot)$ antisymmetric, i.e.~$g(s,x,y)=-g(s,y,x)$.  Then we focus on the work functional  $F(X_{[0,t]})$, given by the work done by all forces (also the time-dependent ones producing the perturbation). We have	\be\label{lavoro} F(X_{[0,t]}) := -V(X_t ) + V(X_0) + \sum_{s\in (0,t] } f(X_{s-} , X_s ) + 2\l \sum_{s\in (0,t]} g(s,X_{s-} , X_s ) \,.
	\en
\begin{Proposition} Suppose that the unperturbed process has a stationary distribution $\pi$, from which it is started (see Propositions \ref{kalush} and \ref{kalushbis} for sufficient conditions).  Suppose that  $f$ and $g$ satisfy Condition $C[\pi,t]$ and that $V\in L^p(\pi)$ for some $p\in (1,+\infty)$ (by Lemma \ref{grandepuffo} and since $V\to+\infty$ 
it suffices to require $\pi[e^{V/2}]<+\infty$, which reads $\sum_{y\in \bbZ^d} e^{-V(y)/2}<+\infty$
in the case of  zero external force $f \equiv 0$). Then 
\[\begin{split}  \partial_{\l =0} \bbE_\pi [ F(X_{[0,t]}^\l ) ]  = & 
	-\int_0^t ds \, \bbE_\pi \bigl[ V(X_s) \psi_{t-s} (X_0 ) ] 
	+ \int_0^t \bbE_\pi \big[ (f g )_r (s,X_s ) \big] ds  
	\\ & + \int_0^t ds \int_0^s du \, 
	\bbE_\pi \Big[ f_r (s, X_s ) \psi_{s-u}(X_{s-u} )  \Big] 
	+ \int_0^t \bbE_\pi [ g_r (s,X_s ) ] ds \,,
	\end{split} 
	\]
with $\psi_t(x)$ defined as  in \eqref{def:psi}.
\end{Proposition}
\begin{proof}
By linearity, we  have 
	\[ \partial_{\l =0} \bbE_\pi [ F(X_{[0,t]}^\l ) ] = \partial_{\l =0} \bbE_\pi [ F_1(X_{[0,t]}^\l ) ]
	+ \partial_{\l =0} \bbE_\pi [ F_2(X_{[0,t]}^\l ) ] + \partial_{\l =0} \big( 2\l \bbE_\pi [ F_3 (X_{[0,t]}^\l ) ] \big) ,
	\]
where 
	\[F_1(\xi_{[0,t]} ) := -V(\xi_t ) , \quad \; 
	F_2 (\xi_{[0,t]} ) := \sum_{s\in (0,t] } f(\xi_{s-} , \xi_s ) , 
	\quad \;
	F_3 (\xi_{[0,t]} ) := \sum_{s\in (0,t]} g(s,\xi_{s-} , \xi_s )  \,.\]
Note that, referring to the beginning of Section \ref{sec:types},  $F_1$ is a functional of type (1), while $F_2$ and $F_3$ are functionals of type (3) with $f$ and $g$ bounded.

If the bounded functions $f$ and $g$ satisfy Condition $C[\pi,t]$ and  $V\in L^p(\pi)$ for some $p\in (1,+\infty)$  (which, by stationarity, is equivalent to  $V(X_t ) \in L^p( \bbP_\pi )$), then  the assumptions of Theorem \ref{cor:JS} are satisfied. We point out that  we have excluded apriori the case  $p=+\infty$ since $V(y) \to+\infty $ as $|y|\to\infty$, and therefore it cannot be  $V\in L^\infty(\pi)$.
We observe that  $\pi[e^{V/2}]<+\infty$ and the boundedness of $f$ and $g$  imply that $f$ and $g$ satisfy Condition $C[\pi,t]$ by Lemma \ref{grandepuffo}. If $\pi[e^{V/2}]<+\infty$, then trivially we also have $V\in L^p(\pi)$ for any $p\in (1,+\infty)$.  If $f\equiv 0$, then $\pi(y)=e^{-V(y)}/Z$ where $Z:=\sum _{y} e^{-V(y)}<+\infty$ (see Proposition~\ref{kalush}--(ii)). On the other hand,  the condition $Z<+\infty$ is trivially satisfied if  $\sum_{y\in \bbZ^d} e^{-V(y)/2}<+\infty$ as $V$ is a diverging function. In particular, for $f\equiv 0$ and under the assumption $\sum_{y\in \bbZ^d} e^{-V(y)/2}<+\infty$, we get  $\pi[e^{V/2}]=\sum_{y\in \bbZ^d} e^{-V(y)/2}<+\infty$.

By applying Theorem \ref{cor:JS} we then get 
\begin{align*}
&   \partial_{\l=0} \bbE_\pi \bigl[ F_1 ( X^\l_t) \bigr]  = -
	\int_0^t ds \, \bbE_\pi \bigl[ V(X_s) \psi_{t-s} (X_0 ) ]
 , \\
 & \partial_{\l =0} \bbE_\pi [ F_2(X_{[0,t]}^\l ) ]
	= \int_0^t \bbE_\pi \big[ (f g )_r (s,X_s ) \big] ds  
	+ \int_0^t ds \int_0^s du \, 
	\bbE_\pi \big[ f_r (s, X_s ) \psi_{s-u}(X_{s-u} )  \big] , 
	\\
&  \partial_{\l =0} \big( 2\l \bbE_\pi [ F_3 (X_{[0,t]}^\l ) ] \big) 
	= 2\lim_{\l \to 0} \bbE_\pi [ F_3 (X_{[0,t]}^\l ) ] =2 \bbE_\pi [ F_3 (X_{[0,t]} ) ] 
	=2 \int_0^t \bbE_\pi [ g_r (s,X_s )] ds.
	\end{align*}
In the last line, the second equality follows from \eqref{titans2X} in Proposition \ref{derivoX}, and in the third equality we have used that $G_t(X_{[0,t]})$ introduced in \eqref{ginoX} defines a martingale, as anticipated in Remark \ref{rem:Gmart}. Putting all together we get our claim.\end{proof}

\begin{Remark}
We point out that the non-explosion of the perturbed chain could have alternatively been derived using Theorem 1 in \cite{CZ} with Lyapunov function $e^{V/2}$, with  computations similar to the one in \eqref{nielsen}, under the assumption that the perturbation $g$ is continuous  in time.
\end{Remark}

 \subsection{Birth and death processes}  
 Consider a birth and death process on the set of non-negative integers $\bbN$, that is a Markov jump process  (in particular, a continuous-time Markov chain) $(X_t)_{t\geq 0}$  with transition rates 
 	\[ r(0,1) = r_0^+ >0  \qquad 
 	r(k, k\pm 1 ) = r_k ^\pm >0 \]
 and $r(k,j) = 0$ otherwise (for later use we set $r_0^-:=0$).
This can of course be seen as a particular instance of a random walk in confining potential with external field, and thus analyzed as in the previous section. We take here a different approach.

It is known (cf.~\cite[Corollary~3.18]{Chen}, \cite[Remark~4]{CZ}) that the the unperturbed process a.s. does not explode if and only if 
\be\label{felice1} \sum_{k=0}^\infty \g_k =+\infty\,, \;\; \text{ with } \;\; \g_k:=\frac{1}{r_k^+}+ \frac{r_k^-}{r_k^+}\cdot \frac{1}{r_{k-1}^+}+\cdots +  \frac{r_k^-}{r_k^+} \cdot \frac{r_{k-1}^-}{r_{k-1}^+}\cdots \frac{r_1^-}{r_1^+} \cdot \frac{1}{r_{0}^+}\,.
\en
Hence, we assume \eqref{felice1} to be satisfied.
If in addition 
\begin{equation}\label{felice2}
 	 Z := 1 + \sum_{k\geq 1} \frac{ r_0^+ r_1 ^+ \cdots r_{k-1}^+}{r_1^- r_2 ^- \cdots r_k^-} < +\infty\,,  	\end{equation}
	 then the unperturbed process admits a invariant distribution $\pi$, which is unique, reversible and given by 
\begin{equation}\label{stationaryBDbis}
	 \pi (0) = \frac 1Z , \qquad \pi (k) = \frac 1Z \frac {r_0^+ r_1 ^+ \cdots r_{k-1}^+}{r_1^- r_2 ^- \cdots r_k^-} , \quad k\geq 1
	 \end{equation}
	  (this statement can be verified by simple computations). 
%
%
%
%
%
%
%
Note that when $r_k^+ = r^+$ for all $k\geq 0$ and $r_k^- = r^-$ for all $k \geq 1$ then \eqref{felice1} is always satisfied, while \eqref{felice2}  reduces to $r^-> r^+$, and $\pi (k)$ is proportional to $(r^+/r^-)^k$ for all $k\geq 0$. \\

For the perturbation fix $\l >0$ and a bounded measurable function $g: [0,t ] \times \bbN \times \bbN \to \bbR$, and set 
	\[ r^\l_t  (k, k\pm 1 ) = e^{\l g (t , k , k\pm 1)} r_k^{\pm} . \]
Then, if $\nu$ is a probability distribution on $\bbN$ and $t>0$, the function $g$ satisfies Condition $C[\nu , t]$ in Definition \ref{def_papaya} if and only if for some $\theta >0$ 
	\begin{equation}\label{BDC}
	 \bbE_\nu \Big[ \exp \Big\{ \theta \int_0^t |g(s, X_s , X_{s} +1)| r_{X_s}^+ ds 
	+ \theta \int_0^t |g(s, X_s , X_{s} -1)| r_{X_s}^- ds  \Big\} \Big] < \infty . 
	\end{equation}
If the above condition is satisfied then the perturbed process $X^\l$ almost surely  does not explode in $[0,t]$ for $\l$ small  by Theorem \ref{th:explosionX}, and the linear response results described in Theorems \ref{th:JS} and \ref{cor:JS} hold.

Note that, since $g$ is bounded, \eqref{BDC} trivially holds if, writing $\hat r_k = r_k^+ + r_k^-$, the collection $(\hat {r}_k)_{k\geq 0}$ is uniformly bounded. 
If, on the other hand,  $\sup_{k\geq 0} \hat r_k =+\infty$ then again \eqref{BDC} trivially holds if $g$ is only non-zero on a finite number of edges (i.e.\ if the perturbation is finitely supported). We now discuss sufficient conditions for \eqref{BDC} to hold in the general case $\sup_{k\geq 0} \hat r_k =+\infty$ and $g$ non-zero on infinitely many edges. 
To this aim we first observe that, if $\limsup_{k\to\infty } r_k^+ / r_k^- < 1$, then both \eqref{felice1} and \eqref{felice2} are satisfied. In particular, the unperturbed system a.s.\ does not explode and it admits the invariant distribution $\pi$.
\begin{Lemma} \label{LeBD}
Assume that $\limsup_{k\to\infty } r_k^+ / r_k^- < 1$. Let $\a:[0,t]\times \bbN\times \bbN\to\bbR$ be measurable and bounded (e.g.~take $\a=g$). Then, for any $B>1$,  there exists  $\theta>0$ such that  $\a$  satisfies Condition $C[\nu,t]$ with parameter $\theta$ for any distribution $\nu$  satisfying $\nu[ W]<+\infty$ where $W(k):=B^k$. In particular, $\a$ satisfies Condition $C[\d_x,t]$   with the same  parameter $\theta$ for all $x\in \bbZ^d$.
\end{Lemma}
\begin{proof} 
Recalling Lemma \ref{tik_tok}, to guarantee \eqref{BDC} it suffices to find a positive  function $U: \bbN \to \bbR$, strictly bounded away from zero, such that  $\nu[U]<+\infty$ and such that $LU \leq CU - \theta |\a|_r U$ for some  $C,\theta>0$. 
The last property holds provided
	\begin{equation}\label{BDC2} 
	\Big( \frac{U(k+1)}{U(k)} -1 + \theta \|\a \|_\infty  \Big) r_k^+ 
	+ \Big( \frac{U(k-1)}{U(k)} -1 + \theta \|\a \|_\infty \Big) r_k^- \leq C \qquad \forall k \in \bbN\,.
	\end{equation}
Under the assumption that $ \limsup_{k\to\infty } r_k^+ / r_k^- < 1$,  there exists $\gamma < 1$ such that $r_k^+ \leq \gamma r_k^- $ for all $k$ sufficiently large. Set $U(k) := A^k$ for $A\in(1,B]$ to be chosen later. Then, by taking $\nu$  with $\nu[W]<+\infty$, we have $\nu[U]<+\infty$. Moreover
the inequality \eqref{BDC2} reads 
	\[ \Big( A -1 + \theta \|\a \|_\infty  \Big) r_k^+ 
	+ \Big(\frac 1 A -1 + \theta \|\a \|_\infty \Big) r_k^- \leq C \qquad \forall k \in \bbN
	\,.	\]
Using that $r_k^+ \leq \gamma r_k^-$ we see that the left hand side is bounded by 
$ ( \gamma (A-1) + \theta (\gamma +1 ) \| \a \|_\infty + 1/A -1 )r_k^- $ for $k$ large enough, so  \eqref{BDC2} holds provided 
	\[ \gamma (A-1) + \theta (\gamma +1 ) \| \a \|_\infty + 1/A -1 \leq 0 . \]
Writing $A = 1+\varepsilon$, and multiplying both members by $(1+\e)/\e$,  it can be easily checked that 
the last expression is equivalent to 
\be\label{tonno_cipolla} \g  (1+\e) + \theta (\gamma +1 ) \| \a \|_\infty (1+\e)/\e  \leq 1\,.  \en
As $\g<1$ we can take $\e$ small to have $A=1+\e\leq B$ and $\g(1+\e)<1$, afterwards we  can take $\theta $ small to ensure \eqref{tonno_cipolla}. This proves the first part of the lemma, while the last statement follows immediately from the first part.
\end{proof} 

We conclude by discussing linear response formulas when starting from the stationary distribution $\pi$ defined in \eqref{stationaryBDbis} assuming both \eqref{felice1} and \eqref{felice2}.  We suppose that $g$ satisfies condition $C[\pi,t]$. For example, according to Lemma 
\ref{LeBD} and due to the explicit form \eqref{stationaryBDbis} of $\pi$, $g$ satisfies condition $C[\pi,t]$ if   $\g:=\limsup_{k\to\infty } r_k^+ / r_k^- < 1$  and 
\[
\sum_{k=1}^\infty \frac { r_1 ^+}{r_1^-} \cdot \frac{r_2^+}{r_2^-} \cdots  \frac{r_{k-1}^+}{r_{k-1}^-}\cdot \frac{B^k}{r_k^-}<+\infty
\]
for some $B>1$. As $\g<1$, it is enough that  $\sum_{k=1}^\infty  {\tilde\g}^k / r_k^-<+\infty$ for some $ \tilde \g\in (\g,1)$.

Note that, since the unperturbed dynamics is reversible with respect to the stationary distribution $\pi$, then $r^*(k, k\pm 1) = r(k, k\pm 1) = r_k^\pm$. It thus follows from Theorem \ref{cor:JS} that if $v : \bbN\to\bbR $ is a measurable function with $v(X_t ) \in L^p (\bbP_\pi )$  (i.e.\ $v\in L^p(\pi)$) for some $p>1$, then 
\[
 \partial_{\l=0} \bbE_\pi \bigl[ v ( X^\l_t) \bigr]  = 
 \int_0^t ds \, \bbE_\pi \bigl[ v(X_s) \psi_{t-s} (X_0 ) ]
  \]
with $\psi_{t-s}(X_0)$ defined as in \eqref{def:psi}. By reversibility we have 
\[
\psi_s(k)=r_k^+ \bigl( g(s,k+1,k)-g(s,k,k+1)\bigr) + r_k^-\bigl( g(s, k-1,k)-g(s,k,k-1)\bigr)\,.
\]In the decoupled case $g(s,k,k\pm 1) = \tau(s) E_k^\pm$ for $s\in [0,t]$ and $k\geq 0$, we get
	\[ \begin{split} 
	\partial_{\l=0} & \bbE_\pi  \big[ v(X_t^\l ) \big ] = 
	\int_0^t ds \tau(t-s) \bbE_\pi \big[ v(X_s) (E^*_r(X_{0}) - E_r (X_{0})) \big] 
	\\ & = \int_0^t ds \tau (t-s) \bbE_\pi \big[ 
	v(X_s) (
	( E_{X_{0}+1}^- - E_{X_{0} }^+ ) r_{X_{0}}^+ + 
	( E_{X_{0}-1}^+ -  E_{X_{0} }^- ) r_{X_{0}}^- )
	 \big] . 
	\end{split} 
	\]
Note that if  $E_k^+ = E^+$ and $E_k^- = E^-$ the above formula simplifies to 
	\[ \partial_{\l=0} \bbE_\pi  \big[ v(X_t^\l ) \big ] = 
	(E^- - E^+) \int_0^t ds \tau (t-s) \bbE_\pi \big[
	 v(X_s) (r_{X_{0}}^+ - r_{X_{0}}^- )
	 \big].\]
Linear response formulas for the additive functionals discussed in Theorem \ref{cor:JS} can be written down similarly.  To check that the quantities in  \eqref{aa} are in $L^p(\pi)$ for $\a$ bounded, it is enough to check that $\a$ satisfies condition $C[\pi,t]$ and Lemma \ref{LeBD} can help to this aim.

\begin{Remark}
An alternative criterion for non-explosion of the perturbed birth and death process is proved in \cite{CZ}, see Proposition 6 therein.
\end{Remark}

\subsection{Instability of non-explosion under small perturbations}\label{contro}
We provide here the counterexample mentioned in Remark \ref{rem_instabile}. As unperturbed process we  take a birth and death process on $\bbN = \{ 0,1,2,3\ldots \}$  with birth rates $ r_k^+ = (k+1)^2$ for all $k\geq 0$ and death rates $r_k^- = r_k^+$ for all $k\geq 1$. By criterion \eqref{felice1} it is simple to check that   the unperturbed  process a.s.\ does not explode. However, we can perturb it by setting $r_k^{\lambda, - } := r_k^- $ and $r_k^{\lambda , + } := e^\l r_k^+$ (i.e.~$g(s,k,k+1)=1 $ and $g(s,k+1,k)=0$). Then, again by \eqref{felice1}, one can check that   for any $\l >0$ the perturbed process explodes in finite time with positive probability. Indeed, for the perturbed process $\g_k$ in \eqref{felice1} is given by $\g_k=\sum_{j=0}^{k} \frac{1}{ (k+1-j)^2 e^{j\l}}$. In the last sum the first $C\ln k $ terms contribute for at most $C (\ln k) (k+1-C\ln k)^2$ which is summable in $k$, while the remaining terms contribute for at most $(k+1) e^{- C (\ln k)\l}$ which is summable in $k$ when taking $C=C(\l)$ large enough.

\subsection{Birth processes}\label{nascite}
We point out that, while we have presented in Lemma \ref{tik_tok}  sufficient conditions for Condition $C[\nu,t]$ to hold, these are not necessary, and in some cases one can directly and more efficiently verify Condition $C[\nu,t]$ using Definition \ref{def_papaya}. As an example, take a pure birth process on $\bbN = \{0,1,2\ldots \}$, starting at $0$ (hence $\nu=\d_0$) and staying in each state $i \in \bbN$ an exponential time of parameter $\hat r_i \in (0,\infty )$ and then jumping to $i+1$. Then  if $\a : [0,t] \times \bbN \times \bbN \to \bbR$ is a bounded measurable function, 
	\[ \bbE \Big[ \exp \Big\{ \theta \int_0^t |\a |_r (s,X_s ) ds \Big\} \Big] 
	\leq \bbE \Big[ \exp\Big\{ \theta  \| \a \|_\infty  \int_0^t \hat{r} (X_s ) ds \Big\} \Big] 
	, \]
and by taking $\theta = 1 / \| \a\|_\infty $ one can directly check that
	\[ \bbE \Big[ \exp\Big\{   \int_0^t \hat{r} (X_s ) ds \Big\} \Big] 
	= 
	1 + \sum_{n=1}^\infty \hat r_0 \hat r_1 \cdots \hat r_{n-1} \, \frac{t^n}{n!}
	 \]
which is finite as long as $\hat r_n < n/t$ for $n$ large enough. 
On the other hand, if we take e.g. $\a$ constant, the criterion in  Lemma \ref{tik_tok}  cannot be fulfilled for diverging rates $\hat r_n$. To justify  our claim, we take $\a(s,x,y)=1$ without loss of generality. Then (c) in Lemma \ref{tik_tok} reads 
$\hat{r}_n(\frac{ U(n+1)}{U(n)} -1+\theta  ) \leq C $.  If $\lim_{n\to +\infty}\hat r_n=+\infty$, we would have $\limsup_{n\to +\infty} \frac{ U(n+1)}{U(n)} \leq 1-\theta$. As $\theta>0$, this would imply that $\lim_{n\to +\infty} U(n)=0$, thus violating (a) in Lemma \ref{tik_tok}.

\subsection{Random walk on $\bbZ^d$ in a random conductance field}\label{sec:3lune} We consider a random walk  $(Y^\xi _t)_{t\geq 0}$  on $\bbZ^d$ in a random environment $\xi$.   The space of environments is given by  the product space $\Xi:= (0,A]^{\cE_d}$ with the product topology,  endowed with the Borel $\s$--field, $\cE_d$ being the set of non-oriented edges of $\bbZ^d$ and $A$ being  a fixed positive constant. We write  $\xi_{x,y}$ in place of $\xi_{\{x,y\}}$ for the value of $\xi$ at the edge $\{x,y\}$ (note that $\xi_{x,y}=\xi_{y,x}$).
Since the environment $\xi$ at a given  edge does not depend on the orientation of the edge, $\xi$ is also called \emph{conductance field}.
Given  $\xi\in \Xi$ the random walk  $(Y^\xi _t)_{t\geq 0}$ starts at the origin and performs nearest--neighbour jumps with jump rate from $x$ to $y$ given by $r(x,y):=\xi_{x,y}$. 
We consider the perturbed random walk $(Y^{\xi,\l} _t)_{t\geq 0}$  with perturbed jump rates $r^\l_t(x,y)=r(x,y)  e^{\l g^\xi(t,x,y)}= \xi_{x,y}e^{\l g^\xi(t,x,y)}$ where $g$ is bounded and measurable in $\xi, t,x,y$.  As $\xi _{x,y} \leq A$, both the original random walk and the perturbed one a.s.\ do not explode, $g$ satisfies condition $C[\nu,t]$ for any distribution $\nu$ and any time $t$ and  one can therefore apply Theorems \ref{th:JS} and \ref{cor:JS} to deal with the linear response (for each fixed environment $\xi$).

 To benefit from the stationarity and get more explicit formulas,  it is convenient to change viewpoint by considering the process \emph{environment viewed from the particle}, as we now detail.
The group $\bbZ^d$ acts on $\Xi$  by spatial translations as 
$(\t_z \xi)_{x,y}:= \xi  _{x+z, y+z }$.
We fix a probability measure $\cP$  on $\Theta$ which is stationary   w.r.t.\ the spatial translations $\t_z$ and such that 
\be\label{fante}
\cP( \xi\in \Xi \,:\, \t_z \xi = \t_{z'} \xi \text{ for some }z\not =z' \text{ in }\bbZ^d)=0
\en
(for example $\cP$ can be a product probability measure on $\Xi$).  We assume that  also $g$ is stationary, i.e.\ $g$ is of the form 
\[ g^\xi(t,x,y) = h(t, \t_x \xi, y-x)\]
for some bounded measurable function $h:  [0,+\infty) \times \Xi \times \{z\in \bbZ^d:|z|=1\}$.

Given $\xi\in \Xi$ we write $(\bar \xi_t)_{t\geq 0}$  for the    Markov jump process given by the environment viewed from the walker when the latter starts at the origin with environment $\xi$. Simply we have $\bar \xi_0:=\xi$ and  $\bar \xi _t:= \t_{Y^\xi_t} \o$ for all $t\geq 0$.
 The Markov jump process $(X_t)_{t\geq 0}$ we are interested in is just $ (\bar \xi_t)_{t\geq 0}$.   The space $(\cX,\cB)$ is then given by $\Xi$ with the Borel $\s$--field  and the transition  kernel is given  by 
\[
r(\xi, \cdot) := \sum _{z:|z|=1} \xi_{0,z}  \d_{\t_z \xi}(\cdot)\,.
\]Note that now the perturbation is dictated by the new function $\bar g(t,\xi, \xi'):= h(t, \xi, z)$  if $\xi'=\t_z \xi$ with $|z|=1 $ (i.e.\ $r_t^\l(\xi, d \xi')= e^{\l \bar g(t,\xi, \xi')} r(\xi, d \xi')$ as in \eqref{kolpakovZZZ}).
Moreover the random walk $(Y^\xi_t)_{t\geq 0}$ starting at the origin can be written as an additive functional of $\bar \xi _{[0,t]}$:
\be\label{cavallo}
Y_t^\xi =F ( \bar \xi_{[0,t]} ):=\sum_{s\in (0,t]: \bar \xi_{s-}\not = \bar \xi_s} \a(\bar \xi_{s-},\bar \xi_s) \qquad \forall t\geq 0\,,
\en
where 
\[\a(\xi', \xi''):= \begin{cases} z & \text{if } \xi''=\t_z\xi'  \text{ for some $z$ with $|z|=1$}\,,\\
0 & \text{otherwise}\,.
\end{cases}
\]
Although a priori the above  function $\a$ is not well defined pointwise, due to \eqref{fante} for $\cP$--a.a.\  $\xi$ the  expression $F(\bar \xi_{s-},\bar \xi_s)$ in \eqref{cavallo}  is well defined for all times $s$ (similar considerations hold for $\bar g(s,\xi,\xi')$ defined above). 

By the stationarity of $\cP$ w.r.t.\ the spatial translations $\t_z$ and since $\xi_{x,y}=\xi_{y,x}$, we get that $\cP$ is a reversible distribution for the process $(\bar{\xi}_t)_{t\geq 0}$.  Moreover, we have 
\[ (\alpha \bar g )_r ( s, \xi) = \sum_{z:|z|=1} \xi_{0,z} \,z\, h(s, \xi, z) \,,\qquad
	 \a_r (s,\xi)=  \sum_{z:|z|=1} \xi_{0,z} z\,.\]
Since $|\a|_r $ is uniformly bounded, Condition $C[\cP, t]$ is satisfied (see~Definition~\ref{def_papaya}). Hence, by Theorem \ref{cor:JS}, we have 
\begin{multline}\label{conad}
\partial_{\l=0} \int _\Xi d\cP (\xi )  \bbE ^\xi_0 \bigl[ Y^{\xi,\l}_t \rosso{ \bigr]}  = 
	 \int_0^t   ds \int _\Xi d\cP (\xi )  \bbE^\xi_0 \big[ (\alpha \bar g )_r (s, \t_{Y^\xi_s}\xi ) \big] 
	\\
	 - \int_0^t ds \int_0^s du \, 
	 \int _\Xi d\cP (\xi )  \bbE ^\xi_0  \Big[ \alpha_r (s, \t_{Y^\xi_s}\xi )     \psi_{s-u}( \t_{Y^\xi_{s-u}}\xi \rosso{ ) } \Big] 
	\,,
	\end{multline}
	where $\bbE^\xi_0[\cdot]$ denotes the expectation w.r.t. the random walk $(Y^{\xi,\l}_t )_{t\geq 0}$ starting at the origin in the fixed environment $\xi$ and 
	$ \psi_s(\xi)=\sum_{e:|e|=1} \xi_{0,e} ( h(s, \t_e \xi, -e)- h(s,\xi, e) )$.

Let us make more explicit  formula \eqref{conad} when
$g(t,x,y):= (y-x)\cdot v$ for $|x-y|=1$ where $v$ is a fixed vector. This case  corresponds to applying to the particle an external field, constant in time and space. The same external field appears  e.g.~in \cite{FGS2,GGN,GMP,LR,MP}. We write $\varphi (\xi)$ for the so called local drift of the unperturbed random walk, i.e.~$\varphi(\xi):= \sum_{z:|z|=1} \xi_{0,z} z$, and we also write  $\theta(\xi):=  \sum_{z:|z|=1} \xi_{0,z} (z  \cdot v) z$.
 Since $h(t,\xi,z)=z\cdot v$,  it is easy to check  that
\[  \a_r(s,\xi) =\varphi(\xi) \qquad 
(\a \bar g )_r (s, \xi)=  \theta(\xi)  \qquad \psi_s(\xi)=- 2 \varphi(\xi) \cdot v\,.
\]
As a consequence \eqref{conad} reads (by the change of variables $u\to s-u$)
\begin{equation*}
	\begin{split}
\partial_{\l=0} \int _\Xi d\cP (\xi )  \bbE ^\xi_0 \bigl[ Y^{\xi,\l}_t  \bigr]&  = 
	 \int_0^t   ds \int _\Xi d\cP (\xi )  \bbE^\xi_0 \big[ \theta( \t_{Y^\xi_s}\xi ) \big] 
	\\
&	 +2  \int_0^t ds \int_0^s du \, 
	 \int _\Xi d\cP (\xi )  \bbE ^\xi_0  \Big[ \varphi( \t_{Y^\xi_s}\xi )    \big( \varphi( \t_{Y^\xi_{u}}\xi ) \cdot v\big)  \Big] 
	\,.
	\end{split}
	\end{equation*}	
	Since the probability  $\cP$ is invariant (and even reversible) for the process ``environment viewed from the particle", we can simplify the above expression and get 
	\begin{equation*}
	\begin{split}
\partial_{\l=0} \int _\Xi d\cP (\xi )  \bbE ^\xi_0 \bigl[ Y^{\xi,\l}_t  \bigr]&  = 
	 t \int _\Xi d\cP (\xi ) \theta(\xi)
	\\
	 &+2  \int_0^t dr (t-r)  \, 
	 \int _\Xi d\cP (\xi )  \left( \varphi(\xi)\cdot v\right)  \bbE ^\xi_0  \Big[ \varphi( \t_{Y^\xi_{r}}\xi )     \Big] 
	\,.
	\end{split}
	\end{equation*}	
%
%

	\subsection{Complex mobility matrix}\label{es_CM} We use here the notation introduced in Section \ref{sec_CM} \rosso{dealing with nearest-neighbor random walks}.
Suppose that the unperturbed Markov jump process  on the torus $\bbT^d_N=\bbZ^d/ N \bbZ^d$ has spatially homogeneous jump rates, i.e.\ $r(x,y)=r(x+z,y+z)$ for all $x,y\in\bbT^d_N$,  $z \in \bbZ^d$, where the sums $x+z$, $y+z$ are thought modulo $N \bbZ^d$. 
We consider the perturbation with jump rates \eqref{pert_rates} with $\o\not=0$. 
As $r^\l_t(x,y)$ depends on $x,y$ only via $y-x$, 
 one can directly compute the mean instantaneous velocity   $V_\l(t)$ given in \eqref{matteo105} getting 
$ V_\l(t)= \sum_{e: |e|=1} r_t^\l (0,e) e= \sum_{e:|e|=1} \exp\{\l \cos (\o t) e\cdot v\} r(0,e)e$. As a consequence
\begin{align}
 \partial_{\l=0} V_\l(t)& = \sum_{e:|e|=1} \cos (\o t) (e\cdot v)  r(0,e) e= \Re (e^{i\o t} \s(\o) v)\,,\\
 \s(\o) v & =\sum_{e:|e|=1} (e\cdot v)  r(0,e) e \,.\label{frittella}
\end{align}
In particular, denoting  the canonical basis of $\bbR^d$ by $e_1,e_2, \dots, e_d$, we have $\s(\o) e_j=\left(r(0,e_j)+r(0,-e_j) \right) e_j$, i.e. $\s(\o)_{i,j}=\d_{i,j} \left(r(0,e_i)+r(0,-e_i)\right)$. 
Note that, with spatial homogeneity, $\s(\o)$ does not depend on the frequency $\o$.
The direct computation of $V_\l(t)$ becomes  more involved in the presence of spatial heterogeneity, where $\s(\o)$ exhibits a nontrivial dependence on $\o$.

\smallskip
  We now  use  directly   Theorem \ref{teo_CM} to compute $\s(\o)$ in the special case  given by  $d=1$, $N$ even   and  2-periodic unperturbed jump rates of the form  \[ 
r(x, x+1) = \begin{cases} r^+_0 & \text{ if } x\equiv 0  \,,\\
r^+_1& \text{ if }x \equiv 1 \,,
\end{cases}
\qquad
r(x, x-1) = \begin{cases} r^-_0& \text{ if  } x\equiv 0 \,,\\
r^-_1 & \text{ if } x \equiv 1\,, 
\end{cases}
\]
for positive constants  $r^\pm _0$, $r^\pm _1$,
where we write $x\equiv 0$ if $x$ is even, and $x\equiv 1$ if $x$ is odd. Then the unperturbed invariant distribution is given by 
\[ \pi(x)= 
\begin{cases} 
  ( r_1^+ + r_1^-) /\cZ &\text{ if } x\equiv 0   \,,\\
  ( r_0^+ + r_0^-) /\cZ   & \text{ if } x \equiv 1\,,
\end{cases}
\]
where $\cZ$ is the normalizing constant $\cZ= (N/2) ( r_0^+ + r_0^-+r_1^+ + r_1^-)$. Moreover the functions $c,\g$ in Theorem \ref{teo_CM} are given by 
\[
c(x)= 
\begin{cases} c_0:=
 r^+_0 + r^-_0 &\text{ if } x\equiv 0\,,\\
 c_1:=r^+_1 + r^-_1  & \text{ if } x\equiv 1\,,
\end{cases}
\qquad
\g(x) =
\begin{cases} \g_0:= r^+_0 - r^-_0 
& \text{ if  } x\equiv 0 \,,\\
\g_1:= r^+_1 - r^-_1  & \text{ if } x \equiv 1 \,.
\end{cases}
\]
The reversed rates are then given by 
\[ 
r^*(x, x+1) =
\begin{cases} (c_0/c_1 )
 r^-_1
  & \text{ if } x\equiv 0 
 \,,\\
 (c_1/c_0)
 r^-_0
  & \text{ if } x\equiv 1\,,
\end{cases}
\qquad
r^*(x, x-1) =
\begin{cases}
(c_0/c_1)
   r^+_1
  & \text{ if } x\equiv 0 
 \,,\\
(c_1/c_0) r^+_0
  & \text{ if } x\equiv 1\,,
\end{cases}
\]
and the function $\Psi $ in \eqref{kinder_cards} is given by
\be\label{istrione1}
\Psi(x) = 
\begin{cases}
(c_0/c_1) \g_1-\g_0 =c_0( \g_1/c_1 -\g_0/c_0) & \text{ if } x\equiv 0\,,\\
(c_1/c_0) \g_0-\g_1= c_1( \g_0/c_0 -\g_1/c_1) & \text{ if } x\equiv 1\,.
\end{cases} 
\en
If $f:\bbT^d_N\to \bbC$ has period 2 (i.e.\ it is constant on even sites and constant on odd sites), then
\be\label{istrione2}
(i\o -\cL^*) f(x) =
\begin{cases} i \o f(0)- c_0 \bigl ( f(1)- f(0) \bigr) & \text{ if } x\equiv 0\,,\\
i \o f(1)- c_1 \bigl ( f(0)- f(1) \bigr) & \text{ if } x\equiv 1\,.
\end{cases}
\en
By comparing \eqref{istrione1} and \eqref{istrione2} we get
\[
(i\o -\cL^*) ^{-1} \psi(x)=
\begin{cases} 
\frac{ c_0 ( \g_1/c_1- \g_0/c_0)}{  i \o +c_0 +c_1} & \text{ if } x\equiv 0\,,\\
\frac{ c_1 ( \g_0/c_0- \g_1/c_1)}{ i \o +c_0 +c_1}  & \text{ if } x\equiv 1\,.
\end{cases}
\]
By \eqref{jabba1} in Theorem \ref{teo_CM} we then get the following expression for the complex mobility constant:
\be\label{fruttolo}
\s(\o)=\frac{c_0 c_1}{c_0+c_1} \left[ 2 + \left(\frac{\g_1}{c_1}- \frac{\g_0}{c_0}\right) \frac{\g_0 - \g_1}{i\o +c_0 +c_1} \right]\,.
\en
Note that, in the spatially homogeneous case $r_1^+=r_0^+=r^+$ and $r_1^-=r_0^-=r^-$, \eqref{fruttolo} reduces to $\s(\o)=r^++r^-$ in agreement with \eqref{frittella}. Moreover, coming back to the general setting,  
 we have reversibility if and only if $r_1^+/r_1^-= r_0^-/r_0^+$, i.e.\ $r_1^+=\a r_0^-$ and $r_1^-= \a r_0^+$ for some $\a>0$. 
  Finally, we point out that one could have computed directly $V_\l(t)$ by finding the distribution  $\pi_{\l,t}$ of the OSS at time $t$  as $\pi_{\l,t}$ must be spatially 2-periodic. In particular, $\pi_{\l,t}$ can be computed from  the continuity equation:
 \be
 \begin{split} \partial_t \pi_{\l,t}(0)  &+ \pi_{\l,t}(0) \left[ e^{\l \cos (\o t) }r_0^++e^{-\l \cos (\o t) }r_0^-\right]\\& -  \pi_{\l,t}(1) \left[ e^{\l \cos (\o t) }r_1^++e^{-\l \cos (\o t) }r_1^-\right]=0
 \end{split}
 \en
 (use also that $ \pi_{\l,t}(1)=1- \pi_{\l,t}(0) $ and that $  \pi_{\l,t}(0) $ is $T$--periodic  for $T=2 \pi /\o$). The computation of $\s(\o)$ by means on Theorem \ref{teo_CM} is, on the other hand, simpler.

\section{Stochastic calculus background}\label{sec:SC}
We collect here some useful facts from the theory of stochastic calculus for processes with jumps. Our discussion is based on \cite{D} and \cite[Chapter~1]{JS}. 

We first prove Lemma \ref{giacinto} for later use:
\begin{proof}[Proof of Lemma \ref{giacinto}] We just prove \eqref{intA_zero}, as the rest of the lemma follows trivially from \eqref{intA_zero}. Defining  $\a(s,x,y):= 0$ if $s>t$ it is enough to prove that 
\begin{equation}\label{intA_uno}
	  \bbE_{x_0}\Big[ \sum _{s \in (0,+\infty)} |\a  (s,  X_{s-}, X_s)| \Big] 
	=
	 \bbE_{x_0} \Big[ \int_0^{+\infty} |\a|_r (s, X_s) ds \Big]   
	\end{equation}
	for each starting point $x_0$ such that  the unperturbed process has a.s. no explosion (this holds for $\nu$--a.a.~$x_0$).
 Let  $\t_1<\t_2<\t_3<\cdots $ be the  jump times of the unperturbed Markov jump process starting at $x_0$. As  a.s. this process does not explode and since $\hat r (x)\in (0,+\infty)$ for all $x\in \cX$, all times $\t_k$ are finite and diverge to $+\infty$.

We have
\begin{multline*}
 \bbE_{x_0} \Big[ \int_0^{\t_1} |\a|_r (s, X_s) ds \Big]  =
 \int_0 ^{+\infty}  dt_1 e^{- \hat r(x_0)t_1 } \hat r(x_0) \int_0 ^{t_1} ds\,| \a|_r (s,x_0)  \\=
 \int_0^{+\infty} ds \, |\a|_r(s,x_0)\int_{s}^{+\infty} dt_1 e^{- \hat r(x_0)t_1 }  \hat r(x_0)= \int_0^{+\infty} ds \, |\a|_r(s,x_0)  e^{- \hat r(x_0)s } \,,
\end{multline*}
while
\begin{equation*}
\begin{split}
  \bbE_{x_0}\Big[  |\a  (\t_1,  X_{\t_1-}, X_{\t_1})| \Big] & = \int_0^{+\infty} ds\, e^{- \hat r(x_0)s } \hat r(x_0) \int _{\cX}  r(x_0,dx_1)\frac{1}{\hat r(x_0)} |\a(s, x_0,x_1)| \\
&  =  \int_0^{+\infty} ds \,e^{- \hat r(x_0)s }  |\a|_r(s,x_0)  \,.
\end{split}
  \end{equation*}
  The above results imply that $\bbE_{x_0} \big[ \int_0^{\t_1} |\a|_r (s, X_s) ds \big]  =\bbE_{x_0}\big[  |\a  (\t_1,  X_{\t_1-}, X_{\t_1})| \big]$. By  conditioning on $\t_k, X_{\t_k}$, we then get
      \[
      \bbE_{x_0} \big[ \int_{\t_k}^{\t_{k+1}} |\a|_r (s, X_s) ds \big]=    \bbE_{x_0} \big[  |\a  (\t_{k+1},  X_{\t_{k+1}-}, X_{\t_{k+1}})| \big]
  \]
for all $k\geq 0$, where $\t_0:=0$. By summing among $k\geq 0$ and using that $\t_k \to +\infty$  we get \eqref{intA_uno}.
\end{proof}

\subsection{Martingales and local martingales} In this subsection and in the next one we let  $[0,t]$ be the time observation window, which in later applications will be the time window of the perturbed Markov jump process. 
Let us denote by $(\Omega , \cF^0 , \bbP_\nu )$ the probability space on which the unperturbed Markov process $X_{ [0,t]}$ is defined. 
Denote by $(\cF_s^0)_{s\in [0,t]}$ the natural filtration associated to it, that is $\cF_s^{0}$ is the smallest $\sigma$--algebra that makes the random variables $\{ X_u : u\leq s \}$ measurable. We can make this into a right--continuous filtration $(\cF_s)_{s\in [0,t]}$ that satisfies the so called \emph{usual conditions} \cite{JS} by setting $\cF_t := \sigma (\cF^0_t , \cN )$ and, for $s\in [0,t)$,  $\cF_s := \lim_{u \searrow s}  \sigma ( \cF_u^0 , \cN )$, where in general $\s ( \cF^0_s , \cN )$ is the smallest $\s$-algebra containing both $  \cF_s ^0 $ and $\cN$, and $\cN$ is the collection of all subsets of sets in $\cF^0$ with $\bbP_\nu$--measure zero. Similarly we define $\cF:=\s( \cF^0, \cN)$. Then $(\cF_s)_{s\in [0,t]}$ is right--continuous, $\cF_s\subset \cF$  and $\cF_0 \supseteq \cN$. 
We can therefore think of the unperturbed Markov jump process as being defined on the filtered probability space $(\Omega , \cF , (\cF_s )_{s \in [0,t]} , \bbP_\nu )$, where we keep the notation $\bbP_\nu$ for the probability measure on $(\O,\cF)$ given by the completion of  the original $\bbP_\nu$, in particular giving zero mass to the sets in $\cN$. We remark that $\Omega $ can be $D([0,t] , \cX )$, in which case $\bbP_\nu$ coincides with the law of the unperturbed process.

A c\`adl\`ag adapted process $M = (M_s)_{s\in [0,t]}$ on the filtered probability space \linebreak $( \O , \cF , (\cF_s)_{s\in [0,t]} , \bbP_\nu )$ is said to be a \emph{martingale} if $M_s$ is integrable and $\bbE_\nu [ M_s | \cF_u ] = M_u $ almost surely, for all $0\leq u \leq s\leq t$. It is said to be a \emph{local martingale} if there exists a non--decreasing sequence  $(T_n)_{n\geq 0}$ of  stopping times with respect to the filtration $(\cF_s)_{s\in [0,t]}$  such that $T_n  \to t $ almost surely as $n\to \infty$, and the stopped process  $(M^{T_n}_s)_{s\in [0,t]}$ defined by $M^{T_n}_s = M_{s \wedge T_n}$ is a martingale 
 for all $n\geq 0$. We recall that a stopping time $T$ with respect to the filtration $(\cF_s)_{s\in [0,t]}$ is a random time such that $\{ T\leq s \} \in \cF_s $ for all $s \in [0,t]$.

A sufficient condition for a local martingale $(M_s)_{s\in [0,t]}$ to be a true martingale is given by the following result. 
\begin{Lemma}\label{le:sufficientmart}
Let $M=(M_s)_{s\in [0,t]}$ be a local martingale, and assume that there exists an integrable random variable $Y$ such that $|M_s| \leq Y$ for all $s \in [0,t]$. Then $M$ is a true martingale. 
\end{Lemma} 
This is a straightforward corollary of \cite[Proposition 1.47-(c)]{JS} together with the observation that under the assumptions of Lemma \ref{le:sufficientmart} the process $M$ is of class (D), as defined in \cite[Definition 1.46]{JS}.

A local martingale is said to be continuous if its trajectories are continuous. In what follows we will work with a class of martingales which is orthogonal, as defined below, to that of continuous local martingales, namely the class of purely discontinuous local martingales.

\subsection{Purely discontinuous local martingales} \label{sec:purely}
Recall from \cite[Definition 4.11(a)]{JS} that two local martingales are said to be orthogonal if their product is a local martingale. A local martingale equal to zero at time $t=0$ and which  is orthogonal to all continuous local martingales is called a \emph{purely discontinuous local martingale} (\cite[Definition 4.11(b)]{JS}). If, in addition, it is a true martingale, we call it a \emph{purely discontinuous martingale}.

Let $\a : [0, t] \times \cX \times \cX \to \bbR $ be a  measurable function such that 
	\begin{equation}\label{welldef}
	 \sum _{s \in (0,t]} |\a (s,  X_{s-}, X_s)| < \infty , 
	  \qquad 
	  \int_0^{t} |\a |_r  (u,X_u ) du < \infty 
	  \end{equation}
$\bbP_\nu $--almost surely. 
Then similarly to  \cite[Theorem (A4.9), p.~272]{D} 
the process $(M_s)_{s\in [0,t]}$ defined by 
	\begin{equation}\label{E:M}
	\begin{split}
	 M_s & = 
	\sum _{u \in (0,s]} \a (u,  X_{u-}, X_u)
	- \int_0^s  \a_r  (u,X_u ) du
	\end{split} 
	\end{equation}
is a local martingale (see Appendix \ref{appendix} for the details).  
Moreover, it is of finite variation, since it is made of a piecewise constant term and a Lebesgue integral term. 
Since all local martingales starting at zero and of finite variation are purely discontinuous (see Lemma 4.14(b) of \cite{JS}), it  follows that $(M_s)_{s\in [0,t]}$ is a purely discontinuous local martingale. 


For the next lemma recall Definition \ref{def:Pint}. 
\begin{Lemma}\label{L:true}
Let $M=(M_s)_{s\in [0,t]}$ be a local martingale of the form \eqref{E:M} with $\alpha $ satisfying \eqref{welldef}. If $\alpha$ is $\bbP_\nu$--integrable (for example if $\alpha $ satisfies Condition $C[\nu,t]$), then $M$ is a martingale. 
\end{Lemma}
\begin{proof} It is enough to apply  Lemma \ref{le:sufficientmart} with 
\[ Y:= \sum _{u \in (0,t]} |\a| (u,  X_{u-}, X_u)
	+ \int_0^t | \a|_r  (u,X_u ) du\,.
	\qedhere \]
\end{proof}

Let $(N_s)_{s\in[0,t]}$ be another such local martingale, with 
	\[ N_s = \sum _{u \in (0,s]} \gamma  (u,  X_{u-}, X_u)
	- \int_0^s  \gamma_r (u,X_u) du\]
where $\gamma : [0,t] \times \cX \times \cX \to \bbR$ satisfies \eqref{welldef} with $\gamma$ in place of $\alpha$. We define the covariation process $([M,N]_s)_{s\in[0,t]}$ by setting 
	\[[M,N]_s =  \sum_{u\in (0,s] } \a (u,X_{u-} , X_u ) \g (u,X_{u-} , X_u )
	\]
(cf.\ Definition 4.45 and Theorem 4.52 in \cite{JS}, and use that the continuous martingale part, defined in Theorem 4.18 in \cite{JS}, of purely discontinuous local martingales is identically zero). It then follows from Proposition 4.50 of \cite{JS} that 
the process $(M_s N_s - [M,N]_s )_{s\in[0,t]}$ is again a local martingale.

\begin{Remark}\label{fiorellino} In this subsection and in the previous one we have worked  with the unperturbed Markov jump process up to time $t$. Equivalently, one could deal with this process defined for  all times, and in particular  defined on the filtered probability space $(\O,\cF, (\cF_s)_{s\geq 0}, \bbP_\nu)$ where $\cF_s := \lim_{u \searrow s}  \sigma ( \cF_u^0 , \cN )$  for all $s \geq 0$. We point out that in this case, in the definition of local martingale, one has to take a sequence of stopping times $T_n$ such that $T_n \to +\infty$ almost surely.  Then, given $\a: [0,t]\times \cX\times \cX\to\bbR$, in order to define processes as $M_s$ in  \eqref{E:M} for all times $s\geq 0$, one has just to extend $\a$ to $\bbR_+\times \cX\times \cX$  by setting $\a(s,\cdot, \cdot)=0$ for times $s > t$.
 \end{Remark}

%
%
%

\section{Proof of Lemma \ref{ananas78X}}\label{sec:preliminary}
We start with a preliminary lemma. 
\begin{Lemma}\label{mango25X}
Let  $F(s,y,z)  $  be a measurable function on $[0,t]\times \cX\times \cX$ such that 
\be\label{jeegX} (e^F)_r(s,y)= \int _{\cX}  e^{F(s, y,z)} r(y,dz)<+\infty \text{ for all } s\in [0,t]\,, \; y\in 
 \cX
 \en
 and  define  $\bbM_t^F : D_f ([0,t], \cX) \to \bbR$ as 
\be
  \bbM_t^F\bigl(\xi_{[0,t]} \bigr):= \exp\bigl\{ \sum_{s\in[0,t]}
  F(s, \xi_{s-}, \xi_s)  
  -\int_0^t (e^F-1 )_r(s, \xi_s)ds)  \bigr\}\,.
  \en
 Then,  $\bbE_x [\bbM_t^F(X_{[0,t]})]\leq 1 $ for $\nu$--a.a.~$x$.
 \end{Lemma}
Note that $1_r(s,y)= \hat r(y)<+\infty$, hence $(e^F -1)_r$ is well defined and finite by \eqref{jeegX}.
 \begin{proof}
Consider the time-inhomogeneous Markov jump process $X^F_{ [0,t]}$ on $\cX$  with transition kernel $r^F_s(y,dz):= r(y,dz) e^{F(s, y,z)}$,  defined up its explosion time $\t_\infty$. 
Given a  Borel set $B \subset D_f([0,t],\cX)$,  let $P^F_{x,t}(B) $ be  the probability that  $X^F_{[0,t]}\in B$ when starting at $x$  (note that the event  $\{ X^F_{[0,t]}\in B\}$  implies that  $X^F_{[0,t]}$ does not explode in $[0,t]$). Call $P_{x,t}(B)$ the analogous probability for $X_{[0,t]}$.  $P^F_{x,t}$ and $P_{x,t}$  are measures on $D_f([0,t],\cX)$. Take $x$ such that a.s. the \rosso{unperturbed} Markov process  starting at $x$ does not explode (\rosso{thus implying that $P_{x,t}$} is a probability measure). Note that this holds for $\nu$--a.a.~$x$ by our main Assumption in Section~\ref{sec_kalush}.
 Then one easily checks (as for \eqref{ugo}) that  $\bbM_t^F $ is the Radon--Nikodym derivative 
 of the measure $P^F_{x,t} $  w.r.t. $P_{x,t}$. As $P^F_{x,t}$ has total mass bounded by $1$, we have $\bbE_x\big[\bbM_t^F(X_{[0,t]}) \big]= P^F_{x,t} \big  ( D_f([0,t],\cX)\big) \leq 1 $. 
\end{proof}
 \begin{proof}[Proof of Lemma \ref{ananas78X}]
 We  fix $\d>0 $ and
 set $F(s,y,z):=\ln ( 1+\d  |\a|(s,y,z))$. Then 
 $(e^F)_r(s,y)= \hat r(y) + \d | \a|_r (s,y)  
\leq (1+\d \|\a\|_\infty) \hat r(y)$. In particular, condition \eqref{jeegX} is satisfied. 
 By Lemma \ref{mango25X} we then get that $\bbE_\nu[ \bbM_t ^F( X_{[0,t]})]\leq 1 $.
 Since,
  $(e^F-1)_r(s,y)=\d  | \a|_r (s,y)$, $ \bbM_t ^F( \xi_{[0,t]})$ can be rewritten as 
  \be
  \bbM_t ^F( \xi_{[0,t]})=  \exp\bigl\{ \sum_{s\in(0,t]}
  F(s, \xi_{s-}, \xi_s)  
  -\d \int_0^t |\a|_r(s, \xi_s)ds \bigr\} .
  \en
  As $\ln (1+ x) \geq x/2$ for $x\in [0,1]$, by taking $\d$ small such that $\d \|\a\|_\infty\leq 1$ we get 
  that 
  \be\label{sammy2} \bbE_\nu\left[ \rosso{\exp\{N_t ( X_{[0,t]})\}}\right]  \leq \bbE_\nu\left[ \bbM_t ^F( X_{[0,t]})\right]\leq 1
  \en 
  where 
  \be
  N_t(\xi_{[0,t]}) :=  \frac{\d}{2} \sum_{s\in(0,t]}
|\a|(s, \xi_{s-}, \xi_s)  
  - \d \int_0^t |\a|_r(s, \xi_s)  ds  \,.
    \en
We now  observe that, by Schwarz inequality, \eqref{papayaX} and \eqref{sammy2}, for $\d\leq \theta$  it holds \be
\begin{split}
\bbE_\nu \big[ e^{\frac{\d}{4}  \sum_{s\in (0,t]}
  |\a|(s, X_{s-}, X_s)  }  \big]&= \bbE_\nu \big[
  e^{ \frac{1}{2}N_t (X_{[0,t]} ) + \frac{   \d}{2} \int_0^t |\a|_r(s, X_s)ds }\big]\\
  & \leq  \bbE_\nu \big[
  e^{N_t (X_{[0,t]})} \big]^\frac{1}{2} \bbE_\nu \big[
 e^{  \d \int_0^t |\a|_r(s, X_s)ds }\big]^\frac{1}{2}<+\infty\,.
  \end{split}
  \en
By the above considerations,  \eqref{kinney2} holds for $\g:=\d/4$ and in particular for $\g:=  \min\{ \|\a\|^{-1}_\infty , \theta\}/4 $.
  \end{proof}

\section{Proof of Lemma \ref{tik_tok} and its extension} \label{dim_lemma_tik_tok}
The following result reduces to Lemma \ref{tik_tok} when $U_n=U$ for all $n$:
\begin{Lemma}\label{tik_tok_ext} 
For a given function $\alpha : [0,t] \times \cX \times \cX \to \bbR$ suppose that there exist a sequence of measurable real functions $U_n $ on $\cX$ and positive  constants $\theta,C,c$  such that
 \begin{itemize}
\item[(i)]  $U_n(x) \geq c  $ for all $x\in \cX$ and $n\geq 1$;
\item[(ii)]   $\int_{ \cX}  U_n(y) r(x,dy) <+\infty$ for all $x\in \cX$ and $n\geq 1$;
 \item[(iii)]   setting  $V_n(x):= -LU_n(x)/U_n(x)$, the sequence of functions $V_n:\cX \to \bbR$ converges pointwise to some function $V:\cX\to \bbR$;
 \item[(iv)]  $V\geq \theta\, |\alpha |_r - C$;
 \item[(v)]  $U_{\rm sup}(x):= \sup_{n\geq 1} U_n(x) <+\infty$ for each $x \in \cX$;
 \item[(vi)]    $\nu [ U_{\rm sup}] <+\infty$.
  \end{itemize}
 Then $\alpha$ satisfies Condition $C[\nu , t]$ with parameter $\theta$.
\end{Lemma}
\begin{proof} We use Lemma \ref{mango25X} with the function $F_n(y,z) := \ln (U_n(z)/U_n(y) )$, which is well defined by Item (i). Moreover $(e^{F_n})_r(y)= U_n(y)^{-1} \int_{\cX} r(y,dz) U_n(z) <+\infty$ due to Items (i) and (ii). By observing that 
	\[ \exp\Big\{\sum_{s\in(0,t]} F_n (s, \xi_{s-}, \xi_s)  \Big\}= \frac{U_n(X_t)}{U_n(X_0)} \]
and  $ (e^{F_n}-1)_r=LU_n/U_n$, we get   that
 \be\label{zittino}
\bbM_t^{F_n}(X_{[0,t]})= \frac{U_n(X_t)}{ U_n(X_0)} \exp\left\{  - \int_0^t  \frac{LU_n}{U_n } (X_s ) ds  \right\} \geq  \frac{c}{ U_{\rm sup}(X_0)} \exp\left\{ \int_0^t V_n (X_s) ds  \right\} \,. 
 \en
 To get the above lower bound we used Item (i)   and the definitions of $U_{\rm sup},V_n$. As a byproduct  of 
 \eqref{zittino} with the bound  $\bbE_x \bigl[\bbM_t^{F_n}(X_{[0,t]})\bigr]\leq 1$ (which holds for $\nu$--a.a.~$x$ by Lemma \ref{mango25X}) we get that
 	\[ \bbE_x \Big[ \exp\left\{ \int_0^t V_n (X_s) ds  \right\} \Big] \leq \frac{U_{\rm sup} (x)}{c} \]
for $\nu$--a.a.~$ x \in \cX$.  By taking the limit $n\to \infty$ (using Item (iii) and Fatou's lemma) we get $\bbE_x
 \left[ 
 e^{\int_0^t V(X_s) ds } 
 \right]
 \leq
  U_{\rm sup}(x)/c$. By combining the above bound with 
  Item (iv), we get that
  \be\label{crodino}
 \bbE_x\left[ e^{\theta \int_0^t |\alpha |_r (s, X_s) ds  }\right]\leq e^{ Ct} \frac{ U_{\rm sup}(x)}{c} 
 \en
for $\nu$--a.a.~$x\in \cX$. Finally, by  averaging the above bound with respect to $\nu$ and using  Item   (vi), we 
 gather that 
 	\[ \bbE_\nu \left[ e^{\theta \int_0^t |\alpha |_r (s, X_s) ds  }\right]
 	\leq e^{ Ct} \frac{ \nu [ U_{\rm sup}]}{c} < \infty . \]
	This in particular implies  \eqref{papayaX}.
\end{proof}

\section{Proof of Theorem \ref{th:explosionX}} \label{sec:explosion}
To start with, recall that \eqref{ugo} has been obtained under the assumption that the perturbed process does not explode in $[0,t]$ $\bbP_\nu$--a.s.. Nevertheless, the same identity remains valid when dropping the non--explosion assumption by replacing  $\bbE_{\nu}\bigl[  F(X^\l_{[0,t]} )  \bigr]$ in the left hand side of \eqref{ugo} by $\bbE_{\nu}\bigl[  F(X^\l_{[0,t]})\mathds{1}(\t_\infty^\l  >t )   \bigr]$. We recall that  $\t_\infty^\l $ denotes the explosion time of the perturbed process. Then, taking $F\equiv 1 $, 
	\[ \bbP_{\nu}\bigl(  \t_\infty^\l  >t    \bigr)= \bbE_\nu  \bigg[ 
e^{  \int_0^t \bigl[ \hat r (X_s)- \hat{r}^\l_s ( X_s)\bigr] ds} 
  \prod _{\substack{s \in (0,t]: \\ 
  X_{s-} \not = X_s} }e^{\l g(s,X_{s-} , X_s)}\bigg] , \]
and the non--explosion of the perturbed process up to time $t$  becomes equivalent to  
\be\label{midnight}
 \bbE_\nu  \bigg[ 
e^{  \int_0^t \bigl[ \hat r (X_s)- \hat{r}^\l_s ( X_s)\bigr] ds} 
  \prod _{\substack{s \in (0,t]: \\ 
  X_{s-} \not = X_s} }e^{\l g(s,X_{s-} , X_s)}\bigg]=1\,.
\en
Below we prove that \eqref{midnight} holds for $\l$ small enough by observing that the l.h.s.\ is the expectation of an exponential martingale associated to the change of measure $\bbP_{\nu} \mapsto \bbP_\nu^\l $. Our discussion is based on  stochastic calculus (see \cite{JS} and Section \ref{sec:SC} above). 

For $s\in [0,t]$ set 
	\[ Y_s := e^{  \int_0^s \bigl[ \hat r (X_u)- \hat{r}^\l_u ( X_u)\bigr] du} 
  \prod _{\substack{u \in (0,s]: \\ 
  X_{u-} \not = X_u} }e^{\l g(u,X_{u-} , X_u)} , \]
we aim to show that $\bbE_\nu [Y_t ] =1$. It is enough to prove   that the process $Y:=(Y_s)_{s\in [0,t]}$  is a martingale, since this implies that  
 $\bbE_\nu [Y_t] = \bbE_\nu [Y_0] = 1$. We will divide the proof that $Y$ is a martingale in three parts: firstly we introduce  in \eqref{argentina} a process  $Z:=(Z_s)_{s\in [0,t]}$ and show that it is a  purely discontinuous local martingale, secondly we show that 
 $Y$      is the stochastic exponential of    $Z$   and it is a local martingale; thirdly we prove that   $Y$ is uniformly integrable and therefore it is a martingale. It is only in the last part that we will use Condition $C[\nu,t]$ after performing the  Taylor expansion  $e^{\lambda g(u, X_u , y)} -1 \approx \lambda g(u, X_u , y) $ for $\l$ small (this explains why the condition concerns  the exponential moments on  $g$ and not of $e^{\l g}$).

\smallskip

$\bullet$ The process $Z=(Z_s)_{s\in [0,t]}$ mentioned above is defined as 
	\be \label{argentina} Z_s: = \sum _{\substack{u \in (0,s]} } ( e^{\l g(u,X_{u-} , X_u)} -1) - \int_0^s \bigl[ \hat r^\l_u (X_u)- \hat{r} ( X_u)\bigr] du
	\,. \en
	We  claim that  $Z$  is a purely discontinuous local martingale. 	To prove our claim we take $\alpha (u,x,y) := e^{\l g(u,x,y)} -1 $ and observe that  $\|\a\|_\infty<+\infty$ as $\|g\|_\infty <+\infty$. Since $Z_s=\sum _{u \in (0,s]} \a (u,  X_{u-}, X_u)
	- \int_0^s  \a_r  (u,X_u ) du$ and  by the  discussion at the beginning of Section \ref{sec:purely}, to show that $Z$ is a purely discontinuous local martingale we just need to check that $\a$ satisfies condition \eqref{welldef}. Since $\sum _{s \in (0,t]} |\a (s,  X_{s-}, X_s)| $ can be bounded by $\|\a\|_\infty$ times the total number of jumps in $[0,t]$, and since the latter is $\bbP_\nu$--a.s. finite as the unperturbed process has no explosion $\bbP_\nu$--a.s., we conclude that $\sum _{s \in (0,t]} |\a (s,  X_{s-}, X_s)| <\infty$ $\bbP_\nu$--a.s.. Now it remains to prove that  $\int_0^t |\a |_r  (u,X_u ) du < \infty $ $\bbP_\nu$--a.s..  To this aim we observe that 
\[
\int_0^t |\a |_r  (u,X_u ) du \leq \|\a\|_\infty \int_0^t du \int _\cX r (X_u ,dy ) =\|\a\|_\infty \int_0^t du \, \hat r (X_u)\,.
\]
Since $\bbP_\nu$--a.s.\ the trajectory $(X_u)_{u \in [0 ,t]}$ visits a finite number of states   (again as the  unperturbed process does not explode), the last integral is
  finite $\bbP_\nu$--a.s., thus concluding the check of \eqref{welldef} and therefore the proof of our claim.

\medskip

$\bullet$ We now  show that $Y$    is the stochastic exponential of   $Z$ and it is a local martingale. The first property means that $Y$ 
 is the unique (up to indistinguishability) adapted and c\`adl\`ag solution in $[0,t]$ to the SDE 
	\[ \begin{cases}
	dY_s = Y_{s-} dZ_s \\
	Y_0 = 1, 
	\end{cases} \]
where $Y_{s-} = \lim_{u\nearrow s} Y_u$. 
Indeed, by Theorem 4.61 of \cite{JS}, the stochastic exponential of $Z$ is given for $s \in [0,t]$  by 
	\[ \begin{split} 
	\mathcal{E} (Z)_s & = e^{ Z_s - Z_0 } \prod _{\substack{u \in (0,s]: \\ 
  Z_{u-} \not = Z_u} } ( 1+ \Delta Z_u ) e^{- \Delta Z_u } 
  \end{split} \]
where $\Delta Z_s = Z_{s} - Z_{s-} $ denotes the jump of the process $Z$ at time $s$ (which vanishes if $s$ is not a jump time of the process $X$). Since 
$ \Delta Z_s = e^{\lambda g(s , X_{s-} , X_s ) } -1 $, 
we find 
	\[ \begin{split} 
	 \mathcal{E} (Z)_s & = \exp \Big\{ Z_s +  \sum _{u \in (0,s] } \big[  \lambda g(u, X_{u-} , X_u )   - ( e^{\lambda g(u , X_{u-} , X_u ) } -1 )  \big] \Big\}
  \\ & = \exp \Big\{ - \int_0^s \bigl[ \hat r^\l_u (X_u)- \hat{r} ( X_u)\bigr] du
  + \lambda \sum _{u \in (0,s]} g(u, X_{u-} , X_u ) \Big\} = Y_s 
  \end{split} \]
for all $s \in [0,t]$. 
Thus $Y$ is the stochastic exponential of $Z$. 

Now, since $Z$ is a purely discontinuous local martingale, it follows from \cite{JS}, Theorem 4.61(b) that $Y$ is also a local martingale.

\medskip

$\bullet$ We conclude by showing that the process  $Y$ is in fact a true martingale. Due  to Lemma \ref{le:sufficientmart}  it is enough to show that  $0\leq Y_s\leq \bbY$ for all $s \in [0,t]$ and that  $\bbE_\nu[\bbY]<+\infty$, where 
 \[ \bbY:= \exp \Big\{ 2 \lambda  \int_0^t |g|_r (u,X_u )  du
  + \lambda \sum _{u \in (0,t]} |g(u, X_{u-} , X_u )| \Big\} \,.
  \]
   To check that $0\leq Y_s\leq \bbY$ it is convenient to observe that 
  \[
  Y_s = \exp\Big\{  \int_0^s du \int _\cX r(X_u, dy) (1- e^{\l g(u,X_u,y)})+
      \l \sum _{u \in (0,s]}  g(u,X_{u-} , X_u)  \Big\}\,.
 \]  
As a consequence, for any $s\in [0,t]$, we can bound 
	\[ \begin{split}
	0 \leq  Y_s & \leq 
	\exp \Big\{  \int_0^s \int_\cX r (X_u, dy ) \big| e^{\lambda g(u, X_u , y)} -1 \big|   du 
  + \lambda \sum _{u \in (0,s] } |g(u, X_{u-} , X_u )| \Big\} 
  \\ & \leq 
  \exp \Big\{ 2 \lambda  \int_0^t |g|_r (u,X_u )  du
  + \lambda \sum _{u \in (0,t]} |g(u, X_{u-} , X_u )| \Big\} = \bbY  , 
  \end{split}\]
 for  all $\lambda $ small enough such that $\lambda  \| g\|_\infty \leq 1$ (here we used that $|e^x -1| \leq 2 |x| $ for all $x$ with $|x| \leq 1$). 
 To see that $\bbY$ is integrable we note that 
 	\[ \bbE_\nu [\bbY ] \leq  \bbE_\nu \Big[ \exp \big\{ 4 \lambda \int_0^t |g|_r (s,X_s ) ds \big\} \Big]^{1/2} \cdot  \bbE_\nu \Big[ \exp \big\{ 2\lambda \sum_{s \in (0,t]} |g(s, X_{s-} , X_s ) | \big\} \Big]^{1/2} \]
 by Schwarz inequality. Recall that $g$ satisfies Condition $C[\nu, t]$ with some parameter $\theta >0$. It follows that the first expectation in the right hand side is finite provided $4\l \leq \theta$, while by Lemma \ref{ananas78X} the second expectation in the right hand side is finite provided $2\lambda \leq 4^{-1} \min\{ \theta , \| g\|_\infty^{-1} \}$. All the above constraints on $\lambda$ reduce to $\l \leq 8^{-1} \min\{  \theta  ,\; 1/\|g\|_\infty\}$. 
In this case $\bbY$ is integrable and therefore $Y$ is a martingale. This concludes the proof of Theorem \ref{th:explosionX}.

  %
  %
  %

\section{Proof of Proposition \ref{derivoX}} \label{sec:derivoX}
Trivially, by our assumptions, $F(X_{[0,t]})$ is integrable with respect to $\bbP_\nu$. 

In what follows,  $c,C,..$ will denote an absolute  constant  which can change from line to line. Moreover, $q$ will be the exponent conjugate to $p$, i.e. such that $1/p+1/q=1$.  Note that $q\in [1,+\infty)$.
  Let $\xi_{[0,t]}\in D_f([0,t],\cX)$.
 Recall  \eqref{ugo}:
 	\[ \bbE_{\nu}\Big[ F(X^\l_{[0,t]} )  \Big]=\bbE_{\nu}  \Big[ F(X_{[0,t]}) 
e^{ \cR_\l ( X_{[0,t]} )}\Big] \]
with 
\[ 
\cR_\l( \xi_{[0,s]} \bigr) :=- \cA_\l \bigl( \xi_{[0,s]} \bigr)  = \int _0^t ds \int_\cX r(\xi_s,dy) \left(1- e^{\l g(s, \xi_s,y)}\right)  + \l 
 \sum _{s} g(s, \xi_{s-}, \xi_s)  \,.
\]
From now on we restrict to $\l$ small enough that $\l \|g\|_\infty \leq 1/2$.
 As 
  $|1-e^x +x| \leq  c  x^2 $ for $|x|\leq 1$,    it holds
\[
\int_{\cX} r(\xi_s,dy) \bigl|  1- e^{\l g(s,\xi_s,y)}+\l g(s,\xi_s,y)\bigr|\leq c \l^2 (g^2)_r(s, \xi_s)\leq c \|g\|_\infty  \l^2 |g|_r(s, \xi_s)  \,.
\] 
Hence, we get 
\be\label{pannaX}
\big| \cR_\l(\xi_{[0,t]})-\l G_t (\xi_{[0,t]})\big|\leq c  \|g\|_\infty   \l^2 \int_0^t |g|_r(s, \xi_s) ds\,.
\en
As $|e^z-1-z| \leq z^2 e^{|z|}  $  for all $z\in \bbR$, we get $|e^x-e^y| \leq e^y( |x-y|+|x-y|^2 e^{|x-y|})$ for all $x,y\in \bbR$.  Hence, 
\be\label{madame101X}
|e^x-(1+y)| \leq |e^x-e^y|+ |e^y -(1+y)|\leq  e^{|y|}( |x-y|+|x-y|^2 e^{|x-y|}+ y^2 ).
\en
Take  now
\[x:= \cR_\l(X_{[0,t]}) \qquad  \text{ and } \qquad y:=\l G_t(X_{[0,t]})\,.\] As $F \bigl( X_{[0,t] } \bigr)\in L^p(\bbP_\nu)$, by H\"older's inequality and \eqref{ugo}, 
we get
\begin{align}
& \bbE_\nu  \bigl[\, \bigl| F( X^\l_{[0,t]}  ) \bigr|\, \bigr] \leq  \| F( X_{ [0,t]}  )\|_{L^p(\bbP_\nu) } \| e^x\|_{L^q(\bbP_\nu) }\,,\\
&  \bbE_\nu \bigl[\,  \bigl|  F( X_{[0,t]}  )  G_t( X_{[0,t]}  )   \bigr|\,\bigr]  \leq  \| F( X_{ [0,t]}  )\|_{L^p(\bbP_\nu) } \| y/\l \|_{L^q(\bbP_\nu) }\,,\\
& \big| \bbE_\nu  \bigl[\, F( X^\l_{[0,t]}  ) \bigr]  - \bbE_\nu
 \bigl[   F( X_{ [0,t]}  )  \bigr]- \l \bbE_\nu \bigl[   F( X_{[0,t]}  )  G_t( X_{[0,t]}  )  \, \bigr]  \big |\nonumber
\\
& \qquad   =\big| \bbE_\nu \bigl[ F( X_{ [0,t]}  ) \big( e^x-(1+y) \big) \bigr]|  \leq  \| F( X_{ [0,t]}  )\|_{L^p(\bbP_\nu) }  \| e^x-(1+y)\|_{L^q(\bbP_\nu) } 
\,.\label{nuvoletta}
\end{align}
Hence 
to get that all expectations in Proposition \ref{derivoX} are well defined and finite it is enough to prove that $x$, $y$ belong to  $L^q(\bbP_\nu)$, while   to get \eqref{titans2X} it is enough to prove that   the r.h.s.\ of \eqref{madame101X}    has  norm in $L^q(\bbP_\nu)$ bounded by $o(\l)$.
In what follows we focus on the last claim, the proof that $x,y \in L^q(\bbP_\nu)$ can be obtained by similar arguments.

As $g $ is bounded and it satisfies Condition $C[\nu , \l ]$,  by Lemma \ref{ananas78X} we get that $G_t  (X_{[0,t]})$ is upper bounded by the  sum of two non-negative terms, namely $\int_0^t |g|_r (s, X_s) ds$ and $\sum_s |g(s, X_{s-},X_s)|$, each one having finite exponential moment when multiplied by  a suitable small constant (independent from $\l$). By applying Schwarz inequality we then conclude that 
for any $a \in [1,\infty )$ there exists $\l_0(a) < \infty $ such that $e^{|y|}= e^{\l |G_t(X_{[0,t]})| }$ belongs to $L^a ( \bbP_\nu )$ for all $ \l \in [0,\l_0(a) ]$, and moreover 
\be\label{graz}
\sup_{\l\leq \l_0(a)} \| e^{ |y| }\|_{L^a(\bbP_\nu)} <+\infty\,.
\en
 In addition,  since $g$ satisfies Condition $C[\nu , \l ]$ we have that  $\int_0^t |g|_r(s, X_s) ds$   belongs to $L^a(\bbP_\nu)$ for any $a\in[1,+\infty)$.
Moreover,   since  $\l^{-2}  |x-y| \leq c\|g\|_\infty  \int_0^t |g|_r(s, X_s) ds$  (cf. \eqref{pannaX}),  using \eqref{graz} and Schwarz inequality we conclude that \be\label{elisa20}
\sup_{\l \leq \l_0(2q)}  \|  \, \l^{-2}  |x-y| e^{|y|}\, \| _{L^q (\bbP_\nu) } <+\infty\,.
  \en
 
By the same arguments  based on \eqref{pannaX} we also have that   $e^{|x-y|}$ belongs  to $L^a(\bbP_\nu)$ for any $a \in[1,+\infty)$ and  $\l  \leq \l_1(a)$ for some $\l_1(a)>0$,  with 
  \be \label{grazbis}
\sup_{\l \leq \l_1(a)} \| e^{ |x- y| }\|_{L^a(\bbP_\nu)} <+\infty\,.
\en
 Hence, 
 using \eqref{graz}, \eqref{grazbis} and Schwarz inequality, we gather that 
\be\label{elisa21}
\sup_{\l \leq \l_0(4q)\wedge \l_1(4q) }  \|  \,  e^{|y|} e^{|x-y|} \, \| _{L^{2q} (\bbP_\nu) } <+\infty \,.
  \en
By \eqref{pannaX} and the previous observations on $\int_0^t |g|_r(s, X_s) ds$, we get that $\l^{-4} |x-y|^2  $
belongs to $L^{2q}( \bbP_\nu)$ and the norm can be bounded by a $\l$--independent constant.  As a byproduct of 
\eqref{elisa21} and Schwarz inequality, we get that 
\be\label{elisa22}
\sup_{\l \leq \l_0(4q)\wedge \l_1(4q) }  \|  \,\l^{-2} |x-y|^2  e^{|y|} e^{|x-y|} \, \| _{L^{q} (\bbP_\nu) } <+\infty \,.
  \en
 As  $e^{\l_0(1)  |G_t(X_{[0,t]})| }$ belongs to $L^1(\bbP_\nu)$ by \eqref{graz},  we  get that $\l^{-1} y=G_t(X_{[0,t]})$  belongs to $L^a(\bbP_\nu)$ for any $a \in[1,+\infty)$. By taking $a=2q$, by \eqref{graz} and Schwarz inequality, we conclude that 
 \be\label{elisa23}
\sup_{\l \leq \l_0(2q)  }  \|  \,\l^{-2} y^2 e^{|y|}  \, \| _{L^{q} (\bbP_\nu) } <+\infty \,.
  \en
 By combining \eqref{elisa20}, \eqref{elisa22} and \eqref{elisa23} we conclude that the r.h.s.\ of \eqref{madame101X}  has norm in $L^q(\bbP_\nu)$ bounded by $\l^2$ times a $\l$--independent constant.  Hence 
 the r.h.s. of \eqref{nuvoletta} is upper bounded by $C  \| F( X_{ [0,t]}  )\|_{L^p(\bbP_\nu) }  \l  ^2$ for $\l$ small enough.

%
%

%
%

\section{Proof of Theorem \ref{th:JS}} \label{sec:JS1} 
Using that the expectations in the statement of Proposition \ref{derivoX} are well defined and finite and using the bounds in Section \ref{sec:derivoX} as well as the bounds below, it is easy  to prove that expectations in the statement of Theorem \ref{th:JS} are well defined and finite.

The result for case (1) follows directly from \eqref{titans2X} in Proposition \ref{derivoX}. We use it to deduce the linear response formula for case (2). Indeed, by Fubini's theorem,
	\be\label{mango100} \begin{split} 
	&\partial_{\l =0} \bbE_\nu \bigg[  \int_0^t  v(s,X_s^\l ) ds  \bigg] 
	 \\& = \partial_{\l=0} \int_0^t  \bbE_\nu [ v(s,X_s^\l ) ] ds 
	 = \lim_{\l \to 0 } \int_0^t \frac{\bbE_\nu [v(s,X_s^\l )] - \bbE_\nu [v(s,X_s)]}{\l} ds 
	\\ & = \int_0^t \bbE_\nu [ v(s,X_s) G_s (X_{ [0, s]}) ] ds 
	\\ & \quad + \lim_{\l \to 0 } \int_0^t \bigg( \frac{\bbE_\nu [v(s,X_s^\l)] - \bbE_\nu [v(s,X_s)]}{\l} - \bbE_\nu [ v(s,X_s) G_s (X_{ [0, s]}) ] \bigg) ds .
	\end{split}  
	\en
Then, by the last statement in Section \ref{sec:derivoX} applied when  $\| v(s , X_s ) \|_{L^p(\bbP_\nu )} <+\infty$,	we have that for all $s \in [0,t] $ 
	\[ \bigg| \frac{\bbE_\nu [v(s,X_s^\l)] - \bbE_\nu [v(s,X_s)]}{\l} 
	- \bbE_\nu [ v(s,X_s) G_s (X_{ [0, s]}) ] \bigg| 
	\leq C \l \| v(s , X_s ) \|_{L^p(\bbP_\nu )} 
	\]  
which, together with the assumption $\int_0^t \| v(s , X_s ) \|_{L^p(\bbP_\nu )} ds < \infty $, implies that the last term in the chain of equalities \eqref{mango100} vanishes, thus proving the required identity.


	\smallskip 
We now move to case (3). Since 
$\a$ is  $\bbP_\nu$--integrable,  
it satisfies condition \eqref{welldef} $\bbP_\nu$-almost surely. Hence we can use the stochastic calculus techniques for processes with jumps presented in 
Section \ref{sec:SC}. 
Write $G_s $ in place of $G_s (X_{ [0,s]})$, and for $s \in [0,t]$ set
\[
 F_s:= \sum _{u \in (0,s]} \alpha (u,  X_{u-}, X_u)\,.
 \]Note that $F_t= F( X_{[0,t]})$.  Since $F  (X_{ [0,t]}) \in L^p ( \bbP_\nu ) $ for some $p>1$, to get \eqref{res3} we can apply \eqref{titans2X},  hence we just need to show that the r.h.s. of \eqref{res3} equals $\bbE_\nu [ G_t F_t ]$. 

To compute $\bbE_\nu [ G_t F_t ]$ we start by noticing that, since $g$ satisfies condition $C[\nu,t]$, $(G_s)_{s \in[0, t]}$ is a purely discontinuous martingale \rosso{by Lemma \ref{L:true}}. 

Next, we compensate $(F_s)_{s\in [0, t]}$ to make it into a purely discontinuous martingale.
By the $\bbP_\nu$--integrability assumption on $\alpha$, 
we can define 
\[
\bar F_s:= F_s- \int_0^s \int_\cX \alpha (u,X_u , y) r(X_u ,dy) du=F_s- \int_0^s  \alpha_r (u,X_u ) du\,,
\]
and $(\bar F_s)_{s\in [0,t]}$ is a purely discontinuous martingale \rosso{again by Lemma \ref{L:true}}.  
Recall from Section \ref{sec:purely} that 
the covariation process of $G$ and $\bar F$ is given by 
	\[ [G , \bar F ]_s = \sum_{u\in (0,s]  } \alpha (u,X_{u-} , X_u ) g (u,X_{u-} , X_u ) , \]
which is well defined and integrable since $g$ is bounded and $\a$ is $\bbP_\nu $--integrable by \eqref{aa}. 
Then by Proposition 4.50 of \cite{JS}  the process $(G_s \bar F_s - [G , \bar F]_s )_{s \in [0,t]}$ defines a  \rosso{local martingale. We claim that it is a true martingale.  Indeed}, since $g$ is bounded and it satisfies Condition $C[\nu, t]$ (and therefore also \eqref{kinney2} in Lemma \ref{ananas78X}), the assumptions \eqref{aa} on $\a$ together with H\"older's inequality imply that the product
	\[ 
	\bigg( \sum _{s \in (0,t]} |g (s,  X_{s-}, X_s)| + \int_0^t | g |_r (s,X_s )ds
	\bigg) 
	\bigg( \sum _{s \in (0,t]} |\alpha (s,  X_{s-}, X_s)| + \int_0^t | \alpha |_r (s,X_s )ds
	\bigg) 
	\] 
belongs to $ L^1 ( \bbP_\nu )$. It thus follows from Lemma \ref{le:sufficientmart} 
that $(G_s \bar F_s - [G , \bar F]_s )_{s \in [0,t]}$ defines a \rosso{true martingale, thus proving our claim}. 
%
%
\rosso{As a consequence} 
	\be\label{ganger} \begin{split} 
	\bbE_\nu [G_t \bar F_t ] & = \bbE_\nu [ [G , \bar F ]_t ]
	 \\ & = \bbE_\nu \bigg[ \sum_{s\in (0,t] } \alpha (s,X_{s-} , X_s ) g (s,X_{s-} , X_s ) \bigg] 
	 = \int_0^t \bbE_\nu \big[ (\alpha g )_r (s,X_s ) \big] ds , 
	\end{split} \en
where in the second identity we have used that 
	\[ \sum_{u\in (0,s] } \alpha (u,X_{u-} , X_u ) g (u,X_{u-} , X_u )
	- \int_0^s  (\alpha g)_r (u,X_u  )  du \]
defines a martingale for $s\in [0,t]$, as it is of the form \eqref{E:M} and $\a g $ is $\bbP_\nu$--integrable (\rosso{see Lemma \ref{L:true}}).  
To finish the computation of $\bbE_\nu [G_t F_t ]$ we  observe that, by Fubini and the fact that $(G_s)_{s\in[0,t]}$ is a martingale, 
	\be\label{myu} \bbE_\nu \bigg[ G_t \int_0^t  \alpha_r (u,X_u ) du \bigg] = 
	\int_0^t\bbE_\nu \big[\alpha_r (s,X_s ) G_t \big] ds
	= \int_0^t\bbE_\nu \big[\alpha_r (s,X_s ) G_s \big] ds . \en
Putting together \eqref{ganger} and \eqref{myu}, we get 
	\[ \bbE_\nu [G_t F_t ] 
	= \int_0^t \bbE_\nu \big[ (\alpha g )_r (s,X_s ) \big] ds
	+  \int_0^t\bbE_\nu \big[\alpha_r (s,X_s ) G_s \big] ds , \]
which concludes the proof of Theorem \ref{th:JS}.

%
%

\section{Proof of Theorem   \ref{cor:JS}} \label{sec:JS2} 
The decoupled case follows easily from the general case, hence we focus on the first part of the theorem.
We aim to compute the r.h.s. of \eqref{res1}, \eqref{res2} and \eqref{res3}
 in cases $(1)$, $(2)$ and $(3)$ in Theorem \ref{th:JS}.
We achieve this by performing a time-inversion of the unperturbed process.
Recall the definition of the time--reversed process  $(X_s^*)_{s\in [0,t]}$ given in Section \ref{sec:stationary}. 
In particular, we use the  following
 equality in distribution  (valid for any $s\in [0,t]$)
	\begin{equation}\label{eqdistr}
	 \big( X_s , G_s ( X_{ [0,s]} )\big)  
	 \stackrel{\cL}{=} \big( X_0^* , G_s^* ( X^*_{ [0,s]} ) \big)  
	 \end{equation}
by defining 
	\[  G_s^* ( \xi_{[0,s]} ) := 
	\sum_{u \in (0,s] : \xi_{u-} \neq \xi_u } g^*(s-u , \xi_{u-} , \xi_u ) - 
	\int_0^s g_r (s-u , \xi_u) du\, .  \]

Let us consider first case (1). From \eqref{eqdistr} it follows that 
	\be\label{primavera95}
	\bbE_\pi [ v(X_t) G_t (X_{ [0,t]}) ] 
	 = \bbE_\pi [ v(X_0^*) G_t^* (X^*_{[0,t]} ) ] 
	 = \bbE_\pi [ v(X_0^* ) \bbE_{X_0^*} ( G_t^* (X^*_{ [0,t]} ) )]\, .
	\en
We claim that the process 
	\be\label{rosario}  [0,t] \ni s \mapsto \sum_{u\in (0,s] : X^*_{u-} \neq X^*_u} g^* (t-u , X^*_{u-} , X^*_u) - 
	\int_0^s g^*_{r^*} (t-u , X_u^* ) ds \en
defines a martingale for the probability measure $\bbP_x$ and for $\pi$--a.a.~$x\in \cX$, where $g^*_{r^*}$ denotes the contraction of $g^*$ as in \eqref{def:contractionX}, with respect to the transition kernel $r^*(x,dy)$ in place of $r(x,dy)$. Note that, with some abuse of notation, we have written $\bbP_x$ for the probability referred to the time-reserved unperturbed process starting at $x$.
To prove our claim, we observe that  the above process  \eqref{rosario} defines a local martingale since it is of the form \eqref{E:M}. On the other hand, 
by time--inversion and using Lemma \ref{giacinto} (recall that $g$ satisfies Condition $C[\pi,t]$), we have
\[
 \bbE_\pi \Big[ \sum_{u\in (0,t] : X^*_{u-} \neq X^*_u} |g^* (t-u , X^*_{u-} , X^*_u) | \Big]= 
\bbE_\pi\Big [\sum_{u\in (0,t] : X_{u-} \neq X_u}| g(u , X_{u-} , X_u)| \Big ]<+\infty\,.
\]
As a consequence $ \bbE_x \Big[ \sum_{u\in (0,t] : X^*_{u-} \neq X^*_u} |g^* (t-u , X^*_{u-} , X^*_u) | \Big] \rosso{<+\infty}$ for 
$\pi$--a.a.~$x\in \cX$, thus implying that the process \eqref{rosario} is a martingale for  $\bbP_x$ and   for 
$\pi$--a.a.~$x\in \cX$  as explained in Section \ref{sec:SC} (now referred to the time-reversed unperturbed stationary process).

Due to the above claim,  for 
$\pi$--a.a.~$x\in \cX$, 
	\[ \bbE_x \big[  G_t^* (X^*_{ [0,t]} ) \big] 
	= \bbE_x \bigg[ \int_0^t \big( g_{r^*}^* (t-u , X_u^* ) - g_r (t-u , X_u^* ) \big) du \bigg] \,, \]
from which we gather that  (cf.~\eqref{primavera95})
	\begin{equation}\label{eq:suffices}
	\begin{split} 
	 \bbE_\pi [ v(X_t) G_t (X_{ [0,t]}) ] 
	& = \int_0^t \bbE_\pi \big[ v(X_0^*) \big( g_{r^*}^* (t-u , X_u^* ) - g_r (t-u , X_u^* ) \big) \big] du  
	\\ & = \int_0^t \bbE_\pi \big[ v(X_t) \big( g_{r^*}^* (t-u , X_{t-u} ) - g_r (t-u , X_{t-u} ) \big) \big] du \,,
	\end{split}
	\end{equation}
where the second equality follows from \eqref{eqdistr}. 
This concludes the proof of case (1).
	
The result for case (2) follows by combining \eqref{res2} in Theorem \ref{th:JS} and  
\eqref{eq:suffices} with $v(s , \cdot ) $ and $s$  in place of $v(\cdot )$ and $t$, respectively, giving 
 \begin{multline*}
	\int_0^t ds\, \bbE_\pi\bigl[\rosso{v(s,X_s)} G_s(X_{[0,s]}) \bigr]  =\\
	 \int_0^t ds 
	\int_0^s du \,
	 \bbE_\pi \big[ v(s, X_0^*) \big( g_{r^*}^* (s-u , X_u^* ) - g_r (s-u , X_u^* ) \big) \big]
	\\  = 
	 \int_0^t ds 
	\int_0^s du \, 
	 \bbE_\pi \big[ v(s, X_s) \big( g_{r^*}^* (s-u , X_{s-u} ) - g_r (s-u , X_{s-u} ) \big) \big].
\end{multline*}

For case (3), in light of \eqref{res3} in Theorem \ref{th:JS}, it will suffice to show that for all $s\leq t$ it holds 
	\be\label{posterity} 
	\begin{split} 
	\bbE_\pi [\alpha_r (s,X_s ) G_s ] & = 
	 \int_0^s \bbE_\pi \Big[ \alpha_r (s, X_0^* )
	  \big( g^*_{r^*} (s-u , X_u^* ) - g_r (s-u , X_u^* ) \big) \Big] du 
	  \\ & = 
	  \int_0^s \bbE_\pi \Big[ \alpha_r (s, X_s )
	  \big( g^*_{r^*} (s-u , X_{s-u} ) - g_r (s-u , X_{s-u} ) \big) \Big] du .
	  \end{split} 
	 \en
The derivation of \eqref{posterity} is identical to the proof of \eqref{eq:suffices}  (with $t$ replaced by $s$) and uses the time-inversion identity \eqref{eqdistr}.

\section{Time periodic case: Proof of  Lemma \ref{cop25}, Lemma \ref{apogeo} and  Theorem \ref{alpha_omega}}\label{tortelli}

\begin{proof}[Proof of Lemma \ref{cop25}] Since $r^*(x,y)>0$ whenever $r(y,x)>0$, Assumption \ref{semplice} implies the irreducibility of the Markov jump process with generator $\cL^*$, and this is equivalent to the fact that zero is a simple eigenvalue of $\cL^*$ (trivially the non-zero  constant functions are the associated eigenvectors).

Let us move to the other complex eigenvalues.  Write $f\in L^2(\pi)$ as $f= f_R + i f_I$, where $f_R,f_I$ are real functions. Then
 we have 
 $
\Re \bigl( \la f, \cL^* f \ra \bigr)= 
\la f_R, \cL^* f_R\ra + \la f_I , \cL^* f_I\ra $, $\Re(\cdot)$ denoting the real part.
As  for real functions $g$ we have $\la g, \cL^* g\ra= \la \cL g,  g\ra= \la g, \cL g\ra $ we conclude that 
$\Re \bigl( \la f, \cL^* f \ra \bigr)=  \la f_R , S f_R\ra + \la f_I, S f_I\ra $, where $S= (\cL+\cL^*)/2$. As $S g(x)=\sum_y r_S(x,y) [g(y)-g(x)]$ with $r_S(x,y)= ( r(x,y)+ r^*(x,y))/2$, we find that $S$ itself is the Markov generator of a Markov jump process on $\cX$ with rates $r_S(x,y)$ which are easily seen to satisfy  detailed balance w.r.t.\ $\pi$. We therefore get 
\be \label{ucraina}
\la g, -S g\ra = \frac{1}{2}\sum_x \sum _y \pi(x) r_S(x,y) [ g(y)-g(x) ]^2 \geq 0 \qquad g: \cX \to \bbR\,.
\en
Moreover, since $r_S(x,y)>0$ if 
 $r(x,y)>0$, also  $S$ is irreducible. This implies that $\la g, -S g\ra$ in \eqref{ucraina} is zero if and only if $g$ is constant, and otherwise it is strictly positive. Putting all together, we conclude that $\Re \bigl( \la f, \cL^* f \ra \bigr)< 0$ for any $f:\cX \to \bbC$  which is not constant. 
 Now let $f$ be an eigenvector of $\cL^*$ with eigenvalue $\l \not =0$. We have $\la f, \cL^* f\ra= \l \| f \|^2$.  Hence,   $\Re \bigl( \la f, \cL^* f \ra \bigr)= \Re (\l)\|f\|^2$.  As $f$ is not constant (otherwise we would have  $\l=0$), we conclude that  
$0>\Re \bigl( \la f, \cL^* f \ra \bigr)/\|f\|^2=\Re (\l)$.
\end{proof}

\begin{proof}[Proof of Lemma \ref{apogeo}] Recall that we consider $a,\psi_t$ as column vectors, while we consider $\pi, \pi_\l,\dot\pi$ as row vectors. We write  $A^\t$ for the transpose of a matrix  $A$ and we denote by $D$ 
the diagonal matrix with \rosso{diagonal} $x$--entry given by $\pi(x)$.  Letting $P_T^*:=e^{T \cL^*}$, we have  $(P_T^*)_{x,y}= \bbP_x(X_T^*=y) = (P_T)_{y,x} \pi(y)/\pi(x)$. 
  In particular it holds $P_T^\t= D P_T^* D^{-1}$
and $\dot \pi ^\t=D a $, thus implying that 
 \be \label{coca1}
  \left( \dot \pi (P_T-\bbI)\right)^\t= D (P_T^*-\bbI) a\,.\en
 On the other hand, by Theorem \ref{cor:JS}, time-inversion and the $T$--periodicity of $\psi_s$, we have 
 \be\label{coca2}
 \begin{split}
 (\pi \dot{P}_T)(x) & = \sum_{y} \pi(y) \partial _{\l=0} \bbP_y ( X_T^\l =x) =\partial_{\l=0} \bbE_\pi\big[\mathds{1}_{\{ X^\l_T=x\}}\big]\\
 &= \int _0^T ds \bbE_\pi \big[ \mathds{1}_{\{X_0^*=x\}} \psi_{T-s} (X_s^*)
 \big]
 =\pi(x) \Big(\int _0^T ds \, e^{s \cL^*} \psi_{T-s} \Big)(x)\\
 &=\pi(x) \Big(\int _0^T ds \, e^{s \cL^*} \psi_{-s} \Big)(x) \, . 
 \end{split}
 \en
 Hence, rewriting the members in  \eqref{pianino}  as \eqref{coca1} and \eqref{coca2}, we have  $ D(P_T^*-\bbI) a = -D  \int _0^T ds\,e^{s \cL^*} \psi_{-s} $. We therefore conclude that $a\in L^2_0(\pi)$ solves the equation in $\a$
 \be\label{uva}(P_T^*-\bbI) \a =  - \int _0^T ds\,e^{s \cL^*} \psi_{-s} \qquad \a \in L^2_0(\pi)\,.
 \en
 As  $(P_T^*-\bbI)$ is injective on $L^2_0(\pi)$ (recall that $0$ is a simple eigenvalue of $\cL^*$), we have that the solution in $L^2_0(\pi) $ of the above equation \eqref{uva} is unique. Since  $ \int _0^\infty ds  \,e^{s \cL^*}\psi_{-s}$ belongs to $L^2_0(\pi)$, to conclude the proof it remains 
 to check that  $\a := \int _0^\infty ds  \,e^{s \cL^*}\psi_{-s}$  solves \eqref{uva}. By the $T$-periodicity of $\psi _s$, we have 
 \be\label{molluschi}
 \begin{split}
 \int _0^\infty ds  \,e^{s \cL^*}\psi_{-s}=\int _0^T ds \sum_{k=0}^\infty e^{(s+ kT)  \cL^*} \psi_{-s}=\bigl[\sum_{k=0}^\infty e^{ kT  \cL^*}\bigr] \int_0^T ds\,e^{s \cL^*}\psi_{-s}\,.
 \end{split}
 \en
Since $e^{T \cL^*}= P_T^*$, we gather that $(P_T^*-\bbI) \sum_{k=0}^\infty e^{ kT  \cL^*}= - \bbI$ on $L^2_0(\pi )$. This observation and \eqref{molluschi} imply that $\a = \int _0^\infty ds  \,e^{s \cL^*}\psi_{-s}$  solves \eqref{uva}. 
 \end{proof}
\begin{proof}[Proof of Theorem \ref{alpha_omega}]  Since  $\cX$ is finite (and therefore $\sup_{x\in \cX} \hat r(x)<+\infty$)  and $g$ is bounded, $g$ satisfies Condition $C[\pi,T]$.\\
$\bullet$ {\bf Proof of \eqref{kiriku1}}.
We have 
$\bbE_{\pi_\l}  [v(X^\l_t ) ]=
\sum _x \pi (x) \frac{\pi_{\l,t} (x)}{\pi(x)} v(x)$, hence  $\partial_{\l=0} \bbE_{\pi_\l}  [v(X^\l_t ) ]
= \pi [ \la a_t, v \ra ]$, i.e.\ (by Corollary \ref{superapogeo}) $\partial_{\l=0} \bbE_{\pi_\l}  [v(X^\l_t ) ]= \int _0^\infty ds  \, \la v, e^{s \cL^*}\psi_{t-s}\ra$, which allows to conclude.\\
$\bullet$ {\bf Proof of \eqref{kiriku2}}. By \eqref{kiriku1} it is enough to show that $\partial_{\l=0} \bbE_{\pi_\l}  \Big[\int_0^t v(s, X^\l_s)ds \Big]=\int_0^t  \partial_{\l=0} \bbE_{\pi_\l}  \Big[ v(s, X^\l_s) \Big]ds $. 
To this aim we observe that, by Fubini's theorem,
	\begin{equation}\label{exchange}
	 \begin{split} 
	\partial_{\l=0} \bbE_{\pi_\l}  \Big[\int_0^t v(s, X^\l_s)ds \Big] 
	& =  \partial_{\l=0} \int_0^t \bbE_{\pi_\l}  \big[ v(s, X^\l_s) \big] ds
	\\ & = \lim_{\l \to 0} \int_0^t \frac{ \bbE_{\pi_\l}  \big[ v(s, X^\l_s) \big]  - \bbE_{\pi}  \big[ v(s, X^\l_s) \big]  }{\l } \, ds 
	\\ & \quad + \partial_{\l=0} \bbE_{\pi}  \Big[\int_0^t v(s, X^\l_s)ds \Big]  . 
	\end{split}
	\end{equation}
Note that, since $\cX$ is finite, the assumption $\int_0^t | v(s,x) | ds < \infty $ for all $x \in \cX$ easily implies that $\int_0^t \| v(s, X_s ) \|_{L^p (\bbP_\pi )} ds < \infty $ for all $p>1$. Thus in the last term of \eqref{exchange} the derivative can be exchanged with the integration  as follows by comparing \eqref{res1} and \eqref{res2} in  Theorem \ref{th:JS}. The term in the middle line of \eqref{exchange} equals 
	\be\label{mandi} \begin{split} \sum_{x \in \cX} & \lim_{\l \to 0} \int_0^t v(s,x)  \frac{\bbP_{\pi_\l} (X_s^\l =x )  - \bbP_{\pi} (X_s^\l =x ) }{\l} ds 
	\\ & = \sum_{x \in \cX}  \sum_{y \in \cX}  \pi (y)  \lim_{\l \to 0} \left[\frac 1
	\l \Big( \frac{\pi_\l (y)}{\pi (y) } -1 \Big) \int_0^t v(s,x) \,  \bbP_y ( X_s^\l =x) ds \right]
	\\ & = \sum_{x \in \cX}  \sum_{y \in \cX}  \pi (y) a(y) \int_0^t v(s,x) \bbP_y ( X_s =x) ds 
	, 
	\end{split} 
	\en
where in the last equality we have used the dominated convergence theorem to argue that $\lim_{\l \to 0} \int_0^t v(s,x) \bbP_y (X_s^\l =x ) ds = \int_0^t v(s,x) \bbP_y (X_s =x ) ds$, 
since by assumption  $\int_0^t | v(s,x) | ds < \infty$ for all $x \in \cX$  and   by Theorem  \ref{th:JS}  $\bbP_y(X_s^\l=x)$ is differentiable (and therefore continuous) at $\l=0$. Reasoning  as done for  \eqref{mandi} (but without the use of the dominated convergence theorem), we get that the last expression in \eqref{mandi}  equals $ \int_0^t  \lim_{\l \to 0} \frac{ \bbE_{\pi_\l}  \big[ v(s, X^\l_s) \big]  - \bbE_{\pi}  \big[ v(s, X^\l_s) \big]  }{\l } ds$.

Since we have been able to exchange
the limit with the integral in the term in the middle line of \eqref{exchange} and to exchange the derivative 
 with the integral in the last  term of \eqref{exchange}, we conclude that
	\[  \partial_{\l=0} \bbE_{\pi_\l}  \Big[\int_0^t v(s, X^\l_s)ds \Big]=\int_0^t  \partial_{\l=0} \bbE_{\pi_\l}  \Big[ v(s, X^\l_s) \Big]ds \]
as required.
$\bullet$ {\bf Proof of \eqref{kiriku3}}.
We now focus  on 
$\partial_{\l=0} \bbE_{\pi_\l}  \big [\sum_{s\in (0,t]}  \a(s, X^\l_{s-}, X^\l_s ) \big ]$.

Generalizing 
\eqref{def:contractionX}, we set     $\b _{r^\l_s} (s,  x):= \sum_{y \in \cX} \b(s,x,y) r^\l _s(x,y)  $ for any $\b: [0,t]\times \cX\times \cX\to \bbR$. We claim that the process  $[0,t]\ni s \mapsto \sum_{u\in (0,s]}  \a(u, X^\l_{u-}, X^\l_u )  - \int_0^s 
\a_{r^\l_u }  (u, X^\l_{u} )  du \in \bbR$ defines a martingale w.r.t. $\bbP_{\pi_\l}$ for all $\l$. To prove our claim, we think  of  the process  $(\xi_s)_{s \in[0, t] }$, where  $\xi_s:=(s, X_{s-}^\l, X_s^\l)$ and $X^\l_{0-}:=X^\l_0$, as a PDMP  with  state space $\bbR\times \cX\times \cX$ and with the following local characteristics  \cite[Section~24]{D}:
the 
  jump intensity rate at $(s,x,y)$ is  given by $ \hat{r}^\l_s(y)=\sum_{z\in \cX} r_s^\l(y,z)$, the probability transition kernel
equals $Q( (s,x,y), \cdot )= \sum _{z\in \cX}\big( r^\l_s(y,z)/\hat r^\l_s(y)\big) \d_{(s,y,z)}$ and  the vector fields on $\bbR$ associated to each $(x,y)\in \cX \times \cX$ are given by the unit vector field. Then the claim follows from Item 2 of  \cite{D}[Theorem~(26.12)] applied to the process $(M_s^\a)_{ s\in[0,t] }$ defined therein, since  the integrability condition in the above cited theorem reduces to 
$\bbE_{\pi_\l} [ \sum_{s\in (0,t] }|\a(s, X^\l_{s-}, X^\l_s)|] <+\infty$.
 Due to  Item 1 of  \cite{D}[Theorem~(26.12)]  the above bound is equivalent to the bound $\bbE_{\pi_\l} [ \int_0^t |\a|_{r^\l_s} (s,  X^\l_s)] <+\infty$.    This last bound is fulfilled
 since  the expectation inside can be bounded by $e^{\l \|g\|_\infty}\sum _{x\in \cX}  \int_0^t |\a|_r(s,x)ds$, which is finite by our assumptions. 

Due to the above  claim we find
	\begin{equation}\label{eq:chain}
	 \begin{split}
	\partial_{\l=0} \bbE_{\pi_\l}  \Big [\sum_{s\in (0,t]}  \a(s, X^\l_{s-}, X^\l_s ) \Big ] 
	& = \partial_{\l=0} \bbE_{\pi_\l}  \Big [ \int_0^t \a_{r^\l_s }  (s, X^\l_{s} )  ds \Big] 
	\\ & = \partial_{\l=0} \int_0^t \bbE_{\pi_\l}  \big [  \a_{r^\l_s }  (s, X^\l_{s} ) \big]  ds 
	\\ & = \partial_{\l=0}  \int_0^t \sum_{x \in \cX } \pi_{\l , s } (x) 
	\sum_{y \in \cX } \a (s,x,y) r^\l_s  (x,y) ds .
	\end{split}
	\end{equation}
Similarly to \eqref{exchange}, to see that in the last term of \eqref{eq:chain} the derivative can be taken inside the sign of integration we proceed as follows. Since  $ \pi_{\l , s } (x) =\sum_{z\in \cX}\pi_\l(z) \bbP_z( X^\l_s=x)$ and $ \pi(x) =\sum_{z\in \cX}\pi(z) \bbP_z( X_s=x)$, we can rewrite the last term of \eqref{eq:chain}  as the sum of the following three terms:
\begin{align*}
&A:=\sum_{x,y,z\in \cX} \lim_{\l \to 0}  \int_0^t   \Big[ \frac{\pi_\l(z) -\pi(z)}{\l}    \bbP_z( X^\l_s=x)e^{\l g(s,x,y)}\a(s,x,y)r(x,y)ds\Big]\,,\\
&B:= \sum_{x,y,z\in \cX} \lim_{\l \to 0}   \int_0^t   \pi(z) \frac{ \bbP_z( X^\l_s=x)- \bbP_z( X_s=x)}{\l}e^{\l g(s,x,y)}\a(s,x,y)r(x,y)ds\,,\\
& C:= \sum_{x,y,z\in \cX} \lim_{\l \to 0}   \int_0^t   \pi(z)  \bbP_z( X_s=x) \frac{e^{\l g(s,x,y)}-1}{\l} \a(s,x,y)r(x,y)ds\,.
\end{align*}
For all terms $A,B,C$ we get that they remain  unchanged if we move the limit $\lim_{\l\to 0}$ inside the time integral. This can be achieved as follows.
To deal  with  $A$, we take $ \frac{\pi_\l(z) -\pi(z)}{\l} $ outside the time integral,  we use  that  $\lim_{\l \to 0}   \frac{\pi_\l(z) -\pi(z)}{\l} =\pi(z) a(z)$ and we apply the dominated convergence theorem to get  the limit of  the remaining time integral. Indeed,  the remaining  integrand is bounded for, say, all $\l \in [0,1]$, by 
$ e^{ \| g \|_\infty} |\a |_r ( \cdot ,x)  $, which is integrable on $[0,t]$  by assumption.  To deal with $B$ we use that $\bbP_z( X^\l_s=x)$ differs from its first-order expansion 
$\bbP_z( X_s=x) +\l \bbE_z[ \mathds{1}_{\{X_s=x\}} G_s(X_{[0,s]})]$
by at most $c \l^2$,  where $c$ is a constant  independent from $z$ and $s$ (this  follows from the last statement concerning \eqref{nuvoletta} in Section \ref{sec:derivoX}). We then apply the dominated convergence theorem (we use again that $|\a| _r ( \cdot ,x)  $  is integrable on $[0,t]$  and we bound $\bbE_z[ \mathds{1}_{\{X_s=x\}} G_s(X_{[0,s]})]$ by $\|g\|_\infty (t+\bbE_z[N_t])<+\infty$, $N_t$ being the total number of jumps in the time interval $[0,t]$). To deal with $C$ we just apply the  dominated convergence theorem.

As commented above,  all terms $A,B,C$ remain  unchanged if we move the limit $\lim_{\l\to 0}$ inside the time integral. This allows us to conclude that in the last term of \eqref{eq:chain} the derivative can be taken inside the sign of integration. As a consequence, this term  equals
	\[ \sum_{x \in \cX} \pi (x) \sum_{y \in \cX } \int_0^t \a (s,x,y) 
	  \partial_{\l=0}  \Big( \frac{\pi_{\l , s} (x)}{\pi (x)} \, r^\l_s  (x,y) \Big) ds . \]
Using that 
	\[ \partial_{\l=0}  \Big( \frac{\pi_{\l , s} (x)}{\pi (x)} \, r^\l_s  (x,y) \Big) 
	= (a_s (x) + g(s,x,y)) r(x,y)  \]
we end up with 
	\[ \partial_{\l=0} \bbE_{\pi_\l}  \Big [\sum_{s\in (0,t]}  \a(s, X^\l_{s-}, X^\l_s ) \Big ] 
	= \int_0^t ds \la \a_r (s,\cdot) , a_s \ra 
	+ \int_0^t \bbE_\pi \big[ (\alpha g )_r (s,X_s ) \big] ds  . \]
  As  $a_s= \int _0^\infty du  \,e^{u \cL^*}\psi_{s-u}$ (see Corollary \ref{superapogeo}) we have 
\be\label{oceano16X}
\begin{split}
&  \int_0^t ds \la \a_r (s,\cdot) , a_s \ra = \int_0^t ds  \int _0^\infty du  \la \a_r (s,\cdot)  \,e^{u \cL^*}\psi_{s-u}\ra\\
 & = \int_0^t ds  \int _0^\infty du \, \la  e^{u \cL } \a_r (s,\cdot) , \psi_{s-u}\ra
= \int_0^t ds \int _0^\infty du \, 
\bbE_\pi \Big[ \alpha_r (s, X_u ) \psi_{s-u}( X_0 ) \Big ] \,,
\end{split}
\en
thus giving the identity 
	\[ \begin{split}
	\partial_{\l=0} \bbE_{\pi_\l}  \Big [\sum_{s\in (0,t]}  \a(s, X^\l_{s-}, X^\l_s ) \Big ]  
	 = &  \int_0^t \bbE_\pi \big[ (\alpha g )_r (s,X_s ) \big] ds 
	\\  & +   \int_0^t ds \int _0^\infty du \, 
\bbE_\pi \Big[ \alpha_r (s, X_u ) \psi_{s-u}( X_0 ) \Big ] . 
\end{split} \]
\end{proof}
\section{\rosso{Proof of Theorems \ref{teo_CM} and \ref{teo_CM_esteso}}}
\subsection{\rosso{Proof of Theorem \ref{teo_CM}}}\label{sec_dim_teo_CM}
By \eqref{matteo105} \rosso{we have}
$
V_\l (t)= 
\sum_{e:|e|=1 }  \exp \{ \l \cos(\o t ) e \cdot v\}  \bbE_{\pi_\l}\bigl[ r (X^{\l}_t,X^\l_t+e) \bigr] e$.
Hence
\be
\begin{split}
\partial_{\l=0} V_\l (t) & = \sum_{e:|e|=1 }   \cos(\o t ) ( e \cdot v) \bbE_{\pi}\bigl[ r (X_t,X_t+e) \bigr] e\\
& +\sum_{e:|e|=1 } \partial_{\l=0} \bbE_{\pi_\l}\bigl[ r (X^{\l}_t,X^\l_t+e) \bigr] e=: A+B\,.
\end{split}
\en
By stationarity  $ \bbE_{\pi}\bigl[ r (X_t,X_t+e) \bigr]= \pi\bigl[ r(\cdot, \cdot+e) \bigr]$. This observation allows to  rewrite the j$^{th}$ coordinate of the vector $A$ as  $A_j=  \cos(\o t )   v_j  \pi[ c_j]= \Re\bigl( e^{i \o t} v_j \pi [c_j]\bigr)$.
On the other hand, by \eqref{telefono}  we have  
\[
B_j=  \partial_{\l=0} \bbE_{\pi_\l}\bigl[ \g_j (X^{\l}_t) \bigr] = \Re\Big( e^{i \o t} \la \g_j, (i \o -\cL^*)^{-1} (\Psi\cdot v) \ra  \Big)\,.
\]
Hence 
\be
\bigl( \partial_{\l=0} V_\l (t) \bigr)_j= \Re\Big( e^{i \o t} \Big(v_j \pi [c_j]+ \la \g_j, (i \o -\cL^*)^{-1} (\Psi\cdot v) \ra  \Big)\Big)=\Re\Big( e^{i\o t} \sum_{k=1}^d \s(\o)_{j,k} v_k \Big)
\en
where 
$ \s(\o)_{j,k}= \pi [c_j]\d_{j,k}+ \la \g_j, (i \o -\cL^*)^{-1} \Psi_k \ra $. This allows to get \eqref{stilton}, \eqref{jabba1} and \eqref{jabba2} (recall \eqref{linz}).

Let us conclude by showing that the matrix $\s(\o)$ in \eqref{jabba2} is symmetric for the reversible random walk. It is enough to show that  $\la \g_j, (i \o -\cL)^{-1} \g_k \ra= \la \g_k, (i \o -\cL)^{-1} \g_j \ra$ for all $j,k$. As $\cL=\cL^*$, we have $\la \g_j, (i \o -\cL)^{-1} \g_k \ra=\la (-i \o -\cL)^{-1} \g_j,  \g_k \ra$. As $\g_j,\g_k$ are real functions, we have 
\begin{equation*}
\begin{split}
& \la (-i \o -\cL)^{-1} \g_j,  \g_k \ra= \sum _x \pi(x) \overline{\left( (-i \o -\cL)^{-1} \g_j\right) (x)} \g_k(x)\\
& = \sum _x \pi(x) \left( (i \o -\cL)^{-1} \g_j\right)  (x) \overline{ \g_k(x)} =  \la \g_k, (i \o -\cL)^{-1} \g_j \ra\,. \end{split}
\end{equation*}
\qed
\subsection{Proof of Theorem \ref{teo_CM_esteso}}\label{sec_dim_teo_CM_esteso} 
By \eqref{matteo106}  we have
\[  V_\l (t) = \sum_{z\in \cZ}  \exp\left\{  \l \cos(\o t) \left(z \cdot v\right)\right\}
\bbE_{\pi_\l} 
\big[
   r (X^{\l}_t, X^{\l}_t+z ) 
  \big] z\,.
  \] Hence
\be
\begin{split}
\partial_{\l=0} V_\l (t) & = \sum_{z\in \cZ }   \cos(\o t ) ( z \cdot v) \bbE_{\pi}\bigl[ r (X_t,X_t+z) \bigr] z\\
& +\sum_{z\in \cZ } \partial_{\l=0} \bbE_{\pi_\l}\bigl[ r (X^{\l}_t,X^\l_t+z) \bigr] z=: A+B\,.
\end{split}
\en
The j$^{th}$ coordinate of the vector $A$ is given by 
\[ A_j=\sum_{k=1}^d\Re\Big( e^{i \o t} \big( \sum_{z\in\cZ} z_j z_k \pi [ r(\cdot, \cdot+z)] \big)
   \Big)v_k\,.\]
From this point onwards  the conclusion of the proof is then identical to that of Theorem \ref{teo_CM}.
\qed
%
%
%
  \bigskip 
 \bigskip 
 
  \appendix \section{Local martingales for Markov jump processes} \label{appendix}
Fixed $x_0\in \cX$ we consider here the unperturbed Markov process $X:=(X_s)_{s\geq 0}$ starting at $X_0=x_0$ assuming it does not explode and 
  apply the analysis in \cite[App.~A5]{D} to the process $Y:=(Y_s  )_{s\geq 0}$ defined as $Y_s:=(X_{s-}, X_s)$ for $s>0$ and $Y_0:=(x_0,x_0)$. The process $Y$ can be described via the  formalism in \cite[App.~A1]{D}.  To this aim we define  $T_1,T_2,\dots$ as the jump times of $Y$  and set $T_0:=0$, $S_k:= T_k-T_{k-1}$ for $k\geq 1$ and $Z_k := Y_{T_k}\in \cX\times \cX$ for $k\geq 1$. 
  Note that the jump times $T_1<T_2<\dots $ of the process $Y$ coincide with the jump times $\t_1<\t_2<\dots$ of the process $X$.
  Then the process $(x_s)_{s\geq 0} $ in \cite[page~257]{D} associated to the sequence $(S_k,Z_k)_{k\geq 1}$ corresponds to $Y$. We point out that the functions $\mu^k$ introduced in \cite[page~258]{D}  are the following:  $\mu^1$ is the law of $(S_1,Z_1)$ and, for $k\geq 1$,  $\mu^{k}( \o_1,\o_2, \dots, \o_{k-1};\cdot)$ is the law of $(S_k,T_k)$ conditional on the event that $(S_1,Z_1)=\o_1$, $(S_2,Z_2)=\o_2, \ldots , (S_{k-1},Z_{k-1})=\o_{k-1}$ (if the above event has positive probability, otherwise the definition of $\mu^{k}( \o_1,\o_2, \dots, \o_{k-1};\cdot)$ does not play any role).

  \medskip

  We now move to  \cite[App.~A5]{D} and explain how the key objects there read in our context. Below $A$ is a measurable subset of  $ \cX\times \cX$ and $u,s,t$ are  times in  $\bbR_+$.
  
 We want to compute  $\Phi^A_1 (s)$ introduced in \cite[App.~A5]{D}. Setting  
  $F^{A,1}(u):=\bbP_{x_0}( S_1>u, Z_1\in A)$, we have 
 $ \Phi^A_1 (s) :=- \int_{(0,s]} \frac{1}{F^{\cX\times \cX,1} (u-)} d F^{A,1}(u)$. Therefore 
 \[
  d\Phi^A_1 (s)=\bbP_{x_0}( S_1\in  (s,s+ds], Z_1\in A\,|\, S_1\geq s)=  \hat r (x_0) ds \int _\cX\frac{ r(x_0, dx_1)}{\hat r (x_0)} \mathds{1}_A(x_0,x_1)
  \]
and therefore 
$\Phi^A_1 (s)=s  \int _\cX r(x_0, dx_1)\mathds{1}_A(x_0,x_1)= s (\mathds{1}_A)_r (x_0) $.

   We now want to compute  $\Phi^A_2 (\omega_1,s)$ of \cite[App.~A5]{D} with $\omega_1 =(s_1, x_0, x_1)$.
  Setting 
   $F^{A,2}(  (s_1, x_0, x_1),  u): =\bbP_{x_0}( S_2>u, Z_2\in A | S_1=s_1, Z_1=(x_0,x_1) )$, we have 
    \[\Phi^A_2( (s_1,x_0,x_1),s):=  -\int_{(0,s]}
    \frac{1}{F^{\cX\times \cX,2} ( (s_1,x_0,x_1),u-) }d F^{A,2}(  (s_1, x_0, x_1),  u)\,.
    \]Therefore
     \begin{equation*}
     \begin{split}
     d\Phi^A_2( (s_1,x_0,y),s)& =
     \bbP_{x_0}( S_2\in (s,s+ds], Z_2\in A\,|\, S_1=s_1, \;Z_1=(x_0,x_1),\; S_2\geq s)\\
   &   =
      ds \int _\cX r(x_1, dx_2)\mathds{1}_A(x_1,x_2)
  \end{split}
  \end{equation*}
and therefore  $\Phi^A_2( (s_1,x_0,x_1),s)=s  (\mathds{1}_A)_r (x_1 )$.

All other functions $\Phi^A_k$ can be computed similarly.
  It then follows that, for  $s\in  (T_{k-1},T_k]$,  the function $\tilde p(s,A)$ defined in \cite[page~276]{D} equals
 \[   \tilde p (s,A) = S_1 (\mathds{1}_A)_r(X_0)+ S_2 (\mathds{1}_A)_r (X_{T_1})+\cdots + S_{k-1} (\mathds{1}_A)_r (X_{T_{k-2}})+ (s-T_{k-1}) (\mathds{1}_A)_r (X_{T_{k-1}})\,.
  \]   
On the other hand $p(s,A)$ and $q(s,A)$  in \cite[App.~A5]{D} are given by $p(s,A):= \sum _{ u\in (0,s]} \mathds{1}_A (X_{u-},X_u)$ and $q(s,A)=p(s,A)-\tilde p(s,A)$. 
    Hence, given a measurable function  $\a:[0, \infty ) \times \cX\times \cX \to \bbR$, we have
  \begin{equation} \label{Ma}
  M_s^\a:= \int _{(0,s]\times\cX\times \cX} \a(u, x,y) q(du, dx, dy)
  =\sum_{u\in (0, s ] } \a (u, X_{u-}, X_u) -
  \int _0 ^s  \a_r (u,X_u) du \,.
  \end{equation}
 Recall from 
\cite[pages~270,~276]{D} that the (deterministic) measurable function $\a:[0,+\infty)\times \cX\times \cX \to \bbR$ belongs to $L_1^{\textrm{loc}}(p) $ if there exists a non-decreasing sequence of stopping times $(\xi_n)_{n\geq 1} $ such that $\xi_n \to \infty  $ almost surely as $n \to\infty$ and $ \a \mathbf{1}_{[0,\xi_n) }$ is in $L^1 (p) $ for all $n\geq 1$, i.e.
	\[\bbE_{x_0} \bigg( \sum_{u \in (0,s]} |\a ( u ,X_{u-}  ,X_u ) |  \mathbf{1}_{[0, \xi_n ) } (u) \bigg) 
	= \bbE_{x_0} \bigg( \sum_{u \in (0,s] \cap (0,\xi_n )} |\a ( u ,X_{u-}  ,X_u ) |\bigg) < \infty 
	 \]
for all $n\geq 1 $. 
The space $L_1^{\textrm{loc}}(\tilde{p})$ can be defined analogously by integrating with respect to $\tilde{p}$ rather than $p$, and it is proved in \cite[Proposition (A4.5) and page~276]{D} that $ \a \in L_1^{\textrm{loc}}(p) $ if and only if $ \a \in L_1^{\textrm{loc}}(\tilde{p}) $. Moreover, if $\a \in L_1^{\textrm{loc}}(p) $ then the process $M^\a $ defined in \eqref{Ma}  above  is a local martingale  \cite[Proposition (A5.3)]{D}.

We conclude this appendix by showing that if $\a : [0, \infty ) \times \cX\times \cX \to \bbR$ is a measurable function satisfying \eqref{welldef}, then the associated process $M^\a$ defined in \eqref{Ma} is a local martingale. To this end it will suffice to show that $\a \in L_1^{\textrm{loc}}(p) $.  Indeed, define the non-decreasing sequence of stopping times $(\xi_n)_{n\geq 1}$ by setting $\xi_n := \inf\{ s \in [0, \infty )  : \sum_{u \in (0,s] } |\a ( u , X_{u-} , X_u )| \geq n \} $.  Then the requirement $\a \mathbf{1}_{[0,\xi_n) } \in L_1 (p) $ is trivially satisfied, and $\xi_n \to\infty  $ as $n \to\infty$ by the non-explosion assumption of the unperturbed process  in $[0, \infty )$.

When working with functions $\a : [0, t] \times \cX\times \cX \to \bbR$ as in the previous sections,    one can apply the above results by setting $\a(s,\cdot, \cdot)=0$ for $s >t$. As a consequence,  the process $(M_s)_{s\in [0,t]}$ defined by \eqref{E:M} is a local martingale.

\section{\rosso{Comparison with linear response when starting with the invariant distribution of the perturbed process}}\label{sec:pollofritto}

 Trivially, our results apply also to  a time-independent  
 perturbation function $g$. In this case  we  write the perturbed rates simply as $r^\l(x,dy)=r(x,dy) e^{\l g(x,y)}$. Due to  time-independence, one can ask whether the perturbed Markov jump  process admits an invariant distribution and if this is unique at cost of restricting to the class $\cC$ of distributions which are absolutely continuous w.r.t.\ the invariant distribution $\pi$ of the unperturbed Markov jump process. In the case where there is a unique invariant distribution $\pi_\l$ (in $\cC$), it is natural then to investigate the linear response of the perturbed system with  initial distribution given by $\pi_\l$ (analogously to what we have done in Theorem \ref{maldive} for the OSS when $g$ is time-periodic).
 
Concerning the linear response, using the same notation of Proposition \ref{derivoX}, at a  formal level we would  get 
 \be\label{flautino}
 \partial_{\l=0} \bbE_{\pi_\l} \bigl[ F \bigl(  X^\l _{ [0,t] } \bigr) \bigr] =\partial_{\l=0} \bbE_{\pi_\l} \bigl[ F \bigl(  X_{ [0,t] } \bigr) \bigr]    +\partial_{\l=0} \bbE_{\pi} \bigl[ F \bigl(  X^\l _{ [0,t] } \bigr) \bigr]  \,.
\en
 Our results in Section \ref{sec:stationary} give information on the term 
 $\partial_{\l=0} \bbE_{\pi} \bigl[ F \bigl(  X^\l _{ [0,t] } \bigr) \bigr]$ in the r.h.s.. Proving existence and uniqueness of $\pi_\l$ and analyzing the  term $\partial_{\l=0} \bbE_{\pi_\l} \bigl[ F \bigl(  X_{ [0,t] } \bigr) \bigr]$ in the r.h.s.\ is usually hard (cf.\ e.g.\ \cite{FGS2,GGN,GMP,HM,KO,LR,MP}).  The analysis simplifies when one can use perturbation theory, e.g.\ when    the unperturbed process is an irreducible Markov chain with finite state space  $\cX$.    This case is  indeed covered by Section \ref{sec_OSS} as a degenerate case, since any time-independent function $g$ is also $T$--periodic (for any $T>0$) and in this case the OSS coincide with the stationary state.  In particular, Theorems \ref{alpha_omega} and \ref{maldive}  give the linear response of the perturbed system with  initial distribution given by $\pi_\l$. 
 
  We now sketch a direct analysis of term  $\partial_{\l=0} \bbE_{\pi_\l} \bigl[ F \bigl(  X_{ [0,t] } \bigr) \bigr]$ for this particular case  (i.e.~irreducible Markov chain with finite state space), without passing through time-periodic systems. As well as giving a more natural derivation, this will  also  explain why 
the decomposition \eqref{flautino} leads indeed to the same formulas appearing in 
  Theorems \ref{alpha_omega} and \ref{maldive}.

Recall the definition of the operator $\cL^*$ and its rates $r^*(x,y)$ given in Section \ref{sec_OSS}.
  Since $r^\l(x,y)>0$ if and only if $r(x,y)>0$, the perturbed Markov chain remains irreducible and therefore it has a unique invariant distribution $\p_\l$ (since $\p(x)>0$ for all $x \in \cX$, trivially $\p_\l\ll \p$).
 The time-invariance of $\pi_\l$ corresponds to the system 
  \be\label{sole25}
  \sum_y \left( \pi_\l (y) r^\l(y,x)  -\pi_\l(x) r^\l(x,y)\right)=0 \qquad \forall x\in \cX\,.
 \en
If we see $\pi_\l$ as a row vector and the infinitesimal generator $\cL^\l$ of the perturbed Markov chain  as a matrix, the identity \eqref{sole25} corresponds to $\pi_\l \cL^\l =0$. By matrix perturbation theory \cite{Ka}, we get that $\pi_\l$ is differentiable at $\l=0$. Therefore, setting $\dot \p(x) := \partial _{\l=0} \pi_\l(x)$, from \eqref{sole25} we get 
\be\label{luna25}
\sum _y \big(\dot\p(y) r(y,x)-\dot\pi(x) r(x,y)\big)= \sum_y \big( \pi(x) r(x,y) g(x,y)-\pi(y) r(y,x) g(y,x) \big) \,,\; \forall x\in \cX.
\en
Dividing by $\pi(x)$, using the intertwining relation  $\pi(a) r(a,b)=\pi(b) r^*(b,a)$ and that $\sum_y r(x,y)= \sum _y r^*(x,y)$ (which can be derived from $\sum_y\pi(x)r(x,y)=\sum_y \pi(y) r(y,x)$ and the above intertwining relation),
we get that \eqref{luna25} is equivalent to 
 \be\label{tresoli}\cL^* \frac{\dot\pi}{\pi}=- \psi\,, \qquad   \frac{\dot\pi}{\pi}(x):= \frac{\dot\pi(x)}{\pi(x)}\,, \qquad \psi(x)= \sum_y\big( r^*(x,y) g(y,x)-r(x,y)g(x,y)\big)\,. 
 \en
Note that $\psi(x)=\psi_t(x)$ for all $t\geq 0$, with $\psi_t$ defined as in \eqref{def:psi} and \eqref{def:psi:bis}.
 Since $\cL^*$ is an isomorphism when restricted to $L^2_0(\pi)$  (see Section~\ref{sec_OSS}) and $\frac{\dot\pi}{\pi},\psi \in L^2_0(\pi)$ (as can be easily checked), with the notation introduced in \eqref{raggio} and due to Lemma \ref{cop25}   we get $\frac{\dot\pi}{\pi}= -(\cL^*)^{-1} \psi=\int_0^\infty e^{s \cL_*} \psi ds $.  Note that this last identity coincides with that of Lemma \ref{apogeo}.  In the same context of Proposition \ref{derivoX} we then  get that 
 \be\label{auto}
 \begin{split}
 \partial_{\l=0} \bbE_{\pi_\l} \bigl[ F \bigl(  X_{ [0,t] } \bigr) \bigr]=
 \bbE_{\pi}\bigl[ -(\cL^{-1})^* \psi (X_0) F \bigl(  X_{ [0,t] } \bigr) \bigr]\\
 = 
 \int _0^\infty  \bbE_{\pi}\bigl[ e^{s \cL^*}\psi (X_0) F \bigl(  X_{ [0,t] } \bigr) \bigr]ds\,.
 \end{split}
 \en
 If e.g.~one takes $F \bigl(  \xi_{ [0,t] } \bigr) =v(\xi_t)$, the rightmost term in  \eqref{auto} equals 
 \[ \int _0^\infty   \bbE_{\pi}\bigl[ \psi (X_0) v(  X_{t+s}) \bigr]ds =   \int_t^\infty \bbE_\pi[ \psi(X_0) v(X_s)]ds\,. \]  Hence, due to decomposition \eqref{flautino} and the first formula in Theorem \ref{cor:JS}, we obtain \eqref{kiriku1} in Theorem \ref{alpha_omega}. The same analysis can be carried out also for \eqref{kiriku2} and \eqref{kiriku3}. 
 
We stress that the  above derivation is based on matrix perturbation theory, while for more general stochastic systems more sophisticated approaches are necessary (cf.~\cite{HM,KO,MP}).

\end{document}